\documentclass[11pt, a4paper]{amsart}
\usepackage{amscd,amsmath,amssymb,amsfonts,verbatim}
\usepackage[cmtip, all]{xy}
\usepackage{comment}

\usepackage[ left=3cm, right=3cm, top = 2.8cm, bottom = 2.8cm ]{geometry}
\usepackage[colorlinks,linkcolor=blue,citecolor=blue,urlcolor=red]{hyperref}

\usepackage{MnSymbol}

\usepackage{tikz-cd}

\newtheorem{thm}{Theorem}[section]
\newtheorem{prop}[thm]{Proposition}
\newtheorem{lem}[thm]{Lemma}
\newtheorem{cor}[thm]{Corollary}

\theoremstyle{definition}

\newtheorem{defn}[thm]{Definition}
\theoremstyle{remark}
\newtheorem{remk}[thm]{Remark}
\newtheorem{remks}[thm]{Remarks}

\newtheorem{exm}[thm]{Example}
\newtheorem{exms}[thm]{Examples}
\newtheorem{notat}[thm]{Notation}
\numberwithin{equation}{section}

{\hfill$\square$\end{defn}}
{\hfill$\square$\end{remk}}
{\hfill$\square$\end{remks}}
{\hfill$\square$\end{exm}}
{\hfill$\square$\end{exms}}
{\hfill$\square$\end{notat}}

\newcommand{\thmref}{Theorem~\ref}
\newcommand{\propref}{Proposition~\ref}
\newcommand{\corref}{Corollary~\ref}

\newcommand{\lemref}{Lemma~\ref}

\newcommand{\sA}{{\mathcal A}}

\newcommand{\sC}{{\mathcal C}}

\newcommand{\sE}{{\mathcal E}}
\newcommand{\sF}{{\mathcal F}}

\newcommand{\sH}{{\mathcal H}}
\newcommand{\sI}{{\mathcal I}}

\newcommand{\sK}{{\mathcal K}}

\newcommand{\sM}{{\mathcal M}}
\newcommand{\sN}{{\mathcal N}}
\newcommand{\sO}{{\mathcal O}}

\newcommand{\sR}{{\mathcal R}}

\newcommand{\sZ}{{\mathcal Z}}
\newcommand{\A}{{\mathbb A}}

\newcommand{\G}{{\mathbb G}}
\renewcommand{\H}{{\mathbb H}}

\renewcommand{\P}{{\mathbb P}}
\newcommand{\Q}{{\mathbb Q}}

\newcommand{\W}{{\mathbb W}}

\newcommand{\Z}{{\mathbb Z}}

\newcommand{\fp}{{\mathfrak p}}

\newcommand{\Lci}{{\rm lci}}

\newcommand{\Ker}{{\rm Ker}}

\newcommand{\Alb}{{\rm Alb}}
\newcommand{\CH}{{\rm CH}}

\newcommand{\surj}{\twoheadrightarrow}
\newcommand{\inj}{\hookrightarrow}
\newcommand{\red}{{\rm red}}
\newcommand{\codim}{{\rm codim}}

\newcommand{\Pic}{{\rm Pic}}

\newcommand{\Hom}{{\rm Hom}}

\newcommand{\Spec}{{\rm Spec \,}}
\newcommand{\sing}{{\rm sing}}
\newcommand{\Char}{{\rm char}}

\newcommand{\ab}{\rm ab}

\newcommand{\Gal}{{\rm Gal}}
\newcommand{\divf}{{\rm div}}

\newcommand{\sHom}{{\mathcal{H}{om}}}

\newcommand{\sExt}{{\mathcal{E}{xt}}}

\newcommand{\Sch}{{\operatorname{\mathbf{Sch}}}}

\newcommand{\Sm}{{\mathbf{Sm}}}

\newcommand{\cyc}{{\operatorname{\rm cyc}}}

\newcommand{\et}{{\text{\'et}}}

\newcommand{\ds}{{/\kern-3pt/}}

\renewcommand{\log}{{\operatorname{log}}}

\newcommand{\un}{\underline}
\newcommand{\ov}{\overline}

\renewcommand{\dim}{\text{\rm dim}}

\newcommand{\tuborg}{\left\{\begin{array}{ll}}
\newcommand{\sluttuborg}{\end{array}\right.}

\newcommand{\zar}{{\rm zar}}
\newcommand{\nis}{{\rm nis}}

\newcommand{\edim}{{\rm edim}}

\newcommand{\reg}{{\rm reg}}

\newcommand{\Lef}{{\rm Lef}}

\newcommand{\Spf}{{\rm Spf}}
\newcommand{\tor}{{\rm tor}}
\newcommand{\ns}{{\rm NS}}
\newcommand{\dlog}{{\rm dlog}}

\newcommand{\wt}{\widetilde}
\newcommand{\wh}{\widehat}

\newcommand{\coker}{{\rm Coker}}

\newcommand{\free}{{\rm free}}

\newcommand{\etl}{{\acute{e}t}}

\newcounter{elno}

\newcounter{elno-abc}   
\newenvironment{listabc}{
                         \begin{list}{\alph{elno-abc})
                                     }{\usecounter{elno-abc}}
                      }{
                         \end{list}}
\newcounter{elno-abc-prime}

\begin{document}
\title{Zero-cycles on normal projective varieties}
\author{Mainak Ghosh, Amalendu Krishna}
\address{School of Mathematics, Tata Institute of Fundamental Research,  
1 Homi Bhabha Road, Colaba, Mumbai, 400005, India.}
\email{mainak@math.tifr.res.in}
\address{Department of Mathematics, Indian Institute of Science,  
Bangalore, 560012, India.}
\email{amalenduk@iisc.ac.in}



\maketitle

\begin{quote}\emph{Abstract.} 
We prove an extension of the Kato-Saito unramified class field theory for
smooth projective schemes over a finite field to a class of normal projective
schemes.
As an application, we obtain Bloch's formula for 
the Chow groups of 0-cycles on such schemes.
We identify the Chow group of 0-cycles on a normal projective
scheme over an algebraically closed field to the Suslin homology of its
regular locus. 
Our final result is a Roitman torsion theorem for smooth quasi-projective
schemes over algebraically closed fields. This completes the missing
$p$-part in the torsion theorem of Spie{\ss} and Szamuely.
\end{quote}
\setcounter{tocdepth}{1}
\tableofcontents

\section{Introduction}\label{sec:Intro}
\subsection{Motivation}\label{sec:Motiv*}
It is well known that the Chow group of 0-cycles on a smooth projective scheme over an
appropriate field describes many other invariants of the scheme such as
Suslin homology,
cohomology of $K$-theory sheaves and 
abelianized {\'e}tale fundamental groups. However, this is not the case when the underlying
scheme is not projective. The latter case is a very challenging problem in the theory of algebraic
cycles.
The principal motivation of this paper is to explore if the Levine-Weibel Chow group \cite{Levine-Weibel} of
normal projective schemes could be used to solve this problem for those smooth quasi-projective
schemes which are the regular loci of normal projective schemes.
The results that we obtain in this paper suggest that this strategy is indeed 
a promising one. Below we describe our main results in some detail.

\subsection{Levine-Weibel Chow group and Class field theory}
\label{sec:CFT}
The aim of the class field theory in the geometric case is to describe 
the abelian {\'e}tale coverings (which are extrinsic to the scheme)
of a scheme over a finite field
in terms of some arithmetic or geometric invariants (such as the Chow groups of 
0-cycles) which are intrinsic to the scheme.
Let $k$ be a finite field and $X$ an integral smooth projective scheme over $k$.
Let $\CH^F_0(X)$ denote the classical (see \cite{Fulton}) Chow group of 0-cycles 
and let $\CH^F_0(X)^0$ denote the kernel of the degree map 
${\rm deg}_X \colon \CH^F_0(X) \to \Z$. Let $\pi^{\rm ab}_1(X)$ denote the
abelianized {\'e}tale fundamental group of $X$ (e.g., see \cite[\S~5.8]{Szamuely})
and let $\pi^{\rm ab}_1(X)^0$ denote the
kernel of the canonical map $\pi^{\rm ab}_1(X) \to \Gal({\ov{k}}/k)$
induced by the structure map of $X$.
The following is the main theorem of the geometric class field 
theory for smooth projective schemes. The case of curves was
earlier proven by Lang \cite{Lang-CFT}, which was based on Artin's reciprocity theorem
for local and global fields \cite{Artin-Tate}.
 
\begin{thm}$($\cite[Theorem~1]{Kato-Saito-1}$)$\label{thm:CFT-KS}
Let $X$ be an integral smooth projective scheme over a finite field.
Then the map $\phi^0_X \colon \CH^F_0(X)^0 \to \pi^{\rm ab}_1(X)^0$, induced by 
sending a closed point to its associated Frobenius element, is an isomorphism 
of finite groups.
\end{thm} 

If $U$ is a smooth quasi-projective scheme over a finite field
$k$ which is not projective, then one does not know in general how to 
describe $\pi^{\rm ab}_1(U)$ in terms of 0-cycles.
It was shown by Schmidt and Spie{\ss} \cite{Schmidt-Spiess} and
Schmidt \cite{Schmidt} 
that the tame quotient of $\pi^{\rm ab}_1(U)$ is described by the
Suslin homology of $U$. But we do not yet know if the abelian covers of $U$ 
with wild ramifications could be described in terms of the 
Chow group of 0-cycles on a compactification of $U$.
Our main result in this direction provides a partial answer to this
problem.

Let $\CH^{LW}_0(X)$ denote the Levine-Weibel Chow group of 0-cycles of
a scheme $X$ \cite{Levine-Weibel} (see \S~\ref{sec:LWC-0} for a reminder
of its definition).

\begin{thm}\label{thm:Main-1}
Let $X$ be an integral projective scheme over a finite field which is
regular in codimension one. Then the Frobenius substitution associated to
the regular closed points gives rise to a reciprocity homomorphism
\[
\phi_X \colon \CH^{LW}_0(X) \to \pi^{\rm ab}_1(X_\reg)
\]
which restricts to a surjective homomorphism
$\phi^0_X \colon \CH^{LW}_0(X)^0 \surj \pi^{\rm ab}_1(X_\reg)^0$.
The map $\phi^0_X$ is an 
isomorphism of finite groups in any of the following cases.
\begin{enumerate}
\item
$X$ has only isolated singularities.
\item
$X$ is regular in codimension three and its local rings satisfy Serre's
$S_4$ condition.
\end{enumerate}
\end{thm}

\vskip .3cm

It readily follows that in cases (1) and (2), the map $\phi_{X}$ is injective 
with uniquely divisible cokernel ${\wh{\Z}}/{\Z}$
(see ~\eqref{eqn:Rec-map-2}). Note also that the finiteness of
the source and target of $\phi^0_X$ is part of our assertion and was
not known before.

\vskip .3cm

Without the assumption (1) or (2) in 
\thmref{thm:Main-1}, we prove the
following.

\begin{thm}\label{thm:Main-1-0}
Let $X$ be an integral projective scheme over a finite field which is
regular in codimension one. Then reciprocity homomorphism
of \thmref{thm:Main-1} induces an isomorphism of finite groups
\[
\phi_X \colon {\CH^{LW}_0(X)}/m \to {\pi^{\rm ab}_1(X_\reg)}/m
\]
for every integer $m \in k^{\times}$.
\end{thm}

\subsection{Bloch's formula for Levine-Weibel Chow group}
\label{sec:BF}
In the theory of algebraic cycles, Bloch's formula
describes the Chow group of algebraic cycles of codimension $d$ 
on a smooth scheme (of any dimension) over a field as the $d$-th Zariski or
Nisnevich cohomology of an appropriate Milnor or Quillen $K$-theory sheaf.
A statement of this kind plays a central role in the study of algebraic cycles
on smooth schemes. Bloch's formula for smooth schemes 
in $d =1$ case is classical, the $d =2$ case is due to Bloch \cite{Bloch74} and
the general case is due to Quillen \cite{Quillen}. This
formula for the Chow group of 0-cycles on smooth schemes in terms of 
the Milnor $K$-theory is due to Kato \cite{Kato86}

Bloch's formula for the Levine-Weibel Chow group is well known
for singular curves (see \cite[Proposition~1.4]{Levine-Weibel}). 
However, it is a very challenging problem in higher dimensions.
This formula for singular surfaces over algebraically closed fields
is due to Levine \cite{Levine85}. For projective surfaces over infinite fields,
this formula was recently proven by Binda, Krishna and Saito \cite[Theorem~8.1]{BKS}.
Bloch's formula for the Levine-Weibel Chow group of singular projective schemes
over non-algebraically closed fields is yet unknown in any other case.

Suppose that $X$ is a quasi-projective scheme of pure dimension $d$ 
over a perfect{\footnote{Perfectness is not required by
    \cite[Lemma~3.7]{Gupta-Krishna-CFT}.}} field $k$ and
$x \in X_\reg$ is a regular closed point.
One then knows by \cite[Theorem~2]{Kato86} that there is a canonical isomorphism
$\Z \xrightarrow{\cong} K^M_0(k(x)) \xrightarrow{\cong} 
H^d_x(X, \sK^M_{d, X})$, where the latter is the Nisnevich cohomology 
with support, $K^M_n(R)$ is the Milnor $K$-theory on a ring $R$ and $\sK^M_{i,X}$ is the Nisnevich sheaf of Milnor $K$-theory on $X$
(as defined, for instance, in \cite[\S~0]{Kato86}){\footnote{We could
    use the improved Milnor $K$-theory of Kerz instead, but it will make no
    difference in the top cohomology.}}.
Hence, using the `forget support' map for $x$ and extending
it linearly to the free abelian group on all regular closed points of
$X$, we get the cycle class homomorphism
\begin{equation}\label{eqn:Cycle-map}
\cyc_{X} \colon \sZ_0(X_\reg) \to H^d_\nis(X, \sK^M_{d, X}).
\end{equation}

As an application of \thmref{thm:Main-1} and the class field theory of Kato 
and Saito \cite{Kato-Saito-2}, we prove the following.

\begin{thm}\label{thm:Main-2}
Let $X$ be an integral projective scheme of dimension $d$ 
over a finite field which satisfies one of the following.
\begin{enumerate}
\item
$X$ has only isolated singularities.
\item
$X$ is regular in codimension three and its local rings satisfy Serre's
$S_4$ condition.
\end{enumerate}
Then the cycle class map induces an isomorphism
\[
\cyc_{X} \colon \CH^{LW}_0(X) \xrightarrow{\cong} H^d_\nis(X, \sK^M_{d, X}).
\]
\end{thm}

\vskip .3cm

In case of isolated singularities, one can also include the Zariski 
topology and Quillen $K$-theory sheaf $\sK_{d,X}$ in 
\thmref{thm:Main-2} (see \S~\ref{sec:Isolated}).
We remark that the existence of $\cyc_X$ on $\CH^{LW}_0(X)$ is part of
our assertion, and was not previously known.

\subsection{Levine-Weibel Chow group and 0-cycles with modulus}
\label{sec:BSC}
Let $X$ be an integral projective scheme of dimension $d \ge 2$
over a field which is regular in codimension 
one. Let us assume that
a resolution of singularities $f\colon \wt{X} \to X$ exists in the sense of 
Hironaka. Let $E \inj \wt{X}$
denote the reduced exceptional divisor. Let $\CH_0(\wt{X}|mE)$ denote the
Chow group of 0-cycles with modulus (see \S~\ref{sec:Res-surface}). 
It is not hard to see that the identity map of  
$\sZ_0(X_{\rm reg})$ induces a surjection
$\CH^{LW}_0(X) \surj \CH_0(\wt{X}|mE)$ for all integers $m \ge 1$.
The following application of \thmref{thm:Main-1} is an extension of the
Bloch-Srinivas conjecture (which was proven for normal surfaces in
\cite{Krishna-Srinivas}) to higher dimensions over finite fields.

\begin{thm}\label{thm:Main-3}
Let $X$ be an integral projective scheme of dimension $d \ge 2$ 
over a finite field which satisfies one of the following.
\begin{enumerate}
\item
$X$ has only isolated singularities.
\item
$X$ is regular in codimension three and its local rings satisfy Serre's
$S_4$ condition.
\end{enumerate}
Let $f \colon \wt{X} \to X$ be a resolution of
singularities with the reduced exceptional divisor $E$.
Then the pull-back map $f^* \colon \sZ_0(X_{\rm reg}) \to 
\sZ_0(\wt{X} \setminus E)$ induces an isomorphism
\[
f^* \colon \CH^{LW}_0(X) \xrightarrow{\cong} \CH_0(\wt{X}|mE)
\]
for all $m \gg 0$.
\end{thm}

If $X$ is defined over an algebraically closed field, \thmref{thm:Main-3}
was proven by Gupta and Krishna \cite[Theorem~1.8]{Gupta-Krishna}
(see also \cite{Krishna-Crelle} and \cite{Krishna-JAG}
for earlier results).

\subsection{Levine-Weibel Chow group and Suslin homology}
\label{sec:LWS**}
The Suslin homology was introduced by Suslin and Voevodsky \cite{Suslin-Voev}
with the objective of constructing an analogue of singular homology of
topological spaces for algebraic schemes.
It is now a part of the motivic cohomology with compact support of 
schemes in the sense of $\A^1$-homotopy theory. 
Because of this, Suslin homology is now a well studied theory and there are many
known results which can be used for its computation. On the contrary, the
Levine-Weibel Chow group is only conjecturally a part of a motivic
cohomology theory of singular schemes and is much less accessible.

Our next main result however provides an identification between 
the two groups and provides a strong evidence that over algebraically closed
fields of positive characteristics, the conjectural motivic cohomology of 
a normal projective scheme should coincide with  the already known motivic
cohomology with compact support of its regular locus.

\begin{thm}\label{thm:Main-4}
  Let $X$ be an integral projective scheme of dimension $d \ge 1$ over a
  field $k$ which is regular in codimension one. Then
there is a canonical surjective homomorphism
\[
  \theta_X \colon \CH^{LW}_0(X) \surj H^S_0(X_\reg).
\]

Furthermore, we have the following if $k$ is algebraically closed.

\begin{enumerate}
\item
  $\theta_X$ is an isomorphism if ${\rm char}(k) > 0$.
\item
  ${\theta_X}/n \colon {\CH^{LW}_0(X)}/n \to {H^S_0(X_\reg)}/n$
  is an isomorphism for all integers $n \neq 0$ if ${\rm char}(k) = 0$.
\end{enumerate}
\end{thm}

\vskip .3cm 
We shall show (see \S~\ref{sec:LWCH-fin})
that the condition that $k$ is algebraically closed in (1)
is essential.
When ${\Char}(k) = 0$, we expect that $\Ker(\theta_X)$ is
a (large) divisible group.

\subsection{Roitman torsion theorem for Suslin homology}
\label{sec:RTT*}
Let $k$ be an algebraically closed field.
Let $U$ be a smooth quasi-projective scheme over $k$ which admits an
open embedding $U \inj {X}$, where ${X}$ is smooth and projective
over $k$. Then Spie{\ss} and Szamuely \cite{SS} showed that
the Albanese map into the generalized Albanese variety ({\`a} la  Serre) of $U$
induces a homomorphism $\vartheta_U \colon H^S_0(U) \to \Alb_S(U)(k)$
which is an isomorphism on the prime-to-$p$ torsion subgroups,
where $p$ is the exponential characteristic of $k$. 
This was a crucial 
breakthrough in eliminating the projectivity hypothesis from the
famous Roitman torsion theorem for the Chow group of 0-cycles \cite{Roitman}.

Geisser \cite{Geisser} subsequently showed that the prime-to-$p$ condition in
the torsion theorem of Spie{\ss} and Szamuely could be eliminated if one
assumed resolution of singularities. 
The final result of this paper eliminates the prime-to-$p$ condition from
the torsion theorem of Spie{\ss} and Szamuely without assuming resolution of
singularities.

\begin{thm}\label{thm:Main-5}
Let $U$ be a smooth quasi-projective scheme 
over an algebraically closed field $k$.
Suppose that there exists an open immersion $U \subset X$ such that
$X$ is smooth and projective over $k$. Then the Albanese homomorphism 
\[
\vartheta_U \colon H^S_0(U)_{\tor} \to \Alb_S(U)(k)_\tor
\]
is an isomorphism. 
\end{thm}

\vskip .3cm

\subsection{Overview of proofs}\label{sec:Outline}
The proofs of our main results broadly have two main steps, namely, the
construction of the underlying maps and, the proof that these maps are
isomorphisms. The first part is achieved by means of a moving lemma for
the Levine-Weibel Chow group. The heart of the problem is the
more challenging second part.

Following an induction technique, we first prove our results for 
surfaces. To take care of higher
dimensions, we establish new Lefschetz hypersurface{\footnote{We call 
    hypersurface section theorems because we use hypersurfaces
    of large degrees instead of only hyperplanes used in some of the
    classical Lefschetz theorems.}} section theorems for
several invariants of smooth quasi-projective (but not necessarily projective)
and singular projective schemes. We may emphasize that 
these Lefschetz theorems are of
independent interest and we expect them to have several applications elsewhere.
These Lefschetz theorems allow us to reduce the proofs of the main
results to the case of surfaces. 

We prove the moving lemma (see \lemref{lem:Moving-nor})
and its consequences in \S~\ref{sec:LWC}.
Using this, we construct the reciprocity map (see \corref{cor:Rec-map-0-main})
and prove the reciprocity isomorphism for surfaces (see \thmref{thm:CFT-surface})
in \S~\ref{sec:Surface}.
We prove the Lefschetz theorems for the {\'e}tale cohomology
(see \propref{prop:Lef-et}) and the abelianized {\'e}tale fundamental group
(see \thmref{thm:Lef-EFG}) in \S~\ref{sec:LFG}.
They allow us to prove Theorem~\ref{thm:Main-1} in \S~\ref{sec:CFT**}.
We then combine this result with the class field theory of \cite{Kato-Saito-2} to prove
Theorems~\ref{thm:Main-2} and ~\ref{thm:Main-3} in \S~\ref{sec:CFT**}.  
 
The heart of the proofs of Theorems~\ref{thm:Main-4} and ~\ref{thm:Main-5}
are the two results: 
(1) a Lefschetz hypersurface section theorem for the generalized Albanese
variety of smooth quasi-projective schemes and,
(2) an identification of Suslin homology with the Chow group of 0-cycles
of a certain modulus pair. This first result is shown in
\S~\ref{sec:Alb*} and the second in \S~\ref{sec:LWS}. The latter section
also contains the proof of Theorem~\ref{thm:Main-4}. 
We combine the Lefschetz theorem with a result of
Geisser \cite{Geisser} to prove Theorem~\ref{thm:Main-5} in \S~\ref{sec:RTT}
following a delicate blow-up trick.

\subsection{Notations}\label{sec:Notn}
We shall, in general, work with an arbitrary base field $k$ even if our main 
results are over either finite or algebraically closed fields.
We let $\Sch_k$ denote the category of quasi-projective $k$-schemes and 
$\Sm_k$ the category of smooth quasi-projective $k$-schemes.
A product $X \times_k Y$ in $\Sch_k$ will be simply written as $X \times Y$.
For a reduced scheme $X$, we let $X_n$ denote the normalization of $X$.

For any excellent scheme $X$, we let $X^o$ denote the regular locus of $X$.
One knows that $X^o$ is an open subscheme of $X$ which is dense if
$X$ is generically reduced. We shall let $X_\sing$ denote the singular locus of
$X$. If $X$ is reduced, we shall consider $X_\sing$ as a closed subscheme of
$X$ with the reduced induced structure. 
For any Noetherian scheme $X$, we shall denote its {\'e}tale fundamental
group with a base point $x \in X$ by $\pi_1(X,x)$. We shall let
$\pi^{\ab}_1(X)$ denote the abelianization of $\pi_1(X,x)$. One knows that
$\pi^{\ab}_1(X)$ does not depend on the choice of the base point $x \in X$.
We shall consider $\pi^{\ab}_1(X)$ a topological abelian group with its
profinite topology.

For $X \in \Sch_k$ equidimensional,
we let $X^{(i)}$ denote the set of codimension $i$ points 
and $X_{(i)}$ the set of dimension $i$ points on $X$
for $i \ge 0$. We shall let $\sZ_i(X)$ denote the
free abelian group on $X_{(i)}$ and $\CH^F_i(X)$ the Chow group
of cycles of dimension $i$ as defined in \cite[Chapter~1]{Fulton}.
For an abelian group $A$, we shall denote 
the torsion and the $p$-primary torsion
subgroups of $A$ by $A_\tor$ and $A\{p\}$, respectively, for any prime $p$. 
For a commutative ring $\Lambda$, we shall
write $A_{\Lambda}$ for $A \otimes_{\Z} \Lambda$.

\section{Zero-cycles on $R_1$-schemes}
\label{sec:LWC}
In this section, we shall recall the definition of the Levine-Weibel Chow group
of an $R_1$-scheme and prove some preliminary results about this group.
Recall from \cite[p.~183]{Matsumura} that a Noetherian scheme $X$ is called
an $R_a$-scheme if it is regular in codimension $a$, where $a \ge 0$.
One says that $X$ is an $S_b$-scheme if for all points $x \in X$, one has
${\rm depth}(\sO_{X,x}) \ge {\rm min}\{b, \dim(\sO_{X,x})\}$. 
We shall say that a Noetherian scheme $X$ is an $(R_a+S_{b})$-scheme
for $a \ge 0$ and $b \ge 0$ if it is an $R_a$ as well as an $S_b$-scheme.
A Noetherian commutative ring will be called an $R_a$-ring (resp.
$S_b$-ring) if its Zariski spectrum is so.

\subsection{The Levine-Weibel Chow group}\label{sec:LWC-0}
Let $k$ be any field and $X$ a
reduced quasi-projective scheme of dimension $d \ge 2$ over $k$. Recall from
\cite[Definition~3.5]{Binda-Krishna} that
a Cartier curve on $X$ is a purely one-dimensional reduced closed 
subscheme $C \inj X$ none of whose irreducible components is contained in 
$X_{\rm sing}$ and whose defining sheaf of ideals is a local complete
intersection in $\sO_X$ at every point of $C \cap X_{\rm sing}$. 
Let $k(C)$ denote the ring of total quotients for a Cartier curve $C$ on $X$
and let $\{C_1, \ldots , C_r\}$ be the set of irreducible components 
of $C$. 
For $f \in \sO^{\times}_{C, C \cap X_{\rm sing}} \subset k(C)^{\times}$, 
let $f_i \in k(C_i)^{\times}$ be the image of
$f$ under the projection $\sO^{\times}_{C, C \cap X_{\rm sing}}
\inj k(C)^{\times} \surj k(C_i)^{\times}$.
We let $\divf(f) = \stackrel{r}{\underset{i =1}\sum} \divf(f_i) \in
\sZ_0(X^o)$ be the cycle associated to $f$.
We let $\sR^{LW}_0(X)$  denote the subgroup of $\sZ_0(X^o)$
generated by $\divf(f)$, where $C \subset X$ is a Cartier curve
and $f \in \sO^{\times}_{C, C \cap X_{\rm sing}}$.
The Levine-Weibel Chow group of $X$ is the quotient
${\sZ_0(X^o)}/{\sR^{LW}_0(X)}$ and is denoted by
$\CH^{LW}_0(X)$.

It is immediate from the definition that $\CH^{LW}_0(X)$ coincides with 
$\CH^F_0(X)$ if $X$ is regular. We also remark that the above definition of
$\CH^{LW}_0(X)$ is slightly different from that of \cite{Levine-Weibel}.
However, this difference disappears if $k$ is infinite as a consequence
of \cite[Lemma~1.4]{Levine-2}.
If $X$ is integral and projective over $k$, there is a degree map
$\deg \colon \CH^{LW}_0(X) \to \Z$ and we let $\CH^{LW}_0(X)^0$ be the
kernel of this map.

Let $\sR^{\Lci}_0(X)$ be the subgroup of $\sZ_0(X^o)$ generated by
$\nu_*(\divf(f))$ for $f \in \sO^{\times}_{C, \nu^{-1}(X_\sing)}$,
where $\nu \colon C \to X$ is a finite morphism
from a reduced curve of pure dimension one over $k$   
such that the image of none of the irreducible components of $C$ is contained in 
$X_{\rm sing}$ and $\nu$ is a local complete intersection (lci) morphism at every   
point $x \in C$ such that $\nu(x) \in X_\sing$.
Such a curve $C$ is called a good curve relative to $X_\sing$.
The lci Chow group of 0-cycles for $X$ is the quotient
$\CH^{\Lci}_0(X) = {\sZ_0(X^o)}/{\sR^{\Lci}_0(X)}$.
This modification of the Levine-Weibel Chow group was introduced in
\cite{Binda-Krishna}.

Clearly, the identity map of
$\sZ_0(X^o)$ induces a surjection $\CH^{LW}_0(X) \surj \CH^{\Lci}_0(X)$.
If $k$ is infinite, we can say the following.

\begin{lem}\label{lem:0-cycle-com-nor}
Assume that $k$ is infinite and $X$ is an $R_1$-scheme. Then the 
canonical map
$\CH^{LW}_0(X) \surj \CH^{\Lci}_0(X)$ is an isomorphism.
\end{lem}
\begin{proof}
Let $\nu \colon C \to X$ be a good curve relative to $X_\sing$, and let $Z=\nu^{-1}(X_\sing)$.
It suffices to show that $\nu_*({\rm div}(f)) \in \sR^{LW}_0(X)$ 
for any $f \in \sO^{\times}_{C, Z}$. Since $\nu$ is finite,
we can find a factorization
\begin{equation}\label{eqn:0-cycle-com-nor-0}
  \begin{array}{c}
\xymatrix@C1pc{
& \P^n_X \ar[d]^g \\
C \ar[ur]^i \ar[r]_{\nu} & X,}
  \end{array}
\end{equation}
where $i$ is a closed immersion and $g$ is the canonical projection.
Setting $X' = \P^n_X$, we see that $X'$ is an $R_1$-scheme and
$X'_{\rm sing} = g^{-1}(X_{\rm sing})$. In particular, $i^{-1}(X'_{\rm sing}) 
\subseteq Z$. Since $\nu$ is an lci morphism along $X_{\rm sing}$ and
$g$ is smooth, it follows that the closed immersion $i$ is regular along  
$X'_{\rm sing}$.  In other words, $C$ is embedded as a Cartier curve on $X'$.
One deduces that ${\rm div}(f) \in \sR^{LW}_0(X')$.  

Since $X'$ is $R_1$ and $k$ is infinite, 
it follows from \cite[Lemma~2.1]{BS} (see also \cite[Lemma~1.4]{Levine-2}) that
there are closed reduced curves $C'_i \subset X'$ and rational functions
$f_i \in k(C'_i)^{\times}$ such that $C'_i \subset X'^o$ and 
${\rm div}(f) = \sum_i {\rm div}(f_i)
\in \sZ_0(X'^o)$. If $g(C'_i)$ is a point, then clearly 
$g_*({\rm div}(f_i)) = 0$ as the map $C'_i \to X$ then factors through a 
regular 
closed point and $g_*({\rm div}(f_i))$ is already zero in the Chow group of the
closed point.
Otherwise, we let $C_i = g(C'_i)$ and assume that the map 
$C'_i \to C_i$ is finite. Then each $C_i \subset X$ is clearly a
Cartier curve as it does not meet $X_\sing$. Let $N_i: k(C'_i)^{\times} \to
k(C_i)^{\times}$ denote the norm map. We then have
\[
\nu_*({\rm div}(f)) = g_*({\rm div}(f)) = \sum_i g_*({\rm div}(f_i))
= \sum_i {\rm div}(N_i(f_i))
\]
and it is clear that ${\rm div}(N_i(f_i)) \in \sR^{LW}_0(X)$ 
for every $i$.
\end{proof}

\subsection{The moving lemma}\label{sec:ML}
One of the key ingredients for proving \thmref{thm:Main-1}
is a moving lemma for the Levine-Weibel Chow group over finite fields that we
shall prove in this subsection.

Let $k$ be any field and $X \in \Sch_k$ an integral $R_1$-scheme. Let
$A \subset X$ be a closed subset of $X$ of codimension at least two such that
$X_{\rm sing} \subseteq A$.
We let $\sR^{LW}_0(X, A)$ be the subgroup of
$\sZ_0(X \setminus A)$ generated by $\divf(f)$, where   
$f$ is a nonzero rational function
on an integral curve $C \subset X$ such that $C \cap A = \emptyset$.
We let $\CH^{LW}_0(X, A) := {\sZ_0(X \setminus A)}/{\sR^{LW}_0(X, A)}$.
We define $\CH^{\Lci}_0(X,A)$ in an analogous way.
The inclusion $\sZ_0(X \setminus A) \inj \sZ_0(X^o)$ preserves the
subgroups of rational equivalences. Hence, we get
canonical maps
\[
\CH^{LW}_0(X,A) \to \CH^{LW}_0(X), \ \ \CH^{\Lci}_0(X,A) \to \CH^{\Lci}_0(X).
\]

\begin{lem}\label{lem:LW-LCI-A}
  The canonical surjection $\CH^{LW}_0(X,A) \surj \CH^{\Lci}_0(X,A)$
  is an isomorphism.
\end{lem}
\begin{proof}
Let $\nu \colon C \to X$ be a finite morphism from an integral
curve such that $\nu(C) \cap A = \emptyset$.
It suffices to show that $\nu_*({\rm div}(f)) \in \sR^{LW}_0(X,A)$ 
for any $f \in k(C)^\times$. But the proof of this is identical to that of
\cite[Theorem~1.4]{Fulton} (see the last part of the proof
of \lemref{lem:0-cycle-com-nor}).
\end{proof}

We shall need the following application of the Bertini theorems of
Altman-Kleiman \cite{KL} and Wutz \cite{Wutz} (which is
a small modification of the Bertini theorem of Poonen \cite{Poonen-Annals}).

\begin{lem}\label{lem:Bertini-spl}
  Assume that $k$ is perfect, $\dim(X) \ge 2$ and $W \subset X^o_{(0)}$
  a finite set. Let $B \subset A$ be any closed subset containing $X_\sing$
  such that $B \cap W = \emptyset$.
  We can then find a smooth integral curve $C \subset X$ containing $W$
  such that $C \cap B = \emptyset$ and $C \not\subset A$.
\end{lem}
\begin{proof}
  We let $\nu \colon X_n \to X$ be the normalization morphism and
  let $\wt{B} = \nu^{-1}(B)$. 
  We choose a dense open immersion $X_n \inj Y$ such that $Y$ is an integral
  projective normal $k$-scheme. We let $B'$ be the Zariski closure of $\wt{B}$ in $Y$.
  Then $\dim(B') \le \dim(X) - 2$. We let $S \subset B'$ be a finite closed subset
  whose intersection with every irreducible component of $B'$ is nonempty.
  We fix a closed embedding $Y \inj \P^n_k$. We choose a closed point
  $x \in \pi^{-1}(X^o \setminus A)$ and set $Z = W \cup \{x\}$.
  
By the Bertini theorems of Altman-Kleiman \cite[Theorem~1]{KL}
(when $k$ is infinite) and Wutz \cite[Theorem~3.1]{Wutz}
(with $C = S$ and $T = H^0_\zar(S, \sO^\times_S)$ when $k$ is finite),
we can find a hypersurface $H \subset \P^n_k$ containing $Z$ and
disjoint from $S$ such that $Y_\reg \cap H$ is smooth. The condition
$S \cap H = \emptyset$ implies that $\dim(B' \cap H) \le \dim(B') - 1 \le \dim(X) -3$.

Since $Y$ is normal, it follows from \cite[Expos{\'e}~XII, Corollaire~3.5]{SGA-2}
that $Y \cap H$ is connected. In particular, $Y_\reg \cap H$ is connected and smooth.
Hence, it is integral. By iterating this process
$\dim(X) -1$ times, we get an integral curve $C' \subset Y$ containing $Z$
such that $\dim(B' \cap C') \le (\dim(X) - 2) - (\dim(X) -1) < 0$.
In particular, $B' \cap C' = \emptyset$ and $C' \not\subset \pi^{-1}(A)$.
We let $C = \pi(C' \cap X_n)$.
Then $C$ satisfies the desired properties.
\end{proof}

The moving lemma we want to prove is the following.

\begin{lem}\label{lem:Moving-nor}
 Let $k$ be any field. Then the canonical map
    \begin{equation}\label{eqn:Moving-nor-00}
      \CH^{LW}_0(X,A) \to \CH^{LW}_0(X)
    \end{equation}
    is an isomorphism. The same
holds also for the lci Chow group.
\end{lem}
\begin{proof}
We can assume that $\dim(X) \ge 2$ because the lemma is trivial otherwise. 
Before we begin the proof, we note that one always has the commutative diagram
\begin{equation}\label{eqn:Moving-nor-1}
\xymatrix@C.8pc{
\CH^{LW}_0(X,A) \ar[r] \ar[d] & \CH^{LW}_0(X) \ar[d] \\
\CH^{\Lci}_0(X,A) \ar[r] & \CH^{\Lci}_0(X).}
\end{equation}

Suppose now that $k$ is infinite. Since $X$ is $R_1$, 
the top horizontal arrow is known to be an isomorphism 
(see \cite[Lemma~2.1]{BS} and \cite[Lemma~1.4]{Levine-2}).
The right (resp. left) vertical arrow in ~\eqref{eqn:Moving-nor-1} is an isomorphism by
\lemref{lem:0-cycle-com-nor} (resp. \lemref{lem:LW-LCI-A}). It follows that the
bottom horizontal arrow in ~\eqref{eqn:Moving-nor-1} is an isomorphism.

We now assume that $k$ is finite and prove the
injectivity of ~\eqref{eqn:Moving-nor-00}.
Let $\alpha \in \CH^{\Lci}_0(X,A)$ 
be a cycle which dies in $\CH^{\Lci}_0(X)$.
Let $\ell_1 \neq \ell_2$ be two
distinct prime numbers and let $k_i$ be the pro-$\ell_i$ algebraic
extension of $k$ for $i =1,2$.
For $i = 1,2$, \cite[Proposition~6.1]{Binda-Krishna} says that
there is a commutative diagram
\begin{equation}\label{eqn:Moving-nor-2}
\xymatrix@C.8pc{
\CH^{\Lci}_0(X,A) \ar[r] \ar[d] & \CH^{\Lci}_0(X) \ar[d] \\
\CH^{\Lci}_0(X_{k_i}, A_{k_i}) \ar[r] & \CH^{\Lci}_0(X_{k_i}).}
\end{equation}

The bottom horizontal arrow is an isomorphism because $k_i$ is
infinite. It follows that $\alpha$ dies in $\CH^{\Lci}_0(X_{k_i}, A_{k_i})$.
It follows from \cite[Proposition~6.1]{Binda-Krishna} (by a 
straightforward modification, explained in
\cite[Proposition~8.5]{Gupta-Krishna-CFT}) that there is a finite
algebraic extension $k \inj k'_i$ inside $k_i$ such that
$\alpha$ dies in $\CH^{\Lci}_0(X_{k'_i}, A_{k'_i})$.
We apply \cite[Proposition~6.1]{Binda-Krishna} once again to conclude that
$\ell^n_i \alpha = 0$ for $i =1, 2$ and some $n \gg 1$.
Since $\ell^n_1$ and $\ell^n_2$ are relatively prime, it follows that
$\alpha = 0$. We have thus shown that the map
$\CH^{\Lci}_0(X,A) \to \CH^{\Lci}_0(X)$ is injective.

To show the same for the Levine-Weibel Chow group, we use
~\eqref{eqn:Moving-nor-1} again.
The left vertical arrow in this diagram is an isomorphism by
\lemref{lem:LW-LCI-A}. We just showed that the bottom
horizontal arrow is injective. It follows that the top horizontal
arrow is also injective.

It remains to show that ~\eqref{eqn:Moving-nor-00} is surjective when
$k$ is finite.  For this, we fix a closed point $x \in X^o$.  We need to show
that there is a 0-cycle $\alpha \in \sZ_0(X^o)$ supported on $X^o \setminus A$
such that the cycle class $[x]$ coincides with $\alpha$ in $\CH^{LW}_0(X)$.

To that end, we let $C \subset X$ be
a curve as in \lemref{lem:Bertini-spl} with $W = \{x\}$ and $B = X_\sing$.
We let $\CH^{F}_0(C, C \cap A)$ be the cokernel of the map
$\sO^{\times}_{C, C\cap A} \xrightarrow{\divf} \sZ_0(C \setminus A)$.
Using the isomorphism $\CH^{F}_0(C) \cong H^1_\zar(C, \sO^{\times}_C)$
and taking the colimit over the exact sequences for cohomology of
$\sO^{\times}_C$ with support in finite closed subsets 
$S \subset C \setminus A$, one easily checks that the canonical map
$\CH^{F}_0(C, C \cap A) \to \CH^{F}_0(C)$ is an isomorphism.
In particular, the 0-cycle $[x]$ on $C$ coincides with a 0-cycle $\alpha' \in
\CH^{F}_0(C)$ supported on $C \setminus A$.

Since $C \cap X_\sing = \emptyset$, we have a push-forward map
$\iota_* \colon \CH^F_0(C) \to \CH^{LW}_0(X)$, where $\iota \colon C \inj X$
is the inclusion. Letting $\alpha = \iota_*(\alpha')$, we achieve our claim.
This proves the surjectivity of ~\eqref{eqn:Moving-nor-00} for the
Levine-Weibel Chow group.
Since the map $\CH^{LW}_0(X) \to \CH^{\Lci}_0(X)$ is surjective by
definition, it follows that the map
$\CH^{\Lci}_0(X,A) \to \CH^{\Lci}_0(X)$ is also surjective.
This concludes the proof.
\end{proof}

We now draw some consequences of the moving lemma.
The first is an extension of \lemref{lem:0-cycle-com-nor} for integral
$R_1$-schemes over finite fields.

\begin{thm}\label{thm:Moving-nor-C}
Let $k$ be any field and $X \in \Sch_k$ an integral $R_1$-scheme. 
Then the canonical map
$\CH^{LW}_0(X) \to \CH^{\Lci}_0(X)$ is an isomorphism.
\end{thm}
\begin{proof}
By \lemref{lem:0-cycle-com-nor}, we assume that $k$ is finite.
  We look at the commutative diagram ~\eqref{eqn:Moving-nor-1}.
  Its left vertical arrow is an isomorphism (without any condition on $k$)
  by \lemref{lem:LW-LCI-A}. The two horizontal arrows are isomorphisms by
 \lemref{lem:Moving-nor}.  The theorem now follows.
\end{proof}

The next two applications show that the Levine-Weibel Chow groups of
$R_1$-schemes admit pull-back and push-forward maps. Note that
neither of these maps was previously known to exist.

\begin{cor}\label{cor:Pull-back-nor}
Let $k$ be any field and $X \in \Sch_k$ an integral $R_1$-scheme.
Let $f \colon X' \to X$ be a morphism in $\Sch_k$ such that
$X'_\sing \subseteq f^{-1}(X_\sing)$ and the resulting map
$f^{-1}(X_\reg) \to X_\reg$ is finite and surjective. Then 
$f^* \colon \sZ_0(X^o) \to \sZ_0(X'^o)$ induces a pull-back homomorphism
\[
f^* \colon \CH^{LW}_0(X) \to \CH^{LW}_0(X').
\]
This is an isomorphism if $f$ is the normalization morphism.
\end{cor}
\begin{proof}
The proof of the existence of $f^*$ is a routine construction following
\cite[Chapter~1]{Fulton} once we note that the map 
$f^{-1}(X_\reg) \to X_\reg$ is finite and flat.
The assertion that $f^*$ is an isomorphism if $f$ is the normalization 
follows directly from \lemref{lem:Moving-nor}.
\end{proof}

\begin{cor}\label{cor:PF-LW}
Let $f \colon X' \to X$ be a proper morphism of integral quasi-projective
schemes over a field $k$. Assume that $f^{-1}(X_{\rm sing})$ 
has codimension at least two in $X'$ and it contains
$X'_{\rm sing}$. Then the push-forward between the cycle groups
$f_* \colon \sZ_0(X' \setminus f^{-1}(X_{\sing})) \to \sZ_0(X^o)$ descends to
a homomorphism
\[
f_* \colon \CH^{LW}_0(X') \to \CH^{LW}_0(X).
\]
\end{cor}
\begin{proof}
The hypothesis of the corollary implies that $X'$ is $R_1$. We let 
$A' = f^{-1}(X_{\rm sing})$. Using 
\lemref{lem:Moving-nor} and the canonical map
$\CH^{LW}_0(X, X_{\rm sing}) \to \CH^{LW}_0(X)$, it suffices to construct the
homomorphism $f_* \colon \CH^{LW}_0(X', A') \to 
\CH^{LW}_0(X, X_{\rm sing})$.  But this can be achieved by repeating the
construction of the proper push-forward map for the classical Chow groups 
in \cite[\S~1.4]{Fulton}.
\end{proof}

\subsection{The cycle class map}\label{sec:CCM}
Let $k$ be any field and $X \in \Sch_k$ of pure dimension $d$. 
Let $\sK^M_{i, X}$ be the 
Zariski (or Nisnevich) sheaf of Milnor $K$-theory on $X$
(see \cite[\S~0]{Kato86}). Let $\sK_{i,X}$ denote the Zariski (or Nisnevich) sheaf 
of Quillen $K$-theory on $X$.
The product structures on the Milnor and Quillen $K$-theory yield
a natural map of sheaves $\sK^M_{i, X} \to \sK_{i,X}$.
In ~\eqref{eqn:Cycle-map}, we defined the cycle class homomorphism
(this could be trivial, e.g., if $X^o = \emptyset$)
\begin{equation}\label{eqn:Cycle-map-0} 
\cyc_{X} \colon \sZ_0(X^o) \to H^d_\zar(X, \sK^M_{d, X}).
\end{equation}

Using the Thomason-Trobaugh spectral sequence for Quillen $K$-theory
(see \cite[Theorem~10.3]{TT}), there
are canonical maps $H^d_\zar(X, \sK_{d, X}) \to H^d_\nis(X, \sK_{d, X}) \to 
K_0(X)$. These maps fit into a commutative diagram
(see \cite[Lemma~3.2]{Gupta-Krishna})

\begin{equation}\label{eqn:CCM-0}
\xymatrix@C.8pc{
& H^d_\zar(X, \sK^M_{d, X}) \ar[r] \ar[d] &
H^d_\nis(X, \sK^M_{d, X}) \ar[d] & \\
\sZ_0(X^o) \ar[r] \ar[ur]^-{\cyc_X} \ar@/_1cm/[rrr]^-{\wt{\cyc}_X} 
&  H^d_\zar(X, \sK_{d, X}) \ar[r] &
H^d_\nis(X, \sK_{d, X}) \ar[r] & K_0(X)}
\end{equation}
such that the composite map $\sZ_0(X^o) \to K_0(X)$ is the cycle class
(see \cite[Proposition~2.1]{Levine-Weibel})
map which takes a closed point $x \in X^o$ to the class $[k(x)] \in K_0(X)$.
Note also that any closed point $x \in X^o$ has a class $[k(x)] \in
K_0(X, D)$, where the latter is the 0-th homotopy group of the
$S^1$-spectrum $K(X, D)$ defined as the homotopy fiber of the
pull-back map of spectra $K(X) \to K(D)$ for any closed subscheme $D$
supported on $X_\sing$. Moreover, there is a factorization 
$\sZ_0(X^o) \to K_0(X, D) \to K(X)$.

As an analogue of Bloch's formula, one asks if the cycle class homomorphism
$\sZ_0(X^o) \to H^d_\nis(X, \sK^M_{d, X})$ factors through $\CH^{LW}_0(X)$
and, if the resulting map is an isomorphism when $X$ is reduced.
We shall prove some new results on this question in this paper.
We consider a special case in what follows.

\subsection{The case of isolated singularities}\label{sec:Isolated}
Let us now assume that $X$ is reduced, equidimensional, and
has only isolated singularities. Let $S$ denote the
finite set of singular points.
We consider the sequence of Zariski sheaves of Milnor $K$-groups
(see \cite[\S~0]{Kato86}):
\begin{equation}\label{eqn:Milnor-K}
\sK^M_{m, X} \xrightarrow{\epsilon}
\left(\begin{array}{c}{\underset{x \in X^{(0)}}\coprod}\ (i_x)_* K^M_{m}(k(x)) 
\\ \oplus \\
{\underset{P \in S}\coprod}\ (i_P)_* K^M_{m}(\sO_{X,P}) 
\end{array}\right) \xrightarrow{e_0}
\left(\begin{array}{c}
{\underset{x \in X^{(1)}}\coprod} \ (i_x)_* K^M_{m-1}(k(x)) 
\\ \oplus \\
{\underset{P \in S}\coprod}
{\underset{P \in \ov{\{x\}}, x \in X^{(0)}}\coprod}\ (i_P)_* K^M_{m}(k(x)) 
\end{array}\right)
\xrightarrow{e_1} \cdots
\end{equation}
\[
\cdots \xrightarrow{e_{d-1}}
\left(\begin{array}{c}
{\underset{x \in X^{(d)}}\coprod} \ (i_x)_* K^M_{m-d}(k(x)) 
\\ \oplus \\
{\underset{P \in S}\coprod}
{\underset{P \in \ov{\{x\}}, x \in X^{(d-1)}}\coprod}\ (i_P)_* K^M_{m-d+1}(k(x)) 
\end{array}\right)
\xrightarrow{e_d} 
\left(\begin{array}{c}
0 \\ \oplus \\
{\underset{P \in S}\coprod} \ (i_P)_* K^M_{m-d}(k(P)) 
\end{array}\right)
\to 0.
\]
Here, the map $\epsilon$ is induced by the inclusion into both terms and
the other maps are given by the matrices
\[
e_0 = \left(\begin{array}{cc}
\partial_1 & 0 \\
-\Delta & \epsilon
\end{array}\right), 
e_1 = \left(\begin{array}{cc}
\partial_1 & 0 \\
\Delta & \partial_2
\end{array}\right),
\cdots ,
e_d = \left(\begin{array}{cc}
0 & 0 \\
\pm \Delta & \partial_2
\end{array}\right)
\]
with $\partial_1$ and $\partial_2$ being the differentials of
the Gersten-Quillen complex for Milnor $K$-theory sheaves as
described in \cite{Kato86} and $\Delta$'s being the 
diagonal maps.

\begin{lem}\label{lem:Milnor-res}
The above sequence of maps forms a complex which gives a flasque resolution
of the sheaf $\epsilon(\sK^M_{m,X})$ in Zariski and Nisnevich topologies.
\end{lem}
\begin{proof}
A similar complex for the Quillen $K$-theory sheaves is constructed in 
\cite[\S~5]{Pedrini-Weibel86} and it is shown there that this complex is 
a flasque resolution of $\epsilon(\sK_{m,X})$. 
The same proof works here verbatim. On all stalks except at $S$, the 
exactness follows from \cite[Proposition~10]{Kerz10}. The
exactness at the points of $S$ is an immediate consequence of the way the
differentials are defined in ~\eqref{eqn:Milnor-K} (see
\cite{Pedrini-Weibel86} for details). 
\end{proof}

\begin{prop}\label{prop:Milnor-res-*}
Let $X \in \Sch_k$ be integral of dimension $d \ge 1$  
with only isolated singularities and let $\tau$ denote Zariski or
Nisnevich topology. Then there are canonical maps
\[
\CH^{LW}_0(X) \xrightarrow{\cyc_X} H^d_\tau(X, \sK^M_{d,X}) \to 
H^d_\tau(X, \sK_{d,X}) 
\to \CH^F_0(X),
\]
in which the middle arrow is an isomorphism. Furthermore, all arrows in the
middle square of ~\eqref{eqn:CCM-0} are isomorphisms.
\end{prop} 
\begin{proof}
The $d = 1$ case is well known (see \cite[Proposition~1.4]{Levine-Weibel}).
We can thus assume that $d \ge 2$. Let $S$ denote the singular locus of $X$ and
let $X^{(j)}_S$ denote the set of points $x \in X$ of codimension $j$ such that
$S \cap \ov{\{x\}} = \emptyset$. 
We first observe that the map of sheaves $\sK^M_{d, X} \surj 
\epsilon(\sK^M_{d, X})$ is generically an isomorphism
and the same holds for the Quillen $K$-theory sheaves.
It follows that (see \cite[Exercise~II.1.19, Lemma~III.2.10]{Hartshorne})
that the map $H^d_\tau(X, \sK^M_{d,X}) \to H^d_\tau(X, \epsilon(\sK^M_{d,X}))$
is an isomorphism and ditto for the Quillen $K$-theory sheaves.
It follows from \lemref{lem:Milnor-res} that both $H^d_\tau(X, \sK^M_{d,X})$ 
and $H^d_\tau(X, \sK_{d,X})$ are given by the middle homology of the complex
$\sC_X$: 
\[
\left(\begin{array}{c}
{\underset{x \in X^{(d-1)}}\coprod} \ K_{1}(k(x)) 
\\ \oplus \\
{\underset{P \in S}\coprod}
{\underset{P \in \ov{\{x\}}}\coprod}\ K_{2}(k(x)) 
\end{array}\right) 
\xrightarrow{e_{d-1}}
\left(\begin{array}{c}
{\underset{x \in X^{(d)}}\coprod} \ K_{0}(k(x)) 
\\ \oplus \\
{\underset{P \in S}\coprod}
{\underset{P \in \ov{\{x\}}}\coprod}\ K_{1}(k(x)) 
\end{array}\right)
\xrightarrow{e_d} 
\left(\begin{array}{c}
0 \\ \oplus \\
{\underset{P \in S}\coprod} \ K_{0}(k(P)) 
\end{array}\right).
\]

On the other hand, letting $\sC^0_X$ and $\sC^{F,0}_X$ denote the complexes
\[
{\underset{x \in X^{(d-1)}_S}\coprod} \ K_{1}(k(x)) 
\xrightarrow{{\rm div}}  {\underset{x \in X^{(d)}_S}\coprod} \ 
K_{0}(k(x)) \to 0 \ \ 
{\rm and}  
\]
\[
{\underset{x \in X^{(d-1)}}\coprod} \ K_{1}(k(x)) 
\xrightarrow{{\rm div}}  {\underset{x \in X^{(d)}}\coprod} \ K_{0}(k(x)) \to 0, 
\]
respectively, we see that there are canonical maps of chain complexes
$\sC^0_X \inj \sC_X \surj \sC^{F,0}_X$. This 
yields canonical maps $H_1(\sC^0_X) \to H^d_\tau(X, \sK^M_{d,X}) 
\xrightarrow{\cong} 
H^d_\tau(X, \sK_{d,X}) \to H_1(\sC^{F,0}_X)$. 
It follows however from \lemref{lem:Moving-nor} that $H_1(\sC^0_X) \cong 
\CH^{LW}_0(X)$. It is also clear that $H_1(\sC^{F,0}_X) \cong \CH^F_0(X)$.
The second assertion is clear. This concludes the proof.
\end{proof}

\section{Zero-cycles on surfaces}\label{sec:Surface}
In this section, we shall define the reciprocity maps. We shall then
restrict to the case of surfaces and prove several results which
will essentially be enough to prove the main theorems of this paper 
in this special case.

\subsection{The reciprocity maps}\label{sec:Rec}
Let $k$ be a finite field and $X \in \Sch_k$ an integral $R_1$-scheme of 
dimension $d \ge 1$. Given a closed point
$x \in X^o$, we have the push-forward map
$(\iota_x)_* \colon \pi_1(\Spec(k(x))) \to \pi_1(X^o)$,
where $\iota \colon \Spec(k(x)) \inj X^o$ is the inclusion.
We let $\phi_X([x])  = (\iota_x)_*(F_x)$, where $F_x$ is the
Frobenius automorphism of $\ov{k(x)}$, which is the topological generator
of $\Gal({\ov{k(x)}}/{k(x)}) \cong \pi_1(\Spec(k(x)))
\cong \wh{\Z}$. Extending it linearly, we get the reciprocity map
$\phi_X \colon \sZ_0(X^o) \to \pi^{\ab}_1(X^o)$.
By \cite[Theorem~2.5]{Kato-Saito-2}, the cycle class map
$\cyc_X \colon \sZ_0(X^o) \to H^d_\nis(X, \sK^M_{d,X})$ is surjective.

\begin{lem}\label{lem:Factor-rho}
  There exists a homomorphism $\rho_X \colon H^d_\nis(X, \sK^M_{d,X}) \to
  \pi^{\ab}_1(X^o)$ such that $\phi_X = \rho_X \circ \cyc_X$.
 \end{lem}
\begin{proof}
We let $\sK^M_{i, (X,Y)} = \Ker(\sK^M_{i,X} \surj \sK^M_{i,Y})$
  for any closed subscheme $Y \subset X$. We now look at the diagram
  \begin{equation}\label{eqn:Rec-map-00}
 \xymatrix@C.8pc{
   & {{\underset{Y}\varprojlim}\ H^d_\nis(X, \sK^M_{d,(X, Y)})} \ar[r]^-{\cong}
   \ar[d] \ar[dr]^-{\psi_X} &
   H^d_\nis(X, \sK^M_{d,X}) \ar@{.>}[d]^-{\rho_X} \\
   \sZ_0(X^o) \ar[ru]^-{\cyc'_X} \ar[r]^-{\eta_X}  \ar@/_1cm/[rr]^-{\phi_X} &
   {{\underset{Y, m}\varprojlim}\ {H^d_\nis(X, \sK^M_{d,(X, Y)})}/m} \ar[r]^-{\cong} &
 \pi^{\ab}_1(X^o),}
 \end{equation}
where $Y$ ranges over all closed subschemes of $X$ contained in $X_\sing$
  and $m$ ranges over all nonzero integers.
  Top horizontal arrow is the map induced on the cohomology by
  the inclusion of Nisnevich sheaves $\sK^M_{d, (X,Y)} \inj \sK^M_{d,X}$.
  All such maps are isomorphisms because $\dim(Y) \le d-2$. This explains
  why the top horizontal arrow is an isomorphism.
  The cycle class map $\cyc_X$ clearly factors through
  $\sZ_0(X^o) \to  H^d_\nis(X, \sK^M_{d,(X, Y)})$ for any $Y \subset X_\sing$
  because $H^d_{\{x\}}(X, \sK^M_{d,(X, Y)}) \xrightarrow{\cong}
  H^d_{\{x\}}(X, \sK^M_{d, X})$ for any closed point $x \in X^o$.
  The limit of all these maps is $\cyc'_X$.
  Hence, $\cyc_X$ is the composition of $\cyc'_X$ with the top horizontal
  arrow. 

  The middle vertical arrow is
  induced by the canonical surjections
  $H^d_\nis(X, \sK^M_{d,(X, Y)}) \surj {H^d_\nis(X, \sK^M_{d,(X, Y)})}/m$.
  The bottom horizontal arrow on the right is given by
  \cite[Theorem~9.1(3)]{Kato-Saito-2}, and is an isomorphism.
 The arrows $\eta_X$ and $\psi_X$ are defined so that
  the triangles on the two sides of the middle vertical arrow
  commute. It follows from
  \cite[Proposition~3.8(2)]{Kato-Saito-2} that $\psi_X \circ \cyc'_X = \phi_X$.
  We deduce from this that there is a unique homomorphism $\rho_X \colon
  H^d_\nis(X, \sK^M_{d,X})  \to \pi^{\ab}_1(X^o)$ which factors $\phi_X$
  via $\cyc_X$. 
  \end{proof}

\begin{lem}\label{lem:Rec-map-0}
For any closed subset $A \subset X$ of codimension at least two and
containing $X_\sing$, the reciprocity map $\phi_X$ descends to a homomorphism
\[
\phi_X \colon \CH^{LW}_0(X, A) \to \pi^{\ab}_1(X \setminus A).
\]
\end{lem}
\begin{proof}
If $d =1$, then $X$ is a connected smooth projective curve over $k$
(note that $k$ is perfect) with $A = \emptyset$
and one knows that $\CH^{LW}_0(X) \cong
\CH^F_0(X)$. The lemma therefore follows from the classical class field theory
in this case.

To prove that $\phi_X$ kills $\sR^{LW}_0(X,A)$ in general, we need to show that
$\phi_X(\divf(f)) = 0$ if $f$ is a nonzero rational function on an
integral projective curve $C \subset X \setminus A$.
So we choose any such curve $C$ and $f \in k(C)^{\times}$.
Let $\nu \colon C_n \to C \inj X$ be
the composite map, where $\nu$ is the normalization morphism.
We then have a commutative diagram
(see \cite[Lemma~5.1(1)]{Raskind})
\begin{equation}\label{eqn:Rec-map-1}
\xymatrix@C.8pc{
\sZ_0(C_n) \ar[r]^-{\phi_{C_n}} \ar[d]_-{\nu_*} & \pi^{\ab}_1(C_n) 
\ar[d]^-{\nu_*} \\
\sZ_0(X\setminus A) \ar[r]^-{\phi_X} & \pi^{\ab}_1(X\setminus A).}
\end{equation}

Since $\divf((f)_C) = \nu_*(\divf((f)_{C_n}))$, this diagram reduces the
problem to showing that $\phi_{C_n}(\divf((f)_{C_n})) = 0$. But this has
been shown above.
\end{proof}

In view of \lemref{lem:Moving-nor} (with $A=X_{\sing})$, 
\lemref{lem:Rec-map-0} implies the following.

\begin{cor}\label{cor:Rec-map-0-main}
The reciprocity map $\phi_X$ descends to a homomorphism
\[
\phi_X \colon \CH^{LW}_0(X) \to \pi^{\ab}_1(X^o).
\]
\end{cor}

It is clear that there is a commutative
diagram (with exact rows)
\begin{equation}\label{eqn:Rec-map-2}
\xymatrix@C.8pc{
0 \ar[r] & \CH^{LW}_0(X)^0 \ar[r] \ar[d]_-{\phi^0_X} &
\CH^{LW}_0(X) \ar[r]^-{\deg} \ar[d]^-{\phi_X} & \Z \ar@{^{(}->}[d] \\
0 \ar[r] & \pi^{\ab}_1(X^o)^0 \ar[r] & \pi^{\ab}_1(X^o) \ar[r] &
\wh{\Z}.}
\end{equation}

\subsection{Reciprocity map for surfaces}\label{sec:Rec*}
We assume now that $X$ is an integral projective $R_1$-scheme of dimension two
over a field $k$. 
By \propref{prop:Milnor-res-*}, there are canonical maps
\begin{equation}\label{eqn:Surface-0}
\CH^{LW}_0(X) \xrightarrow{\cyc_X} H^2_\tau(X, \sK^M_{2,X}) \to K_0(X),
\end{equation}
where $\tau = {\zar}/{\nis}$ and the composite arrow is the cycle class
map to $K$-theory. We let $F^2K_0(X)$ denote the image of this composite
arrow. 
For any closed subscheme $D \subset X$ supported on $X_\sing$, we let
$F^2K_0(X, D)$ denote the image of the cycle class map
$\sZ_0(X^o) \to K_0(X,D)$ (see \S~\ref{sec:CCM}). 
The main result about $\CH^{LW}_0(X)$ is the following.

\begin{prop}\label{prop:Surface-1}
Under the above assumptions, the following hold.
\begin{enumerate}
\item
There are cycle class maps
$\CH^{LW}_0(X) \to F^2K_0(X,D) \to F^2K_0(X)$ which are isomorphisms.
\item
  The maps $\CH^{LW}_0(X) \xrightarrow{\cyc_X} H^2_\tau(X, \sK^M_{2,X})
  \leftarrow H^2_\tau(X, \sK^M_{2,(X,D)})$
are isomorphisms for $\tau = {\zar}/{\nis}$.
\end{enumerate}
\end{prop}
\begin{proof}
  By \cite[Lemma~2.2]{Krishna-Srinivas}, there is an
  exact sequence
  \[
    SK_1(D) \to F^2K_0(X,D) \to F^2K_0(X) \to 0,
  \]
  where $SK_1(D)$ is the kernel of the edge map
  $K_1(D) \to H^0_\zar(D, \sO^\times_D)$ in the Thomason-Trobaugh spectral
  sequence (see \cite[Theorem~10.3]{TT}). But this edge map is an isomorphism
  because $D$ is a 0-dimensional $k$-scheme.

By \lemref{lem:Moving-nor} and \cite[Theorem~2.2]{Pedrini-Weibel86},
there is an exact sequence
\[
SK_1(X_\sing) \to \CH^{LW}_0(X) \to  F^2K_0(X) \to 0.
\]
The item (1) now follows because $SK_1(X_\sing) = 0$. 
To prove (2), it is enough to work with the Nisnevich topology by
\propref{prop:Milnor-res-*}.
We now have the maps
$\CH^{LW}_0(X) \xrightarrow{\cyc_X} H^2_\nis(X, \sK^M_{2,X}) \surj F^2K_0(X)$.
The map $\cyc_X$ is surjective by \cite[Theorem~2.5]{Kato-Saito-2} 
(see \cite[\S~3.5]{Gupta-Krishna-CFT} for an explanation as to why it
suffices to know that $X^o$ is regular instead of being nice).
The map $\cyc_X$  in (2) is now isomorphism by (1). The other map
in (2) is an isomorphism for dimension reason.
\end{proof}

\subsection{Zero-cycles on a surface and its desingularizations}
\label{sec:Res-surface}
Let $k$ be any field. We recall the definition of the Chow group of
0-cycles with modulus from \cite{Binda-Saito}.
Let $X$ be a quasi-projective scheme over $k$ and $D \subset X$ an 
effective Cartier divisor.
Given a finite map $\nu \colon C \to X$ from a regular integral
curve whose image is not contained in $D$, we let $E = \nu^*(D)$.
We say that a rational function $f$ on $C$ has modulus $E$ if
$f \in {\rm Ker}(\sO^{\times}_{C,E} \surj \sO^{\times}_E)$.
We let $G(C,E)$ denote the multiplicative subgroup of such rational
functions in $k(C)^{\times}$.  Then $\CH_0(X|D)$
is the quotient $\sZ_0(X \setminus D)$ by the subgroup
$\sR_0(X|D)$ generated by $\nu_*(\divf(f))$ for all possible choices of 
$\nu \colon C \to X$ and $f \in G(C,E)$ as above.

Let $D \subset X$ be as above. We have seen in \S~\ref{sec:CCM} that
there is a cycle class homomorphism
$\wt{\cyc}_{X|D} \colon \sZ_0(X^o \setminus D) \to K_0(X,D)$.
If $k$ is perfect and $X \in \Sm_k$, then it was shown in
\cite[Theorem~12.4]{Binda-Krishna} that $\cyc_{X|D}$ descends to a
group homomorphism
\begin{equation}\label{eqn:CCM-modulus}
\cyc_{X|D} \colon \CH_0(X|D) \to K_0(X,D).
\end{equation}
If $\dim(X) \le 2$, this map exists without any condition on $k$ by
\cite[Theorem~1.2]{Krishna-ANT}.
We let $F^dK_0(X,D)$ be the image of this map if $X$ is of pure dimension
$d$.

Let us now assume that $k$ is any field and $X$ is an integral 
projective $R_1$-scheme of dimension two over $k$. 
Let $f \colon X_n \to X$ be the normalization morphism and 
$f' \colon \wt{X} \to X_n$ a resolution of singularities 
of $X_n$ (assuming it exists).
Let $E \subset \wt{X}$ be the reduced exceptional divisor.
Then $\pi = f \circ f' \colon \wt{X} \to X$ is a resolution of
singularities of $X$ with reduced exceptional divisor $E$.
We let $E' = \pi^{-1}(X_\sing)_\red$ so that $E'$ is a union of $E$
and a finite set of closed points.
Note that such a resolution of singularities exists (e.g., by \cite{Lipman}).
We write $S = X_\sing$ and $S' = (X_n)_\sing$ with reduced structures.

\begin{prop}\label{prop:RES-Norm}
  Let $k$ be a field and $X$ an integral and projective $k$-scheme of dimension
  two which is $R_1$. Let $\wt{X} \xrightarrow{f'} X_n \xrightarrow{f} X$
  be the desingularization and normalization morphisms as above.
  Let $m \ge 1$ be any integer. With the above notations, we have a
  commutative diagram
\begin{equation}\label{eqn:0-C-map-I-0}
\xymatrix@C2.7pc{
\CH^{LW}_0(X) \ar[r]^-{\cyc_{(X,mS)}}_-{\cong} \ar[d]^-{f^*}_-{\cong} 
\ar@/_1.5cm/[dd]_-{\pi^*} & 
F^{2}K_0(X,mS) \ar[d]^-{f^*}_-{\cong} \ar[r]^-{\cong} & F^{2}K_0(X) 
\ar[d]^-{f^*}_-{\cong} \\
\CH^{LW}_0(X_n) \ar[r]^-{\cyc_{(X_n,mS')}}_-{\cong} \ar@{->>}[d]_-{f'^*} & 
F^{2}K_0(X_n, mS') \ar@{->>}[d]^-{f'^*} \ar[r]^-{\cong} & F^{2}K_0(X_n) 
\ar@{->>}[dl]^-{f'^*} \\
\CH_0(\wt{X}|mE) \ar@{->>}[r]_-{\cyc_{{\wt{X}}|mE}} & 
F^{2}K_0(\wt{X},mE). &}
\end{equation}
Moreover, all arrows are isomorphisms for $m \gg 1$.
\end{prop}
\begin{proof}
It is clear that the two squares on the top in the
diagram ~\eqref{eqn:0-C-map-I-0} are commutative.
Furthermore, all arrows in these squares are isomorphisms by
\propref{prop:Surface-1} and \corref{cor:Pull-back-nor}. It is also easy to check
using \lemref{lem:Moving-nor} that the pull-back
$f'^* \colon \sZ_0(X^o_n) \xrightarrow{\cong} \sZ_0(\wt{X} \setminus E)$
induces a pull-back map $f'^* \colon \CH^{LW}_0(X_n) \surj \CH_0(\wt{X}|mE)$
which makes the bottom square commutative. The map
$f'^* \colon  F^{2}K_0(X_n, mS') \to F^{2}K_0(\wt{X},mE)$ is an isomorphism
for all $m \gg 1$ by \cite[Theorem~1.1]{Krishna-Srinivas}.
It follows that all arrows in ~\eqref{eqn:0-C-map-I-0} are isomorphisms
for $m \gg 1$.
\end{proof}

\subsection{Reciprocity theorem for surfaces}\label{sec:RTS}
Assume now that $k$ is a finite field and $X$ an integral projective
$R_1$-scheme of dimension two over $k$. 
It follows from \propref{prop:Milnor-res-*} and \corref{cor:Rec-map-0-main}
that the cycle class and the 
reciprocity homomorphisms give rise to the degree preserving maps
\begin{equation}\label{eqn:CFT-surface-0*}
\CH^{LW}_0(X) \xrightarrow{\cyc_X} H^2_\nis(X, \sK^M_{2,X})
\xrightarrow{\rho_X} \pi^{\ab}_1(X^o).
\end{equation}

Our main result on the class field theory of $X$
is the following.

\begin{thm}\label{thm:CFT-surface}
The cycle class and the reciprocity homomorphisms induce 
isomorphisms of finite groups
\begin{equation}\label{eqn:CFT-surface-0}
\CH^{LW}_0(X)^0 \xrightarrow{\cong} H^2_\nis(X, \sK^M_{2,X})^0
\xrightarrow{\cong} \pi^{\ab}_1(X^o)^0.
\end{equation}
\end{thm}
\begin{proof}
  By \propref{prop:Surface-1}, we only have to show that the composition of
  the two maps in ~\eqref{eqn:CFT-surface-0} is an isomorphism of finite groups.
We now choose a resolution of singularities $\pi \colon \wt{X} \to X$
  (which exists by \cite{Lipman}) with the
  reduced exceptional divisor $E \subset \wt{X}$.
  For every integer $m \ge 1$, we then have a commutative diagram
  \begin{equation}\label{eqn:CFT-surface-1}
    \xymatrix@C1.5pc{
 \CH^{LW}_0(X, X_\sing)^0 \ar@{->>}[d]_-{\pi^*}
 \ar[r]^-{\rho_X} &  \pi^{\ab}_1(X^o)^0 \ar@{->>}[d] \\
  \CH_0(\wt{X}|mE)^0 \ar[r]^-{\rho_{\wt{X}|mE}} & \pi^{\ab}_1(\wt{X}, mE)^0,}
\end{equation}
where $\pi^{\ab}_1(\wt{X}, mE)$ is the abelianized {\'e}tale fundamental group
with modulus (see \cite[\S~8.3]{BKS} for definition) and
$\pi^{\ab}_1(\wt{X}, mE)^0$ is the kernel of the map
$\pi^{\ab}_1(\wt{X}, mE) \to {\rm Gal}({\ov{k}}/k)$.

The left vertical arrow in ~\eqref{eqn:CFT-surface-1} is an isomorphism
for all $m \gg 0$  by \propref{prop:RES-Norm}. Combining this with
\cite[Theorem~III]{Kerz-Saito-2} 
(if ${\rm char}(k) \neq 2$) and \cite[Lemma~8.4, Theorem~8.5]{BKS} (in general),
it follows that all arrows in
~\eqref{eqn:CFT-surface-1} are isomorphisms for all $m \gg 0$. Moreover,
these groups are finite by \cite[Corollary~8.3]{BKS}.
\end{proof}

\section{The Lefschetz condition}\label{sec:Leff}
In this section, we shall prove one of Grothendieck's Lefschetz conditions
as a prelude to our proof of the Lefschetz hypersurface section theorem for
the {\'e}tale fundamental groups of the regular loci of certain projective
schemes over a field. All cohomology groups in this section will be
with respect to the Zariski topology on schemes.

\subsection{Reflexive sheaves}\label{sec:Reflexive-mod}
Let $(X, \sO_X)$ be a locally ringed space.
If $(X, \sO_X)$ is an integral locally ringed space (i.e.,
the stalks of $\sO_X$ are integral domains), we let $\sK_X$
denote the sheaf of field of fractions of $\sO_X$.
Recall that for a sheaf of $\sO_X$-modules $\sE$ on $X$, the dual
$\sE^\vee$ is the sheaf of $\sO_X$-modules $\sHom_{\sO_X}(\sE, \sO_X)$.
There is a natural evaluation map ${\rm ev}_{\sE} \colon
\sE \to \sE^{\vee \vee}$ whose kernel is
the subsheaf of torsion submodules of $\sE$, where the latter
is defined as the kernel of the canonical map $\sE \to \sE \otimes_{\sO_X}
\sK_X$. We denote either of these kernels by $\sE_\tor$.
If $\sE$ is torsion-free, one calls $\sE^{\vee \vee}$ the reflexive hull of
$\sE$.
One says that $\sE$ is reflexive if the evaluation map
$\sE \to \sE^{\vee \vee}$ is an isomorphism.
The following lemma is elementary.

\begin{lem}\label{lem:D-ref}
The dual of any sheaf of $\sO_X$-modules is reflexive.
\end{lem}
\begin{proof}
  We need to show that for any sheaf of $\sO_X$-modules $\sM$ with dual $\sN$,
  the evaluation map
  ${\rm ev}_N \colon \sN \to \sHom_{\sO_X}(\sN^\vee, \sO_X)$ is an isomorphism.
  Since this is a local condition on $X$, we can assume that $X = \Spec(A)$ for
  a commutative ring $A$ and represent $\sM$ (resp. $\sN$) by $M$ (resp. $N$).
  
  Now, ${\rm ev}_N(f) = 0$ implies
  that $f(x) = {\rm ev}_N(f)({\rm ev}_M(x)) = 0$ for all $x \in M$.
  Equivalently, $f = 0$. This shows that ${\rm ev}_N$ is injective.
  To show the surjectivity, let
  $\alpha \in \Hom_{A}(N^\vee, A)$ and let $f_\alpha \colon M \to A$ be
  given by $f_\alpha(x) = \alpha({\rm ev}_M(x))$. It is then clear that
  $\alpha = {\rm ev}_N(f_\alpha)$.
\end{proof}

Let $(X, \sO_X)$ be a locally ringed space with Noetherian stalks
and let $\sE$ be a sheaf of $\sO_X$-modules. Recall that
$\sE$ is said to satisfy the $S_i$ property for some $i \ge 0$ if for
every $x \in X$, the $\sO_{X,x}$-module $\sE_x$ satisfies
Serre's $S_i$ condition, i.e., ${\rm depth}(\sE_x) \ge {\rm min}\{i, 
\dim(\sE_x)\}$ for every point $x \in X$.
One says that $X$ is a locally $S_i$-ringed
space (or an $S_i$-scheme if $X$ is a Noetherian scheme) if it has
Noetherian stalks and $\sO_X$ satisfies the $S_i$ property.
The following is easy to verify using \cite[Lemma~0AV6]{SP}.

\begin{lem}\label{lem:Depth-reflexive}
  Let $(X, \sO_X)$ be a locally $S_2$-ringed space
  and $\sE$ a coherent $\sO_X$-module. Then $\sE^\vee$
  satisfies the $S_2$ property. In particular, every reflexive
  coherent $\sO_X$-module satisfies the $S_2$ property.
\end{lem}

Let $k$ be a field. 
Suppose that $X$ is an integral $k$-scheme and
there is a locally closed immersion $X \inj \P^N_k$. We let
$\sO^r_X(m)$ be the restriction of the sheaf $(\sO_{\P^N_k}(m))^{\oplus r}$
onto $X$. Let $\sE$ be
a reflexive coherent sheaf on $X$. We can then prove the following.

\begin{lem}\label{lem:Reflexive}
  There are integers $q, q' \in \Z$, $r, r' \ge 1$, and a coherent sheaf
  $\sE'$ together with exact exact sequences
\begin{equation}\label{eqn:Reflexive-0}    
  0 \to \sE \to \sO^r_X(q) \to \sE' \to 0;
\end{equation}
\[
  0 \to \sE' \to \sO^{r'}_X(q') \to \sH \to 0.
\]
In particular, each of $\sE$ and $\sE'$ is torsion-free (or zero). If $\sE$ is
locally free, so is $\sE'$.
\end{lem}
\begin{proof}
  Since $\sE^\vee$ is coherent, there is a surjection $\sO^r_X(-q) \surj \sE^\vee$
  for some $q, r \gg 0$. We let $\sF$ be the kernel of this surjection.
  Since $\sF$ is necessarily coherent, we also have a surjection
  $\sO^{r'}_X(-q') \surj \sF$ for some $q', r' \gg 0$. Letting $\sF'$ be
  its kernel,  we get short exact sequences of coherent sheaves
  \[
    0 \to \sF \to \sO^r_X(-q) \to \sE^\vee \to 0;
  \]
  \[
    0 \to \sF' \to \sO^{r'}_X(-q') \to \sF \to 0.
  \]
 By dualizing, we get exact sequences
 \begin{equation}\label{eqn:Reflexive-1}
   0 \to \sE \to \sO^r_X(q) \to \sF^\vee \to \sExt^1_{\sO_X}(\sE^\vee, \sO_X)
   \to 0;
 \end{equation}
 \[
   0 \to \sF^\vee \to \sO^{r'}_X(q') \to \sF'^\vee \to
   \sExt^1_{\sO_X}(\sF, \sO_X) \to 0.
   \]

   Letting $\sE'$ be the cokernel of $\sE \to \sO^r_X(q)$, we get the
   two exact sequences of ~\eqref{eqn:Reflexive-0}. If $\sE$ is locally free, then
 so is $\sE^\vee$. In this case, $\sExt^1_{\sO_X}(\sE^\vee, \sO_X) = 0$ and
 $\sF$ must also be locally free. This implies that $\sE'$ is locally free.
\end{proof}

\subsection{The Hartogs lemma}\label{sec:Hartogs}
We need to prove a version of the Hartogs lemma for formal schemes.
Before we do this, we recall this result for the ordinary schemes.

\begin{lem}\label{lem:Local-Hartogs}
  Let $A$ be a Noetherian integral domain such that $X = \Spec(A)$ is an
$S_2$-scheme.
    Let $U \subset X$ be an open subscheme whose complement has codimension at
    least two in $X$. Let $E$ be a finitely generated reflexive $A$-module
    and $\sE$ the associated Zariski sheaf on $X$. Then the canonical map
    $E \to H^0(U, \sE_U)$ is an isomorphism.
  \end{lem}
  \begin{proof}
    Since $E$ is reflexive, we can find an exact sequence of $A$-modules
    \begin{equation}\label{eqn:Local-Hartogs-0}
      0 \to E \to A^r \to A^{r'}
      \end{equation}
      for some $r, r' \ge 1$.
      This gives rise to a commutative diagram of exact sequences
    \begin{equation}\label{eqn:Local-Hartogs-1}
      \xymatrix@C.8pc{
        0 \ar[r] & E \ar[r] \ar[d]_-{j^*} & A^r \ar[r] \ar[d]^-{j^*}
        & A^{r'} \ar[d]^-{j^*} \\
        0 \ar[r] & H^0(U, \sE_{U}) \ar[r] & \sO(U)^r \ar[r]
        & \sO(U)^{r'},}
    \end{equation}
    where $j \colon U \inj X$ is the inclusion.
    A diagram chase shows that we can assume that $E = A$. But this case
    follows from \lemref{lem:Hartog-classical}.
\end{proof}

The following lemma is well known.

\begin{lem}\label{lem:Hartog-classical}
    Let $A$ be a Noetherian integral domain with quotient field $K$
    which is an $S_2$-ring. Then
    $A = {\underset{{\rm depth} \ A_\fp = 1}\bigcap} \ A_{\fp}$ inside $K$.
  \end{lem}
  \begin{proof}
    The proof of this lemma is identical to that of  \cite[Lemma~031T(2)]{SP}
    once we observe that in a Noetherian integral domain which is an $S_2$-ring, a prime ideal has height one if and only if its depth is one.
    We add a proof of the lemma for completeness.
    
    Suppose $a, b \in A$ are two nonzero elements such that $a \in bA_{\fp}$
    for every prime ideal $\fp \subset A$ of depth one.
    We need to show that $a \in (b)$.

    Let ${\rm Ass}(b)$ denote the set of associated primes of $(b)$.
    Since the integral domain $A$ is an $S_2$-ring  and $b \neq 0$, an easy consequence of
    \cite[Propositions~16.4.6(ii), 16.4.10(i)]{EGA-4} is that ${A}/{(b)}$ is an
    $S_1$-ring (e.g., see \cite[Lemma~2.1]{Ghosh-Krishna}). In particular,
    all associated primes of ${A}/{(b)}$ have height zero.
    The latter statement is equivalent to saying that all associated
    primes of $(b)$ are minimal, and hence have height one by
    Krull's principal ideal theorem (see \cite[Theorem~13.5]{Matsumura}).
    Since $1 \le {\rm depth}(A_{\fp}) \le \dim(A_{\fp}) \le 1$, we
    conclude that ${\rm depth}(A_{\fp}) = 1$ for every $\fp \in {\rm Ass}(b)$.
    This implies by our hypothesis that
    $a \in bA_{\fp}$ for every $\fp \in {\rm Ass}(b)$.
    We are now done because the canonical map
    \[
      {A}/{(b)} \to {\underset{\fp \in {\rm Ass}(b)}\prod}
      \ {A_{\fp}}/{bA_{\fp}}
    \]
    is injective (e.g., see \cite[Lemma~0311]{SP}).
    \end{proof}

    \begin{lem}\label{lem:Global-Hartogs}
  Let $X$ be a Noetherian integral $S_2$-scheme and $j \colon U \inj X$ an
  open immersion whose complement has codimension at least two in $X$.
  Then the unit of adjunction map $\sE \to j_* \sE_U$ is an isomorphism for
  every reflexive coherent Zariski sheaf $\sE$ on $X$.
\end{lem}
\begin{proof}
  This is a local question and hence follows directly from
  \lemref{lem:Local-Hartogs}.
  \end{proof}
  
  \begin{cor}\label{cor:Global-Hartogs-0}
    Let $\sE$ be as in \lemref{lem:Global-Hartogs}. Then the
    restriction map $H^0(X, \sE) \to H^0(U, \sE_U)$ is an isomorphism.
  \end{cor}
  \begin{proof}
Immediate from \lemref{lem:Global-Hartogs}.
\end{proof}

\begin{cor}\label{cor:Reflexive-PF}
  Let $U \subset X$ be an open immersion as in \lemref{lem:Global-Hartogs} and 
  $\sE$ a reflexive coherent sheaf on $U$. Then $j_* \sE$ is 
  a reflexive coherent sheaf on $X$.
\end{cor}
\begin{proof}
  By choosing a coherent extension of $\sE$ on $X$ and taking its double dual,
  we can find a reflexive coherent sheaf $\sE'$ on $X$ such that
  $j^* \sE' \cong \sE$. We now apply \lemref{lem:Global-Hartogs} to conclude
  the proof.
\end{proof}

\subsection{The formal Hartogs lemma}\label{sec:FHL}
Let $X$ be a Noetherian scheme and $Y \subset X$ a closed subscheme.
Let $\wh{X}$ denote the formal completion of $X$ along $Y$ (see
\cite[Chapitre~0, \S~9]{EGA-I} or
\cite[Chapter~II, \S~10]{Hartshorne}). Let $\sI_Y$ denote
the sheaf of ideals on $X$ defining $Y$. 
Let $Y_m$ denote the closed subscheme of $X$ defined by the ideal sheaf
$\sI^m_Y$. For any open subscheme $U \subset X$ such that $V:= U \cap Y$ is
dense in $Y$, we have a commutative diagram of Noetherian formal schemes
\begin{equation}\label{eqn:Formal-completion}
  \xymatrix@C.8pc{
    V_m \ar[r]^-{\iota_U} \ar[d]_-{j_Y} & \ar[r] \wh{U} \ar[r]^-{c_U}
    \ar[d]_{\wh{j}} & U \ar[d]^-{j} \\
  Y_m \ar[r]^-{\iota} & \wh{X} \ar[r]^-{c_X} & X.}
\end{equation}
where $V_m := U\cap Y_m$. In the right square, all arrows are flat morphisms
(of locally ringed spaces),
the vertical arrows are open immersions (see \cite[\S~01HD]{SP}
for the definition of open immersion of locally ringed spaces)
and the horizontal arrows are the completion morphisms.
In the left square, the vertical arrows are open immersions and
the horizontal arrows are closed immersions. The two square are Cartesian.
The compositions of the two horizontal arrows are the given closed
immersions $Y_m \inj X$ and $V_m \inj U$ of schemes.

For any quasi-coherent sheaf $\sF$ on $X$, we let
$\wh{\sF}$ denote the
pull-back of $\sF$ under $c_X$. Note that
the canonical map $\wh{\sF} \to {\underset{m}\varprojlim} \ \sF \otimes_{\sO_X}
\sO_{Y_m}$ is an isomorphism if $\sF$ is coherent. 
We shall use this isomorphism later. For any morphism 
$f \colon X' \to X$ and quasi-coherent sheaf $\sF$ on $X$, we let 
$\sF_{X'} = f^*(\sF)$. We shall write $\sF_{Y_m}$ simply as $\sF_m$ for 
$m \ge 1$.
Our goal in this subsection is to prove a formal version of the
Hartogs lemma.

Let $A$ be an excellent normal integral domain and $J \subset A$
    the radical ideal such that $V(J) = X_\sing$, where  $X = \Spec(A)$. 
Let $\fp \subset A$
    be a complete intersection prime ideal such that ${\rm ht}(\fp + J) \ge
    {\rm ht}(\fp) + {\rm ht}(J)$
    and $A/{\fp}$ is normal. Let $U \subset X$ be
    an open subscheme containing $V(\fp) \cap X^o$ whose complement has
    codimension at least two in $X$. Let $\wh{A}$ be the
    $\fp$-adic completion of $A$ and $\wh{X} = \Spf(\wh{A})$, the formal
spectrum of $\wh{A}$. Let
    $\wh{U}$ be the formal completion of
    $U$ along $V(\fp) \cap U$. We let $Y = \Spec(A/{\fp})$ and 
$A_m:=A/{\fp^m}$. Since $A$ is excellent, so is $U$. It follows therefore
    from \cite[Corollaire~6.5.4, Scholie~7.8.3(v)]{EGA-IV} that
    $\wh{X}$ and $\wh{U}$ are both normal integral formal schemes.
Under these conditions, we have the following.

\begin{lem}\label{lem:Local-analytic-Hartogs}  
    For any finitely generated reflexive $A$-module $E$ with the associated
    Zariski sheaf $\sE$ on $X$, the canonical
    map $\wh{E} \to H^0(\wh{U}, \wh{\sE}_U)$ is an isomorphism.
  \end{lem}
  \begin{proof}
Since $E$ is reflexive, we can find an exact sequence of $A$-modules
    \begin{equation}\label{eqn:Local-analytic-Hartogs-0}
      0 \to E \to A^r \to A^{r'}
      \end{equation}
      for some $r, r' \ge 1$.
      Since all arrows in the right square of ~\eqref{eqn:Formal-completion} are
      flat, we have an exact sequence of coherent sheaves

\begin{equation}\label{eqn:Local-analytic-Hartogs-1}
  0 \to \wh{\sE} \to \sO^r_{\wh{X}} \to \sO^{r'}_{\wh{X}}
\end{equation}
on $\wh{X}$. This gives rise to a commutative diagram of exact sequences
    \begin{equation}\label{eqn:Local-analytic-Hartogs-2}
      \xymatrix@C.8pc{
        0 \ar[r] & \wh{E} \ar[r] \ar[d]_-{\wh{j}^*} & \wh{A}^r \ar[r]
        \ar[d]^-{\wh{j}^*}
        & \wh{A}^{r'} \ar[d]^-{\wh{j}^*} \\
        0 \ar[r] & H^0(\wh{U}, \wh{\sE}_{\wh{U}}) \ar[r] & \sO(\wh{U})^r \ar[r]
        & \sO(\wh{U})^{r'},}
    \end{equation}
    where $j \colon U \inj X$ is the inclusion.
    A diagram chase shows that we can assume that $E = A$.

We now let $V = Y \cap U$. Let $j \colon U \inj X$
    and $\wh{j} \colon \wh{U} \inj \wh{X}$ denote the inclusion maps.
    We then have a commutative diagram
    \begin{equation}\label{eqn:Local-analytic-Hartogs-3}
      \xymatrix@C.8pc{
        \wh{A} \ar[r]^-{\cong} \ar[d]_-{\wh{j}^*} &
        {\underset{m \ge 1}\varprojlim} \ A_m
      \ar[d]^-{\wh{j}^*} \\
      H^0(\wh{U}, \sO_{\wh{U}}) \ar[r]^-{\cong} &
      {\underset{n \ge 1}\varprojlim} \ H^0(V_m, \sO_{V_m}).}
    \end{equation}

The top horizontal arrow is clearly an isomorphism and
    the Milnor exact sequence for the cohomology of
    inverse limit sheaves (e.g., see \cite[Proposition~1.6]{Jannsen})
    implies that the bottom horizontal arrow
    is also an isomorphism. It suffices therefore to show that the
    right vertical arrow is an isomorphism.
    We shall show by induction the stronger assertion that
    the restriction map $A_m \to H^0(V_m, \sO_{V_m})$ is an isomorphism
    for all $m \ge 1$.

Let us first assume that $m =1$. It follows from our assumption that
    $Y$ is a Noetherian normal integral scheme in which the codimension
    of $Y \setminus V$ is at least two. Therefore, the map
    $A_1 \to H^0(V, \sO_{V}) = H^0(V_1, \sO_{V_1})$ is an isomorphism
    by \lemref{lem:Local-Hartogs}.
We now prove the general case by induction on $m$.
    We consider the exact sequence of $A_{m+1}$-modules
    \begin{equation}\label{eqn:Local-analytic-Hartogs-4}
      0 \to {\fp^m}/{\fp^{m+1}} \to A_{m+1} \to A_m \to 0.
    \end{equation}
    Since $\fp \subset A$ is a complete intersection,
    it follows that ${\fp^m}/{\fp^{m+1}} \cong S^m({\fp}/{\fp^{2}})$ is
    a free $A/{\fp}$-module. That is, ${\fp^m}/{\fp^{m+1}} \cong A^{q}_1$
    for some $q \ge 1$.
    We thus get an exact sequence of $A_{m+1}$-modules
  \begin{equation}\label{eqn:Local-analytic-Hartogs-5}
      0 \to A^q_1 \to A_{m+1} \to A_m \to 0.
    \end{equation}  
    This gives rise to the exact sequence of coherent $\sO_X$-modules
 \begin{equation}\label{eqn:Local-analytic-Hartogs-6}
      0 \to \sO^q_{Y} \to \sO_{Y_{m+1}} \to \sO_{Y_m} \to 0.
    \end{equation}

Considering the cohomology groups and comparing them on $X$ and $U$,
    we get a commutative diagram of exact sequences
 \begin{equation}\label{eqn:Local-analytic-Hartogs-7}   
   \xymatrix@C.8pc{
     0 \ar[r] & A^q_1 \ar[r] \ar[d] & A_{m+1} \ar[r] \ar[d] & A_m \ar[r] 
\ar[d] &
     0 \\
     0 \ar[r] & {H^0(V_1, \sO_{V_1})}^q \ar[r] & H^0(V_{m+1}, \sO_{V_{m+1}}) 
\ar[r] &
     H^0(V_m, \sO_{V_m}) \ar[r]^-{\partial} & H^1(V_1, \sO_{V_1})^q .}
 \end{equation}
  The left and the right vertical arrows are isomorphisms by induction on $m$.
 It follows from the 5-lemma that the middle vertical
 arrow must also be an isomorphism. This concludes the proof of the
 claim and the lemma.   
\end{proof}

   Let $X$ be an excellent normal integral scheme and $Y \subset X$
a local complete intersection closed subscheme which is integral and normal.
Let $U \subset X$ be an open subscheme containing $Y \cap X^o$ such that
$X \setminus U$ has codimension at least two in $X$. 
Under these hypotheses, we have the following `formal Hartogs lemma'.

\begin{lem}\label{lem:Global-analytic-Hartogs}
  Let $\sE$ be a reflexive coherent sheaf on $X$. Then the unit of adjunction
  map $\wh{\sE} \to \wh{j}_*(\wh{\sE}_{U})$ is an isomorphism.
\end{lem}
\begin{proof}
  Since this is a local question, we can assume that $X = \Spec(A)$ is affine
  and $Y$ is a complete intersection on $X$. In this case, $\sE$ is
  the Zariski sheaf associated to a finitely generated reflexive
  $A$-module $E$. Furthermore, the lemma is equivalent to showing that
  the canonical map $\wh{E} \to H^0(\wh{U}, \wh{\sE}_U)$ is an
  isomorphism. But this is the content of \lemref{lem:Local-analytic-Hartogs}.
\end{proof}

\begin{cor}\label{cor:Global-analytic-Hartogs-0}
  Let $\sE$ be as in \lemref{lem:Global-analytic-Hartogs}. Then the
  restriction map $H^0(\wh{X}, \wh{\sE}) \to H^0(\wh{U}, \wh{\sE}_{U})$
  is an isomorphism.
\end{cor}
\begin{proof}
Immediate from \lemref{lem:Global-analytic-Hartogs}. 
  \end{proof}

\subsection{The $\Lef(X,Y)$ condition}\label{sec:LEF**}
Let $X$ be a Noetherian scheme and $Y \subset X$ a closed (resp. open)
subscheme. We shall then say that $(X,Y)$ is a closed (resp. an open) pair.
Recall from \cite[\S~2, Expos{\'e}~X]{SGA-2}
that a closed pair $(X,Y)$ is said to satisfy the Lefschetz condition
and one says that $\Lef(X,Y)$ holds, if for any open
subscheme $U \subset X$ containing $Y$ and any coherent locally free sheaf
$\sE$ on $U$, the restriction map $H^0(U, \sE) \to H^0(\wh{X}, \wh{\sE})$ is
an isomorphism.
We shall not recall the effective Lefschetz condition because we do not 
need it.

We shall work with the following setup throughout \S~\ref{sec:LEF**}.
We fix a field $k$ and let $X \inj \P^N_k$ be an integral normal projective
scheme over $k$ of dimension $d \ge 3$.  Suppose we are given a closed
immersion $\iota \colon Y \inj X$ of integral normal schemes such that
$X^o \cap Y \subset Y^o$ and $2 \le \dim(Y) \le d-1$. We further assume that
$Y \subset X$ is cut out by $e$ general hypersurfaces in $\P^N_k$
such that $Y$ is a complete intersection in $X$, where $e = {\rm codim}(Y,X)$.
Let $U \subset X$ be an open subscheme containing $X^o \cap Y$ such that
$X \setminus U$ has codimension at least two in $X$.
Let $j \colon U \inj X$ be the inclusion map.
We shall let $\wh{X}$ (resp. $\wh{U}$) denote the formal completion of
$X$ (resp. $U$) along $Y$ (resp. $U \cap Y$). We shall continue to use the
notations of ~\eqref{eqn:Formal-completion}.

\begin{lem}\label{lem:Lef-proj}
For any coherent reflexive sheaf $\sE$ on $X$, the pull-back map
  $H^0(X, \sE) \to H^0(\wh{X}, \wh{\sE})$ is an isomorphism.
\end{lem}
\begin{proof}
Using the Milnor exact sequence for the cohomology of inverse limit sheaves,
the lemma is equivalent to showing that the map
$H^0(X, \sE) \to {\underset{m \ge 1}\varprojlim} \ H^0(Y_m, \sE_m)$
is an isomorphism. But this follows immediately from
\lemref{lem:Depth-reflexive} and \cite[Proposition~0EL1]{SP} since $X$ is normal
(hence $S_2$) and $d \ge 3$.
\end{proof}

\begin{lem}\label{lem:Lef-0}
  For any coherent reflexive sheaf $\sE$ on $U$, the pull-back map
  $H^0(U, \sE) \to H^0(\wh{U}, \wh{\sE})$ is an isomorphism.
\end{lem}
\begin{proof}
  We let $\sE' = j_* \sE$. It follows from \corref{cor:Reflexive-PF} that
  $\sE'$ is a reflexive coherent sheaf on $X$. We now consider the
  commutative diagram
  \begin{equation}\label{eqn:Lef-0-0}
    \xymatrix@C.8pc{
      H^0(X, \sE') \ar[r]^-{c^*_X} \ar[d]_-{j^*} & H^0(\wh{X}, \wh{\sE'})
      \ar[d]^-{\wh{j}^*} \\
      H^0(U, \sE) \ar[r]^-{c^*_U} & H^0(\wh{U}, \wh{\sE}).}
      \end{equation}

The top horizontal arrow is an isomorphism by \lemref{lem:Lef-proj}.
       Corollary~\ref{cor:Global-Hartogs-0} implies that
      the left vertical arrow is an isomorphism. Since $X$ is excellent
      and normal, Corollary~\ref{cor:Global-analytic-Hartogs-0} implies that
      the right vertical arrow is an isomorphism. A diagram chase now
      finishes the proof.
\end{proof}

\begin{lem}\label{lem:CI-reg}
  One has $Y^o = X^o \cap Y$.
\end{lem}
\begin{proof}
Since we are already given that $ X^o \cap Y \subset Y^o$, the lemma 
follows because at any point $x \in
X_\sing \cap Y$, the ideal of $Y$ is defined by a regular sequence.
\end{proof}

It follows from \lemref{lem:CI-reg} that $(X^o, Y^o)$ is a closed pair.
We can now prove:

\begin{prop}\label{prop:LEF*}
  Assume that $Y$ intersects $X_\sing$ properly in $X$.
  Then $\Lef(X^o,Y^o)$ holds.
\end{prop}
\begin{proof}
Let $U \subset X^o$ be an open subscheme containing $Y^o$.
In view of \lemref{lem:Lef-0}, we only need to show that $X \setminus U$
has codimension at least two in $X$. Suppose to the contrary that
there is an irreducible closed subscheme $Z \subset X \setminus U$ of
dimension $d-1$. We must then have that  the
scheme theoretic intersections $Z \cap Y$ and
$Z \cap (Y \cap X_\sing)$ have same support, and therefore 
$\dim(Z \cap Y) \le \dim(Y\cap X_{\sing})$. We also get
\[
    \dim(Z\cap Y) \ge \dim(Z) - e = d -1- (d-\dim(Y))=\dim(Y)-1 > 
\dim(Y \cap X_{\sing}),
\]
since $Y \subset X$ is cut out by $e$ general hypersurfaces in $\P^N_k$, the 
intersection of $Y$ and $X_\sing$ is proper and $X$ is catenary. This is 
clearly a contradiction.
\end{proof}

\section{Lefschetz for {\'e}tale fundamental group}
\label{sec:LFG}
In this section, we shall prove a Lefschetz hypersurface section
theorem for the abelianized {\'e}tale fundamental group of
the regular locus of a normal projective scheme over a field
under certain conditions. Our setup for this section is the following.

Let $k$ be a field of exponential characteristic $p \ge 1$
and $X \inj \P^N_k$ an integral projective
scheme over $k$ of dimension $d \ge 3$.  
We let $H \subset \P^N_k$ be a hypersurface such that
the scheme theoretic intersection $Y = X \cap H$ satisfies the following.

\begin{enumerate}
\item
$Y$ is integral of dimension $d-1$.
\item
$X^o \cap Y$ is regular.
\item
$Y$ is normal if $X$ is so.
\item
$Y$ does not contain any irreducible component of $X_\sing$.
\end{enumerate}

A hypersurface section $Y$ satisfying the above conditions will be called a `good' hypersurface section. Note that if $X$ is an $R_a$-scheme for some $a\ge 0$ and
$Y$ is good, then $Y$ too is an $R_a$-scheme by (2) and (4).
Since $X_\sing$ is reduced, it follows from (4) that the scheme theoretic
intersection $Y_s := Y \cap X_\sing$ is 
an effective Cartier divisor on $X_\sing$ (see \cite[Lemma~3.3]{Gupta-Krishna})
and $(Y_s)_\red = Y_\sing$ (see \lemref{lem:CI-reg}).
We let 
\[
\iota \colon Y \inj X, \ \iota^o \colon Y^o \inj X^o, \
j \colon X^o \inj X, \ \tilde{j} \colon Y^o \inj Y, \
\tilde{\iota} \colon Y_s \inj X_\sing
\]
be the inclusions. We let $q = \dim(Y_s)$.

\subsection{The Enriques-Severi-Zariski theorem}
\label{sec:ESZT}
We shall need a version of the Enriques-Severi-Zariski theorem 
for some singular projective schemes and their regular loci.
Suppose that $H \subset \P^N_k$ is a hypersurface of degree $m \ge 1$ such that $Y=X\cap H$ is good.
Let $C_{Y/X}$ denote the conormal sheaf on $Y$ associated
to the regular embedding $\iota \colon Y \inj X$ so that $C_{Y/X} = {\sI_Y}/{\sI^2_Y}
\cong \sO_Y(-m)$, where $\sI_Y$ is the sheaf of ideals on $X$ defining $Y$.
For any coherent sheaf $\sF$ on $Y$ and integer $n \ge 1$, we let
$\sF^{[n]} := \sF \otimes_{\sO_Y} S^n(C_{Y/X}) \cong \sF(-nm)$.
For any coherent sheaf $\sF$ on $Y^o$ and integer $n \ge 1$, we let
$\sF^{[n]} := \sF \otimes_{\sO_{Y^o}} S^n(C_{Y^o/X^o}) \cong \sF(-nm)$.

\begin{lem}\label{lem:ESZ-0}
Assume that $X$ is an $(R_1 + S_b)$-scheme for some $b \ge 2$. 
Let $\sE$ be a locally free sheaf on $X$. Then $H^i_\zar(X, \sE(-j)) =0$
for $i \le b-1$ and $j \gg 0$.
\end{lem}
\begin{proof}
Let $\omega^{\bullet}_{X/k}$ denote the dualizing complex for $X$.
Under the assumption of the lemma, it follows 
from \cite[Lemma~4.27]{Lee-Nakayama} that 
$\omega^{\bullet}_{X/k}\in D^{[-d, -b]}(X)$.
Moreover, we have $\sH^{-d}(\omega^{\bullet}_{X/k}) \cong j_*(\omega_{{X^o}/k})$,
where $\omega_{{X^o}/k}$ is the canonical sheaf of $X^o$.
Since $X$ is an $(R_1 + S_b)$-scheme for some $b \ge 2$, it is normal
(e.g., see \cite[Theorem~23.8]{Matsumura}). We can therefore apply
the Grothendieck-Serre duality for normal projective schemes to get
\[
H^i_\zar(X, \sE(-j)) \cong \H^{-i}_\zar(X, \sE^\vee\otimes^L_{\sO_X}
\omega^{\bullet}_{X/k}(j)).
\]
The desired assertion now follows easily from the Serre vanishing 
theorem. 
\end{proof}

\begin{lem}\label{lem:Vanishing-*}
Assume that $X$ is normal and let $\sE$ be a coherent reflexive sheaf on 
$Y^o$. Then $H^0(Y^o, \sE^{[n]}) = 0$ for $n \gg 0$.
\end{lem}
\begin{proof}
It follows from \corref{cor:Reflexive-PF} that $\tilde{j}_* \sE$
  is a coherent reflexive sheaf on $Y$.
  We denote this extension by $\sE$ itself.
  Using the exact sequences of \lemref{lem:Reflexive} and tensoring them
  with $S^n(C_{Y/X})$ (which is invertible on $Y$)
  and subsequently taking the cohomology, we get exact
  sequences
  \begin{equation}\label{eqn:Vanishing-*-0}
    0 \to H^0(Z, \sE^{[n]}) \to H^0(Z, {\sO^r_Z(q)}^{[n]}) \to
    H^0(Z, \sE'^{[n]}) \to H^1(Z, \sE^{[n]}) \to H^1(Z, {\sO^r_Z(q)}^{[n]})
  \end{equation}
  and
\begin{equation}\label{eqn:Vanishing-*-1}  
  0 \to H^0(Z, \sE'^{[n]}) \to H^0(Z, {\sO^{r'}_Z(q')}^{[n]}),
\end{equation}
where $Z \in \{Y, Y^o\}$.
\corref{cor:Global-Hartogs-0} and \lemref{lem:ESZ-0} together tell us that
$H^0(Z, {\sO^r_Z(q)}^{[n]}) = 0 = H^0(Z, {\sO^{r'}_Z(q')}^{[n]})$ for
all $n \gg 0$.
We conclude that $H^0(Z, \sE^{[n]}) = 0$ and
\begin{equation}\label{eqn:Vanishing-*-2} 
  H^1(Z, \sE^{[n]}) \inj H^1(Z, {\sO^r_Z(q)}^{[n]})
\end{equation}
for all $n \gg 0$.
\end{proof}

\subsection{The Gysin map and Poincar{\'e} duality}\label{sec:Gysin}
We fix a prime-to-$p$ integer $n$ and let $\Lambda = {\Z}/n$ be the
constant sheaf of rings on the {\'e}tale site of $\Sch_k$.
For any integer $m \in \Z$, we let $\Lambda(m)$ denote the usual 
Tate twist of  the sheaf of
$n$-th roots of unity $\mu_n$ on the {\'e}tale site of $\Sch_k$.
For any {\'e}tale sheaf of $\Lambda$-modules $\sF$ on $\Sch_k$, we let
$\sF(m) = \sF \otimes_\Lambda \Lambda(m)$.
Let $D^{+}(X, \Lambda)$ denote the derived category of
the bounded below complexes of the sheaves of $\Lambda$-modules on the
small {\'e}tale site of $X$.

Let $\iota \colon Y \inj X$ be as in \S~\ref{sec:ESZT}. Recall from Gabber's construction \cite{Fujiwara} (see also \cite[Definition~2.1]{Navarro}) that the regular closed embedding $\iota \colon Y \inj X$ induces the Gysin homomorphism $\iota_*: H^i_{\etl}(Y, \Lambda(m)) \to H^{i+2}_{\etl}(X, \Lambda(m+1))$ for any pair of integers $i \ge 0$ and $m \in \Z$. This homomorphism is defined as follows.
It follows from \cite[\S~2]{Deligne} that the line bundle $\sO_X(Y)$ (which we shall write in short as $\sO(Y)$) on $X$ has a canonical class $[\sO(Y)] \in H^1_{Y, \etl}(X, \G_m)$ and its image via the boundary map $H^1_{Y, \etl}(X, \G_m) \to H^2_{Y, \etl}(X, \Lambda)$ is Deligne's localized Chern class $c_1(Y)$. 
Here, $H^*_{Y, \etl}(X, -)$ denotes the {\'e}tale cohomology with support
in $Y$.

At the level of the derived 
category $D^{+}(Y, \Lambda)$, this Chern class is given in terms
of the map $c_1(Y) \colon \Lambda_Y \to \iota^{!}\Lambda_X(1)[2]$.
Using this Chern class, Gabber's Gysin homomorphism 
$\iota_* \colon H^*_{\etl}(Y, \Lambda(m)) \to H^{*+2}_{\etl}(X, \Lambda(m+1))$ is 
the one induced on the cohomology by the composite map 
$\iota_* \colon \iota_*(\Lambda_Y) \to \iota_* \iota^{!}(\Lambda_X(1)[2]) \to
\Lambda_X(1)[2]$ in $D^{+}(X, \Lambda)$, where the second arrow is the
counit of adjunction map.

The local complete intersection closed immersions
$Y \inj X$ and $Y_s \inj X_\sing$ induce a diagram of 
distinguished triangles in $D^{+}(X, \Lambda)$ given by

\begin{equation}\label{eqn:DT}
\xymatrix@C.8pc{
\iota_* \circ \tilde{j}_{!}(\Lambda_{Y^o}) \ar[r]  & \iota_*(\Lambda_Y) \ar[r] 
\ar[d]^-{c_1(Y)} &  \iota_* \circ \tilde{\iota}_*(\Lambda_{Y_s}) 
\ar[d]^-{c_1(Y_s)} \\
j_{!}(\Lambda_{X^o})(1)[2] \ar[r]  & \Lambda_X(1)[2] \ar[r]  &
\iota_*(\Lambda_{X_\sing})(1)[2].}
\end{equation}

The Cartesian square 
\begin{equation}\label{eqn:Chern-0}
\xymatrix@C.8pc{
Y_s \ar[r]^-{\tilde{\iota}} \ar[d] & X_\sing \ar[d] \\
Y \ar[r]^-{\iota} & X}
\end{equation}
and the functoriality of Deligne's localized Chern class imply
(see \cite[Corollary~2.12]{Navarro} or \cite[Proposition~1.1.3]{Fujiwara})
that the right side square in ~\eqref{eqn:DT} is commutative.
It follows that there is a Gysin morphism $c^{j_!}_1(Y^o) \colon
\iota_* \circ \tilde{j}_{!}(\Lambda_{Y^o}) \to j_{!}(\Lambda_{X^o})(1)[2]$
in $D^{+}(X, \Lambda)$ such that the resulting left square in ~\eqref{eqn:DT} 
is commutative.

Applying the cohomology functor on $D^{+}(X, \Lambda)$, we get a commutative 
diagram of long exact sequences
\begin{equation}\label{eqn:Chern-1}
\xymatrix@C.4pc{
H^*_{c, \etl}(Y^o, \Lambda(m)) \ar[r] \ar[d]^-{\iota^o_*} &
H^*_{\etl}(Y, \Lambda(m)) \ar[r] \ar[d]^-{\iota_*} & 
H^*_{\etl}(Y_s, \Lambda(m)) \ar[r] \ar[d]^-{\tilde{\iota}_*} &
H^{*+1}_{c, \etl}(Y^o, \Lambda(m)) \ar[d]^-{\iota^o_*} \\
H^{*+2}_{c, \etl}(X^o, \Lambda(m+1)) \ar[r] &
H^{*+2}_{\etl}(X, \Lambda(m+1)) \ar[r] & H^{*+2}_{\etl}(X_\sing, \Lambda(m+1)) 
\ar[r] & H^{*+3}_{c, \etl}(X^o, \Lambda(m+1)),}
\end{equation}
where $H^*_{c, \etl}(-, \Lambda(m))$ denotes the {\'e}tale cohomology with
compact support (see \cite[Chapter~VI, \S~3]{Milne}) and the
vertical arrows are the Gysin homomorphisms.

\vskip .3cm

Assume now that $k$ is either finite or algebraically closed.
Recall in this case 
(see \cite[Introduction]{JSZ} and \cite[Chapter~VI, \S~11]{Milne},
see also \cite[Corollary~4.2.3]{Deglise}) 
that there is a perfect pairing
\begin{equation}\label{eqn:Lef-et-0}
H^i_{c, \etl}(X^o, \Lambda(m)) \times H^{2d+e-i}_{\etl}(X^o, \Lambda(d-m))
\to H^{2d+e}_{c, \etl}(X^o, \Lambda(d)) \cong \Lambda
\end{equation}
for $m \in \Z$, where $e = 1$ if $k$ is finite and $e = 0$ if $k$ is 
algebraically closed.
This pairing exists even if $X^o$ is not smooth, but may not be perfect
in the latter case. 
It is well known and follows from its construction 
(see the proof of \cite[Theorem~VI.11.1]{Milne} or directly use
\cite[\S~3.3.13 and Remark~4.2.5]{Deglise}) that
~\eqref{eqn:Lef-et-0} is compatible with the closed immersion
$\iota^o \colon Y^o \inj X^o$, i.e., there is a commutative diagram

\begin{equation}\label{eqn:PD-*}
\xymatrix@C.3pc{
H^i_{c, \etl}(Y^o, \Lambda(m-1)) \ar[d]_-{\iota^o_*} & \times &
H^{2d-2+e-i}_{\etl}(Y^o, \Lambda(d-m)) \ar[r] 
&  H^{2d-2+e}_{c, \etl}(Y^o, \Lambda(d-1)) 
\ar[d]^-{\iota^o_*} \\
H^{i+2}_{c, \etl}(X^o, \Lambda(m)) & \times & 
H^{2d-2+e-i}_{\etl}(X^o, \Lambda(d-m)) \ar[u]_-{(\iota^o)^*} \ar[r] &
H^{2d+e}_{c, \etl}(X^o, \Lambda(d)).}
\end{equation}

\subsection{A Lefschetz theorem for {\'e}tale cohomology of $X^o$}\label{LEFET}
Let $\iota \colon Y \inj X$ be as in \S~\ref{sec:ESZT}. We shall now prove a Lefschetz theorem for the {\'e}tale cohomology of $X^o$. This is of independent interest in the study of singular varieties.
In this paper, we shall use it in the proofs of the main results.

\begin{prop}\label{prop:Lef-et}
Assume that $k$ is either finite or algebraically closed
and $X$ is an $R_2$-scheme. Then the pull-back map
$H^i_{\etl}(X^o, \Lambda) \to H^i_{\etl}(Y^o, \Lambda)$ is an isomorphism
for $i \le 1$ and injective for $i = 2$.
\end{prop}
\begin{proof}
Using ~\eqref{eqn:PD-*}, the proposition is reduced to
showing that the Gysin homomorphism
\begin{equation}\label{eqn:Lef-et-1}
\iota^o_* \colon H^i_{c, \etl}(Y^o, \Lambda(d-1)) \to
H^{i+2}_{c, \etl}(X^o, \Lambda(d))
\end{equation}
is an isomorphism for $i \ge 2d+e-3$ and surjective for $i = 2d+e-4$.

Using  the long exact sequences of ~\eqref{eqn:Chern-1} and the known bounds on
the {\'e}tale cohomological dimensions of $X_{\sing}$ and $Y_s$,
the assertion that ~\eqref{eqn:Lef-et-1} is an isomorphism
is equivalent to asserting that the Gysin homomorphism
\begin{equation}\label{eqn:Lef-et-2}
\iota_* \colon H^i_{\etl}(Y, \Lambda(d-1)) \to
H^{i+2}_{\etl}(X, \Lambda(d))
\end{equation}
is an isomorphism for $i \ge 2d+e-3$ and surjective for $i = 2d+e-4$.
For $e = 0$, this follows from \cite[Theorem~2.1]{Skorobogatov}.
We shall prove the $e = 1$ case using a similar strategy as follows.

We let $U = X \setminus Y$. 
The Leray spectral sequence for the inclusion $j' \colon U \inj X$
yields a strongly convergent spectral sequence
\begin{equation}\label{eqn:Lef-et-3}
E^{a,b}_2 = H^a_{\etl}(X, R^bj'_*(\Lambda_U(d))) \Rightarrow
H^{a+b}_{\etl}(U, \Lambda(d)).
\end{equation}
We know that the canonical map
$\Lambda_X(d) \xrightarrow{\cong} j'_*(\Lambda_U(d))$ is an isomorphism. 
Furthermore, the canonical map
$R^bj'_*(\Lambda_U(d)) \to \iota_* \iota^{*}R^bj'_*(\Lambda_U(d))$
is an isomorphism for $b > 0$. We let $\sF_b = \iota^{*}R^bj'_*(\Lambda_U(d))$.
The sheaf exact sequence
\[
0 \to \iota_* \iota^{!}(\Lambda_X(d)) \to \Lambda_X(d) \to
j'_*(\Lambda_U(d)) \to 0
\]
shows that the resulting boundary map $\partial' \colon \sF_1 \to 
\iota^{!}(\Lambda_X(d)[2])$ is an isomorphism.
Moreover, it follows from the construction of the spectral sequence
~\eqref{eqn:Lef-et-3} that the map on the cohomology
groups $H^a_{\etl}(Y, \sF_1) \xrightarrow{\partial'} 
H^{a+2}_{\etl}(X, \Lambda_X(d))$, induced by the composite map
$\iota_*(\sF_1) \to \iota_* \iota^{!}(\Lambda_X(d)[2]) \to \Lambda_X(d)[2]$,
is the map
\[
E^{a,1}_2 = H^a_{\etl}(Y, \sF_1) \xrightarrow{\partial}
H^{a+2}_{\etl}(X, \Lambda_X(d)) = E^{a+2, 0}_2.
\]

It follows that the composition
\[
\H^a_{\etl}(X, \iota_* \iota^{!}(\Lambda_X(d)[2])) 
\xrightarrow{\partial'^{-1}} H^a_{\etl}(Y, \sF_1) \xrightarrow{\partial}
H^{a+2}_{\etl}(X, \Lambda_X(d))
\]
is induced by the canonical adjunction map
$\iota_* \iota^{!}(\Lambda_X(d)[2]) \to \Lambda_X(d)[2]$.
Pre-composing the latter with the Gysin map (see \S~\ref{sec:Gysin})
$\iota_* \Lambda_Y(d-1) \xrightarrow{c_1(Y)} 
\iota_* \iota^{!}(\Lambda_X(d)[2])$, we see that the composition
\begin{equation}\label{eqn:Gysin-inverse}
H^a_{\etl}(Y, \Lambda_Y(d-1)) \xrightarrow{\partial'^{-1} \circ c_1(Y)}
H^a_{\etl}(Y, \sF_1) \xrightarrow{\partial}
H^{a+2}_{\etl}(X, \Lambda_X(d))
\end{equation}
coincides with the Gysin homomorphism $\iota_*$ in ~\eqref{eqn:Chern-1}.

We next study the map $H^a_{\etl}(Y, \sF_1) \xrightarrow{\partial}
H^{a+2}_{\etl}(X, \Lambda_X(d))$.
Using the purity theorem for the closed pair $(X^o, Y^o)$ of smooth schemes
(see \cite[Theorem~VI.5.1]{Milne}), it follows that
\begin{equation}\label{eqn:Lef-et-4}
\sF_b|_{Y^o} \cong \left\{
\begin{array}{ll}
\Lambda_{Y^o}(d) & \mbox{if $b =1$} \\
0 &\mbox{if $b > 1$.}
\end{array}
\right. 
\end{equation}

On the other hand, the affine Lefschetz theorem 
(see \cite[Theorem~VI.7.3(d)]{Milne}) implies that
the sheaf $R^bj'_*(\Lambda_U(d))$ is supported on a closed subscheme
of $Y$ whose dimension is bounded by $d - b$. We conclude that 
\begin{equation}\label{eqn:Lef-et-5}
H^a_{\etl}(X, R^bj'_*(\Lambda_U(d))) = 0 \ \ \mbox{if $b \ge 2$ \ and $a
> {\rm min}\{2q+1, 2d-2b+1\}$}.
\end{equation} 

Using ~\eqref{eqn:Lef-et-3}, ~\eqref{eqn:Lef-et-4} and 
~\eqref{eqn:Lef-et-5}, we get exact sequences
\begin{equation}\label{eqn:Lef-et-6}
H^{i-1}_{\etl}(U, \Lambda_U(d)) \to 
H^{i-2}_{\etl}(Y, \sF_1) \xrightarrow{\partial}
H^{i}_{\etl}(X, \Lambda_X(d)) \to H^{i}_{\etl}(U, \Lambda_U(d));
\end{equation}
\[
H^{2d-3}_{\etl}(Y, \sF_1) \xrightarrow{\partial}
H^{2d-1}_{\etl}(X, \Lambda_X(d)) \to H^{2d-1}_{\etl}(U, \Lambda_U(d))
\]
for $i \ge 2d$.
Since $d \ge 3$ and the {\'e}tale cohomological dimension of $k$ is one,
another application of the affine Lefschetz theorem reduces the above
exact sequences to an isomorphism $H^{i-2}_{\etl}(Y, \sF_1) \xrightarrow{\cong}
H^{i}_{\etl}(X, \Lambda_X(d))$
$i \ge 2d$ and a surjection $H^{2d-3}_{\etl}(Y, \sF_1) \surj
H^{2d-1}_{\etl}(X, \Lambda_X(d))$.

Finally, we consider the long exact sequences
\[
\xymatrix@C.6pc{
H^{i-1}_{\etl}(Y_s, \Lambda_{Y_s}(d-1)) \ar[r] & 
H^{i}_{c, \etl}(Y^o, \Lambda_{Y^o}(d-1)) \ar[r] \ar@{=}[d] &
H^{i}_{\etl}(Y, \Lambda_{Y}(d-1)) \ar[r] \ar[d]& 
H^{i}_{\etl}(Y_s, \Lambda_{Y_s}(d-1)) \\
H^{i-1}_{\etl}(Y_s, \tilde{\iota}^* \sF_1) \ar[r] & 
H^{i}_{c, \etl}(Y^o, \tilde{j}^*\sF_1) \ar[r] &
H^{i}_{\etl}(Y, \sF_1) \ar[r] \ar[d]^-{\partial} & 
H^{i}_{\etl}(Y_s, \tilde{\iota}^* \sF_1) \\
& & 
H^{i+2}_{\etl}(X, \Lambda_X(d)). &}
\]

The left vertical isomorphism on the top is a consequence of  
~\eqref{eqn:Lef-et-4}. Using the bound on the cohomological dimension
of $Y_s$, the above diagram shows that the right vertical arrow is an
isomorphism. Since the composite right vertical arrow is the Gysin
homomorphism as observed before, we conclude that the map $\iota_*$
in ~\eqref{eqn:Lef-et-2} is an isomorphism for $i \ge 2d-2$ and surjection for
$i = 2d-3$. This finishes the proof.
\end{proof}

\subsection{A Lefschetz theorem for {\'e}tale fundamental groups}
\label{sec:LFG**}
Recall that the Lefschetz theorem for the
{\'e}tale fundamental groups of smooth projective schemes over a field
was proven by Grothendieck (see \cite[Expos{\'e}~XII, Corollaire~3.5]{SGA-2}).
However, this is a very challenging
problem for smooth non-projective schemes due
to the presence of ramification when we extend {\'e}tale covers 
to compactifications.
We shall prove the following version of the Lefschetz theorem for the
{\'e}tale fundamental groups of smooth but non-projective schemes.

\begin{thm}\label{thm:Lef-EFG}
Assume that $k$ is either finite or algebraically closed and
$X \subset \P^N_k$ is an $(R_3 + S_4)$-scheme.
Let $Y \subset X$ be a good hypersurface section of degree $m \gg 0$.
Then the induced map
$\iota^o_* \colon \pi^{\ab}_1(Y^o) \to \pi^{\ab}_1(X^o)$ is an 
isomorphism of profinite topological abelian groups.
\end{thm}
\begin{proof}
Since $\pi^{\ab}_1(X^o) \xrightarrow{\cong} {\underset{n \in \Z}\varprojlim}
\ {\pi^{\ab}_1(X^o)}/n$ and the same holds for $Y^o$, it suffices to show that
the map ${\pi^{\ab}_1(Y^o)}/n \to {\pi^{\ab}_1(X^o)}/n$ is an isomorphism
for every integer $n \in \Z$.
Using the Pontryagin duality (e.g., see in the middle of the proof of
\cite[Lemma~1.9, p.~99]{Raskind}) $({\pi^{\ab}_1(Z)}/n)^\vee \cong 
H^1_{\etl}(Z, {\Z}/n)$ for $Z  \in \{X^o, Y^o\}$, 
we are reduced to showing that the pull-back map
\begin{equation}\label{eqn:Lef-EFG-0}
(\iota^o)^* \colon H^1_\etl(X^o, {\Z}/n) \to H^1_\etl(Y^o, {\Z}/n)
\end{equation}
is an isomorphism for all $n \in \Z$.
In view of \propref{prop:Lef-et}, we can assume $n = p^r$, where
${\rm char}(k) = p > 0$.
We shall prove this by induction on $r \ge 1$.

We have an exact sequence of constant {\'e}tale sheaves on $\Sch_k$:
\[
0 \to {\Z}/{p} \to {\Z}/{p^r} \to {\Z}/{p^{r-1}} \to 0
\]
for every $r \ge 2$, where $ {\Z}/{p} \to {\Z}/{p^r}$ is
multiplication by $p^{r-1}$.
This yields a long exact sequence of {\'e}tale cohomology groups
\[
H^0_\etl(X^o, {\Z}/{p^r}) \to H^0_\etl(X^o, {\Z}/{p^{r-1}}) \to
H^1_\etl(X^o, {\Z}/{p}) \to H^1_\etl(X^o, {\Z}/{p^r}) \hspace*{2cm}
\]
\[
\hspace*{7cm}
\to H^1_\etl(X^o, {\Z}/{p^{r-1}}) \to H^2_\etl(X^o, {\Z}/{p}).
\]

Since $X^o$ is integral, the first arrow from left in this exact sequence
is surjective. The same applies to $Y^o$ as well. We thus get a commutative
diagram of exact sequences
\begin{equation}\label{eqn:Lef-EFG-1}
\xymatrix@C.8pc{
0 \ar[r] & H^1_\etl(X^o, {\Z}/{p}) \ar[r] \ar[d] & H^1_\etl(X^o, {\Z}/{p^r}) 
\ar[r] \ar[d] & H^1_\etl(X^o, {\Z}/{p^{r-1}}) \ar[r] \ar[d] & 
H^2_\etl(X^o, {\Z}/{p}) \ar[d] \\
0 \ar[r] & H^1_\etl(Y^o, {\Z}/{p}) \ar[r]  & H^1_\etl(Y^o, {\Z}/{p^r}) 
\ar[r] & H^1_\etl(Y^o, {\Z}/{p^{r-1}}) \ar[r] & H^2_\etl(Y^o, {\Z}/{p}),}
\end{equation}
where the vertical arrows are the pull-back maps.
Using a diagram chase and an induction on $r$, we are reduced to showing that
the pull-back map
\begin{equation}\label{eqn:Lef-EFG-2}
(\iota^o)^* \colon H^i_\etl(X^o, {\Z}/p) \to H^i_\etl(Y^o, {\Z}/p)
\end{equation}
is an isomorphism for $i = 1$ and injective for $i = 2$.

We let $U = X \setminus Y$ so that $U^o = X^o \setminus Y^o$.
We let $u \colon U \inj X$ and $u^o \colon U^o \inj X^o$ denote the
inclusion maps. Then ~\eqref{eqn:Lef-EFG-2} is equivalent to the assertion that
$H^i_{\etl}(X^o, u^o_{!} ({{\Z}/p}|_{U^o}))= 0$ for  $i = 1, 2$.
Using the cohomology exact sequence associated to the
relative Artin-Schreier sheaf exact sequence (where $F$ is the
Frobenius of $\sO_{X^o}$)
\[
0 \to u^o_{!} ({{\Z}/p}|_{U^o}) \to  \sO_{X^o}(-Y^o) 
\xrightarrow{1 - F} \sO_{X^o}(-Y^o) \to 0,
\]
we get the long exact sequence
\[
\cdots \to H^{i-1}_\etl(X^o, \sO_{X^o}(-Y^o)) \to 
H^i_{\etl}(X^o, u^o_{!} ({{\Z}/p}|_{U^o})) \to
H^i_\etl(X^o, \sO_{X^o}(-Y^o)) \hspace*{2cm}
\]
\[
\hspace*{8cm} \xrightarrow{1 - F} 
H^i_\etl(X^o, \sO_{X^o}(-Y^o)) \to \cdots
\]
for $i \ge 1$. Using this, it suffices to show that 
$H^i_\etl(X^o, \sO_{X^o}(-Y^o)) = 0$
for $i \le 2$. Equivalently, it suffices to show that
$H^i_\zar(X^o, \sO_{X^o}(-Y^o)) = 0$ for $i \le 2$.

We now consider the exact sequence of Zariski cohomology groups
\[
\cdots \to H^{i-1}(X^o, \sO_{X^o}(-Y^o)) \to
H^{i}_{X_\sing}(X, \sO_{X}(-Y)) \to 
H^{i}(X, \sO_{X}(-Y)) \hspace*{2cm}
\]
\[
\hspace*{8cm} \to H^{i}(X^o, \sO_{X^o}(-Y^o)) \to \cdots .
\]

Since $X$ is an $(R_3+ S_4)$-scheme, we have
${\underset{x \in X_\sing}{\rm  inf}} \{{\rm depth} (\sO_{X}(-Y)_x)\} \ge 4$.
We conclude from \cite[Theorem~3.8]{Hartshorne-LC} and a
spectral sequence argument that $H^{i}_{X_\sing}(X, \sO_{X}(-Y)) = 0$ for
$i \le 3$. The above long exact sequence then tells us that
the map $H^{i}(X, \sO_{X}(-Y)) \to H^{i}(X^o, \sO_{X^o}(-Y^o))$
of Zariski cohomology groups is an isomorphism for $i \le 2$.
The theorem therefore is finally reduced to showing that
$H^{i}(X, \sO_{X}(-Y)) = 0$ for $i \le 2$. But this follows from
\lemref{lem:ESZ-0} since the degree of the hypersurface $H$ is
very large.
\end{proof}

\begin{remk}\label{remk:Frob-eigenvalue}
If one knew that the eigenvalues of 
the Frobenius on $H^2(X^o, \sO_{X^o}(-Y^o))$ are all 
different from 
one, then the hypothesis of \thmref{thm:Lef-EFG} 
(and hence \thmref{thm:Main-1}) could be weakened.
\end{remk}

\enlargethispage{20pt}

\section{Class field theory and applications}\label{sec:CFT**}
In this section, we shall study the class field theory of singular schemes
over finite fields and its applications. In particular, we shall prove 
Theorems~\ref{thm:Main-1}, ~\ref{thm:Main-2}
and ~\ref{thm:Main-3}.

\subsection{Proof of \thmref{thm:Main-1}}\label{sec:Rec-F}
Let $k$ be a finite field and $X \in \Sch_k$ an integral projective 
$R_1$-scheme of dimension $d \ge 1$ over $k$. It was shown in 
\corref{cor:Rec-map-0-main} that the Frobenius substitution associated to
the regular closed points gives rise to a reciprocity homomorphism
$\phi_X \colon \CH^{LW}_0(X) \to \pi^{\ab}_1(X^o)$.
Also the left vertical arrow of ~\eqref{eqn:Rec-map-2}
gives us the restriction map $\phi_{X}^{0}$. What remains to show is that 
$\phi^0_X \colon \CH^{LW}_0(X)^0 \to  \pi^{\ab}_1(X^o)^0$
is surjective, and it is an isomorphism of finite groups under either of the
conditions (1) and (2) of the theorem.
We shall prove all of these by induction on $d$.
Since $d \le 2$ case already follows from \thmref{thm:CFT-surface},
we shall assume that $d \ge 3$.

We let $\pi \colon X_n \to X$ be the normalization morphism.
Using Corollaries~\ref{cor:Pull-back-nor} and ~\ref{cor:PF-LW},
the Zariski-Nagata purity theorem 
(see \cite[Expos{\'e}~X, Th{\'e}or{\`e}me~3.1]{SGA-1}) for $\pi_1((X_n)^o)$ 
and \cite[Lemma~5.1(1)]{Raskind}, we can assume that $X$ is normal.

{\bf Part~1:} 
We first consider the case when $X$ has isolated singularities.
Let $\alpha \in \CH^{LW}_0(X)$
be a 0-cycle such that $\phi_X(\alpha) = 0$. We now fix an embedding 
$X \inj \P^N_k$ and apply 
\cite[Theorem~3.1]{Wutz} to find a hypersurface
$H \subset \P^N_k$ containing ${\rm Supp}(\alpha)$ such that the scheme 
theoretic intersection $Y = X \cap H$ is smooth and
does not meet $X_\sing$. In particular, $Y \subset X^o$.
We can then find a cycle $\alpha' \in \CH^{F}_0(Y)$
such that $\alpha = \iota_*(\alpha')$, where $\iota \colon Y \inj X$
is the inclusion.
Since $X$ is an $S_2$-scheme of dimension at least three and 
$Y \subset X^o$, it follows from 
\cite[Expos{\'e}~XII, Corollaire~3.5]{SGA-2} that $Y$ is connected
(hence integral) and the map
$\pi_1(Y) \to {\underset{W}\varprojlim} \ \pi_1(W)$ is an isomorphism,
where the limit is taken over all open neighborhoods of $Y$ in $X$.

Since $Y \subset X^o$, there is a factorization
\begin{equation}\label{eqn:Main-1-1}
\pi_1(Y) \to {\underset{U}\varprojlim} \ \pi_1(U) \to
{\underset{W}\varprojlim} \ \pi_1(W),
\end{equation}
where the first limit is over all open neighborhoods of $Y$ contained
in $X^o$. Note that the second arrow is an isomorphism and therefore, so is 
the first arrow. On the
other hand, we showed in the proof of \lemref{prop:LEF*} that
for any open $U \subset X^o$ containing $Y$, the codimension of
$X^o \setminus U$ is at least two. We conclude from the Zariski-Nagata purity
theorem that the first limit in ~\eqref{eqn:Main-1-1} is $\pi_1(X^o)$.
It follows that the map $\pi_1(Y) \to \pi_1(X^o)$ is an isomorphism.
We therefore conclude that there is a commutative diagram
\begin{equation}\label{eqn:Main-1-2}
\xymatrix@C.8pc{
\CH^F_0(Y) \ar[r]^-{\phi_{Y}}
\ar[d]_-{\iota_*} & \pi^{\ab}_1(Y) \ar[d]^-{\iota_*}_-{\cong} \\
\CH^{LW}_0(X) \ar[r]^-{\phi_{X}} & \pi^{\ab}_1(X^o)}
\end{equation}
such that the right vertical arrow is an isomorphism and
$\alpha' \in \CH^F_0(Y)$.

It follows that $\iota_* \circ \phi_Y(\alpha') = 0$.
Since $\phi_Y$ is injective by induction, we get $\alpha' = 0$.
In particular, we get $\alpha = \iota_*(\alpha') = 0$.
This shows that $\phi_X$ is injective.
The surjectivity of $\phi^0_X$ and finiteness of $\CH^{LW}_0(X)^0$ and
$\pi^{\ab}_1(X^o)^0$ also follow from
~\eqref{eqn:Main-1-2} and induction on $d$.
We have thus finished the proof of \thmref{thm:Main-1} when $X$ has
isolated singularities.

{\bf Part~2:} 
We now assume that $X$ is an $(R_3 + S_4)$-scheme. In particular, it is
normal. Moreover, it is regular if $d \le 3$, in which case the theorem is
due to Kato and Saito \cite{Kato-Saito-1}. If $d = 4$, then 
$X$ has isolated singularities in which case we have already proven the
theorem in Part~1. We can therefore assume that $d \ge 5$.
As before, we let $\alpha \in \CH^{LW}_0(X)$
be a 0-cycle such that $\phi_X(\alpha) = 0$.

We can now apply \cite[Theorem~6.3]{Ghosh-Krishna} which say that
for every integer $m \gg 1$ there exists a hypersurface $H \subset \P^N_k$ of degree $m$ containing ${\rm Supp}(\alpha)$
such that the hypersurface section $Y = X \cap H$ satisfies the following.
\begin{enumerate}
\item
$Y \cap X^o = Y^o$ is regular.
\item
$Y$ is an $(R_3 + S_4)$-scheme.
\item
$Y$ contains no irreducible component of $X_\sing$.
\end{enumerate}

Since $X$ is normal and integral, 
it follows from \cite[Expos{\'e}~XII, Corollaire~3.5]{SGA-2}
that $Y$ is connected. Since (2) implies that $Y$ is also
normal, it follows that it must be integral. In particular, it is good
(see the beginning of \S~\ref{sec:LFG}). Moreover, there is a 0-cycle
$\alpha' \in \CH^{LW}_0(Y)$ such that $\alpha = \iota_*(\alpha')$
if we let $\iota \colon Y \inj X$ be the inclusion.

Using \corref{cor:PF-LW} and \cite[Lemma~5.1(1)]{Raskind}, 
we get a commutative diagram
\begin{equation}\label{eqn:Main-1-3}
\xymatrix@C.8pc{
\CH^{LW}_0(Y) \ar[r]^-{\phi_{Y}}
\ar[d]_-{\iota_*} & \pi^{\ab}_1(Y^o) \ar[d]^-{\iota_*}_-{\cong} \\
\CH^{LW}_0(X) \ar[r]^-{\phi_{X}} & \pi^{\ab}_1(X^o).}
\end{equation}
The right vertical arrow in this diagram is an isomorphism by
\thmref{thm:Lef-EFG} since $m \gg 0$. It follows that 
$\phi_Y(\alpha') = 0$. We conclude by induction that $\alpha' = 0$.
The surjectivity of $\phi_X$ and finiteness of $\pi^{\ab}_1(X^o)^0$
also follow from ~\eqref{eqn:Main-1-3}
because $\phi_Y$ is surjective  and $\pi^{\ab}_1(Y^o)^0$ is finite by induction.
This finishes the proof of the theorem except that we still need to show
that $\phi^0_X$ is surjective without condition (1) or (2) of the theorem.

As $d \ge 3$ and $X$ is normal (equivalently, $(R_1 + S_2)$),
we can find a hypersurface section
$Y = X \cap H$ which satisfies conditions (1) and (3) in Part~2 above,
and it is an $(R_1 + S_2)$-scheme. We observed above that 
$Y$ is then an integral normal scheme.
Now, the map $\iota_*: \pi^{\ab}_1(Y^o) \to \pi^{\ab}_1(X^o)$ is surjective
by \propref{prop:LEF*} and \cite[Expos{\'e}~X, Corollaire~2.6]{SGA-2}.
Since the map $\phi_Y$ in ~\eqref{eqn:Main-1-3} is surjective on degree zero
subgroups by induction, we conclude that $\phi^0_X$ is surjective.
$\hfill \square$


\subsection{Proof of \thmref{thm:Main-2}}\label{sec:BF-pf}
Let $X$ be as in \thmref{thm:Main-2}.
We need some preparation before we prove the theorem.

By \cite[Theorem~6.2]{Raskind},
the norm maps $N_x \colon K_0(k(x)) \to K_0(k)$ for $x \in X^o_{(0)}$  induce
a natural homomorphism
$\deg \colon H^d_\nis(X, \sK^M_{d,X}) \to \Z$.
Since these norms are multiplication by the degrees of the field
extensions ${k(x)}/k$, we see that the composition
$\sZ_0(X^o) \xrightarrow{\cyc_X} H^d_\nis(X, \sK^M_{d,X})
\xrightarrow{\deg} \Z$ is the degree homomorphism. 
We let $H^d_\nis(X, \sK^M_{d,X})^0$ be the kernel of this map.
We then get a commutative diagram
\begin{equation}\label{eqn:Rec-map-3}
\xymatrix@C.8pc{
0 \ar[r] & H^d_\nis(X, \sK^M_{d,X})^0 \ar[r] \ar[d]_-{\rho^0_X} &
H^d_\nis(X, \sK^M_{d,X}) \ar[r]^-{\deg} \ar[d]^-{\rho_X} & \Z \ar@{^{(}->}[d] \\
0 \ar[r] & \pi^{\ab}_1(X^o)^0 \ar[r] & \pi^{\ab}_1(X^o) \ar[r] &
\wh{\Z},}
\end{equation}
where the right vertical arrow is the canonical inclusion into
the profinite completion and $\pi^{\ab}_1(X^o) \to \wh{\Z}$ is the push-forward
map induced by the structure map of $X^o$.

Since the last term of the top row of ~\eqref{eqn:Rec-map-3} is torsion free, we have an exact sequence of inverse limits
\begin{equation}\label{eqn:Rec-map-4}
0\to {\underset{m}\varprojlim}\ H^d_\nis(X, \sK^M_{d,X})^0/m \to {\underset{m}\varprojlim}\ {H^d_\nis(X, \sK^M_{d,X})}/m \to \wh{\Z}.
\end{equation}

\begin{lem}\label{lem:Finiteness}
  The map $\rho_X \colon H^d_\nis(X, \sK^M_{d,X}) \to
\pi^{\ab}_1(X^o)$ is injective.
\end{lem}
\begin{proof}
  If $\pi \colon X_n \to X$ denotes the normalization of $X$, then
  the maps $\pi^* \colon H^d_\nis(X, \sK^M_{d,X}) \to H^d_\nis(X_n, \sK^M_{d,X_n})$
  and $\pi^* \colon \pi^{\ab}_1(X^o) \to \pi^{\ab}_1(X^o_n)$ are isomorphisms.
  The first isomorphism holds for dimension reason and exactness of $\pi_*$ on
  Nisnevich sheaves. The second isomorphism holds by the Zariski-Nagata purity
  theorem (see \cite[Expos{\'e}~X, Th{\'e}or{\`e}me~3.1]{SGA-1}).
We can therefore assume that $X$ is normal.
 By ~\eqref{eqn:Rec-map-3},
  it suffices to show that $\rho^0_X$ is injective.

  We assume first that $X$ is geometrically connected. 
  Since $H^d_\nis(X, \sK^M_{d,X})^0$ is a finite group (hence profinite complete) by \cite[Theorem~6.2(1)]{Raskind}
  (take $I = \sO_X$ and $T = X$) and $\pi^{\ab}_1(X^o)$ is profinite complete,
~\eqref{eqn:Rec-map-3} and ~\eqref{eqn:Rec-map-4} gives rise to a commutative diagram of exact sequences
\begin{equation}\label{eqn:CFT-surface-00*}
\xymatrix@C.8pc{
0 \ar[r] & H^d_\nis(X, \sK^M_{d,X})^0 \ar[r] \ar[d]_-{\rho^0_X} &
{\underset{m}\varprojlim}\ {H^d_\nis(X, \sK^M_{d,X})}/m \ar[r]^-{\deg} 
\ar[d]^-{\wh{\rho}_X} & \wh{\Z} \ar@{=}[d] \\
0 \ar[r] & \pi^{\ab}_1(X^o)^0 \ar[r] & \pi^{\ab}_1(X^o) \ar[r] &
\wh{\Z}.}
\end{equation}
The middle vertical arrow is an isomorphism by \cite[Theorem~9.1(3)]{Kato-Saito-2}. It follows that the left vertical
arrow must also be an isomorphism.

Suppose now that $X$ is not geometrically connected. 
Since $\wh{\rho}_X$ is always an isomorphism by \cite[Theorem~9.1(3)]{Kato-Saito-2}, it suffices to show, using ~\eqref{eqn:Rec-map-4}, that $H^d_\nis(X, \sK^M_{d,X})^0$ is profinite complete. To show the latter, it suffices to prove the stronger claim that $H^d_\nis(X, \sK^M_{d,X})^0$ is torsion of bounded exponent. Since the image of $\rho^0_X$ is finite by \thmref{thm:Main-1}, the claim is equivalent to proving that the kernel of $\rho^0_X$ is torsion of bounded exponent. Since $\wh{\rho}_X$ is an isomorphism, it is enough to show that the kernel of the completion map
$\alpha_X \colon
H^d_\nis(X, \sK^M_{d,X}) \to {\underset{m}\varprojlim}\ {H^d_\nis(X, \sK^M_{d,X})}/m$ is torsion of bounded exponent.

Now, we know that there exists a finite field extension ${k'}/k$ such that
$X_{k'}$ is a disjoint union of geometrically connected integral normal
schemes. We pick any irreducible component $Y$ of $X_{k'}$ and let
$\pi \colon Y \to X$ be the projection map. Then $\pi$ is a finite
{\'e}tale morphism of normal schemes whose degree divides $[k':k]$.
Since we have shown that $\alpha_Y$ is injective, it suffices to show
that the kernel of the map
$\pi^* \colon H^d_\nis(X, \sK^M_{d,X})\to  H^d_\nis(Y, \sK^M_{d,Y})$
is torsion of bounded exponent. But this follows directly from
\cite[Lemma~4.5(1)]{Kato-Saito-2}.
\end{proof}

We now prove \thmref{thm:Main-2}. The claim, that the cycle class map
$\cyc_X \colon \sZ_0(X^o) \to H^d_\nis(X, \sK^M_{d,X})$ factors through
$\CH^{LW}_0(X)$, is a direct consequence of
Lemmas~\ref{lem:Factor-rho} and ~\ref{lem:Finiteness}, and
\corref{cor:Rec-map-0-main}.
To resulting map $\cyc_X \colon \CH^{LW}_0(X) \to H^d_\nis(X, \sK^M_{d,X})$
is surjective by \cite[Theorem~2.5]{Kato-Saito-2}. It is injective
because $\phi_X = \rho_X \circ \cyc_X$ is injective by \thmref{thm:Main-1}.
$\hfill \square$

\vskip .3cm

\subsection{Proof of \thmref{thm:Main-3}}\label{sec:BSC-pf}
Let $X$ be as in \thmref{thm:Main-3} and let $f \colon \wt{X} \to X$ be a
resolution of singularities with reduced exceptional divisor $E$.
By \corref{cor:Pull-back-nor}, we can assume that $X$ is normal.
Since $f^* \colon \CH^{LW}_0(X) \to \CH_0(\wt{X}|mE)$ is clearly
surjective for all $m \ge 1$, we only need to show that this map
is injective for all $m \gg 1$. The latter is equivalent to showing that the
map $f^* \colon \CH^{LW}_0(X)^0 \to \CH_0(\wt{X}|mE)^0$
is injective for all $m \gg 1$.

We let $C(X^o) = {\underset{m}\varprojlim} \ \CH_0(\wt{X}|mE)$
and let $C(X^o)^0$ denote the kernel of the degree map $C(X^o) \to \Z$.
It was shown in \cite[Proposition~3.2]{Kerz-Saito-2} that the 
Frobenius substitution associated to closed points in $X^o$ defines a 
reciprocity map $\phi_{X^o} \colon C(X^o) \to \pi^{\ab}_1(X^o)$ such that one has a
commutative diagram
\begin{equation}\label{eqn:BS-fin-0}
\xymatrix@C1pc{
\CH^{LW}_0(X) \ar[r]^{\phi_X} \ar[d]_-{f^*} & \pi^{\rm ab}_1(X^o) \ar@{=}[d] \\
C(X^o) \ar[r]_{\phi_{X^o}} & \pi^{\rm ab}_1(X^o).}
\end{equation}

\enlargethispage{30pt}

If we restrict this diagram to the degree zero subgroups, then
\thmref{thm:Main-1} says that the top horizontal arrow is an isomorphism.
On the other hand, \cite[Theorem~III]{Kerz-Saito-2} 
(if ${\rm char}(k) \neq 2$) and \cite[Theorem~8.5]{BKS} (in general)
say that the bottom horizontal arrow is an isomorphism.
It follows that the map
$f^* \colon \CH^{LW}_0(X)^0 \to C(X^o)^0$ is an isomorphism. 
In particular, $C(X^o)^0$ is finite.

Since $C(X^o) \surj \CH_0(\wt{X}|mE)$ for every $m \ge 1$, it follows that
$C(X^o)^0 \surj \CH_0(\wt{X}|mE)^0$ for every $m \ge 1$. 
We conclude that $\{\CH_0(\wt{X}|mE)^0\}_{m \ge 1}$ is an inverse system of
abelian groups whose transition maps are all surjective and whose
limit $C(X^o)^0$ is finite. But this implies that this inverse system
is eventually constant. That is,
the map $C(X^o)^0 \to \CH_0(\wt{X}|mE)^0$ is an isomorphism for all
$m \gg 1$. It follows that the map
$f^* \colon \CH^{LW}_0(X)^0 \to \CH_0(\wt{X}|mE)^0$ is an isomorphism for
all $m \gg 1$.
$\hfill \square$

\vskip .3cm

\subsection{Necessity of $R_1$-condition}\label{sec:No-R1}
We show by an example that it is necessary to assume the $R_1$-condition
in \thmref{thm:Main-1}.
Let $C$ be the projective plane curve over a finite field 
$k$ which has a simple cusp along
the origin and is regular elsewhere. Its local ring at the singular point
is analytically isomorphic to $k[[t^2, t^3]]$ which is canonically a 
subring of its normalization $k[[t]]$.
Let $\pi \colon \P^1_k \to C$ denote the normalization map. Let $S \cong
\Spec({k[t^2,t^3]}/{(t^2,t^3)})$ denote the reduced conductor and 
$\wt{S} \cong \Spec({k[t]}/{(t^2)})$ its scheme theoretic inverse image
in $\P^1_k$. We have a commutative diagram with exact rows:
\begin{equation}\label{eqn:failure-0}
\xymatrix@C1pc{
0  \ar[r] & {\underset{m}\varprojlim}\ {\sO^{\times}(mS)}/{k^{\times}} 
\ar[r] \ar[d] & {\underset{m}\varprojlim}\ K_0(C, mS) \ar[r] \ar[d]_{\cong} &
\Pic(C) \ar[r] \ar[d]^{\pi^*} & 0 \\
0  \ar[r] & {\underset{m}\varprojlim}\  {\sO^{\times}(m\wt{S})}/{k^{\times}} 
\ar[r] & {\underset{m}\varprojlim}\ K_0(\P^1_k, m\wt{S}) \ar[r]  &
\Pic(\P^1_k) \ar[r] & 0.}
\end{equation}
The isomorphism of the middle vertical map follows from the known result that the double relative $K$-groups $K_0(C, \P^1_k, mS)$ and 
$K_{-1}(C, \P^1_k, mS)$
vanish.

It is easy to check from the $K$-theory localization sequence
that $\Pic(\P^1_k, m\wt{S}) \xrightarrow{\cong} K_0(\P^1_k, m\wt{S})$.
On the other hand, the known class field theory for curves (with modulus)
tells us that there is a canonical isomorphism
${\underset{m}\varprojlim} \ \Pic^0(\P^1_k, m\wt{S}) \xrightarrow{\cong} 
\pi^{\rm ab}(C^o)^0$. It follows that there are isomorphisms
$(1 + tk[[t]])^{\times} \xrightarrow{\cong} \W(k) \xrightarrow{\cong}
\pi^{\rm ab}(C^o)^0$.
On the other hand, $\CH^{LW}_0(C)^0 \cong \Pic^0(C) \cong k$.
This shows that there is no reciprocity map $\CH^{LW}_0(C)^0 \to 
\pi^{\rm ab}(C^o)^0$ and the two can not be isomorphic.

\section{Lefschetz for generalized Albanese variety}
\label{sec:Alb*}
In order to prove the remaining of our main results, we need to use a
Lefschetz hypersurface theorem for the generalized Albanese variety
of smooth quasi-projective schemes over algebraically closed fields.
The goal of this section to establish such a Lefschetz theorem.
 
We assume in this section that $k$ is an algebraically closed field
of characteristic $p > 0$. Recall from \cite{Serre-Alb-1} that 
to any quasi-projective scheme $V$ over $k$, there is associated
a semi-abelian variety $\Alb_S(V)$ over $k$ together with a morphism
$alb_V \colon V \to \Alb_S(V)$ which has the universal property that
given any semi-abelian variety $A$ over $k$ and a morphism $f \colon V \to A$,
there exists a unique affine morphism $\tilde{f} \colon \Alb_S(V) \to A$
such that $f = \tilde{f} \circ alb_V$. Recall here that an affine
morphism between two commutative group schemes over $k$ is the
composition of a group homomorphism with a translation of the target scheme. 

The assignment $V \mapsto \Alb_S(V)$ is a covariant functor for 
arbitrary morphisms of quasi-projective schemes.
If $V$ is smooth and projective, then $\Alb_S(V)$ is the 
Albanese variety in the 
classical sense. If $V$ is a smooth curve, then $\Alb_S(V)$ coincides with 
Rosenlicht's generalized Jacobian \cite{Rosenlicht} or Serre's
generalized Jacobian with modulus \cite{Serre-CFT}.
We shall call $\Alb_S(V)$ `the generalized Albanese variety' of $V$.

For any quasi-projective scheme $V$ over $k$, let $\Alb_W(V)$ denote the
Albanese variety of $V$ which is universal for rational maps from
$V$ to abelian varieties over $k$ (see \cite{Weil} or
\cite[Chapter~II, \S~3]{Lang}).
Let ${\rm Cl}(V)$ denote the divisor class group of $V$ and
${\rm Cl}^0(V)$ the subgroup of ${\rm Cl}(V)$ consisting of Weil
divisors which are algebraically equivalent to zero in the sense of 
\cite[Chapter~19]{Fulton}.
If $V$ is projective and $R_1$, then we recall from \cite{Weil} (see also
\cite[Chapter~IV, \S~4]{Lang}) that there is an abelian variety $\Pic_W(V)$
over $k$, known as the Weil-Picard variety of $V$, 
such that $\Pic_W(V)(k) \cong {\rm Cl}^0(V)$.
Moreover, $\Alb_W(V) \cong \Alb_W(V^o)$ is the dual of $\Pic_W(V)$ 
(see \cite[Chapter~VI, p.~152]{Lang}).
We shall therefore refer to $\Alb_W(V)$ as `the Weil-Albanese variety' of $V$.

\subsection{Generalized Albanese of a smooth variety}
\label{sec:Asmooth}
Let $X \in \Sch_k$ be an integral projective
$R_1$-scheme of dimension $d \ge 1$.
Let $U \subset X^o$ be a nonempty open subscheme and set $Z = X \setminus U$,
endowed with the reduced induced closed subscheme structure.
In this case, Serre gave an explicit description of 
$\Alb_S(U)$ in \cite{Serre-Alb-2}. We recall this description.
We remark here that even if we do not assume $X$ to be smooth, the
exposition of \cite{Serre-Alb-2} remains valid in the present case
with no modification.

Let ${\rm Div}(X)$ denote the free abelian group of Weil divisors on $X$.
Let $\Lambda^1_{U}(X)$ denote the image of the 
push-forward map $\sZ_{d-1}(Z) \to \sZ_{d-1}(X) = {\rm Div}(X)$.
There is thus a canonical homomorphism $\iota_U: \Lambda^1_{U}(X) \to 
\frac{{\rm Cl}(X)}{{\rm Cl}^0(X)} = \ns(X)$, where $\ns(X)$ is the 
N{\'e}ron-Severi group of $X$.
Let $\Lambda_U(X)$ denote the kernel of the canonical map
$\Lambda^1_{U}(X) \xrightarrow{\iota_U} \ns(X)$ so that the quotient
${\rm Div}(X) \surj {\rm Cl}(X)$ induces a homomorphism
$\Lambda_U(X) \to {\rm Cl}^0(X)$.
It was shown by Serre \cite{Serre-Alb-2} that
$\Alb_S(U)$ is the Cartier dual of the 1-motive $[\Lambda_U(X) \to \Pic_W(X)]$
(see \cite{Hodge3} for the definitions of 1-motives and their Cartier duals).

We thus have a canonical exact sequence of algebraic groups
\begin{equation}\label{eqn:Alb-0}
0 \to \Lambda_U(X)^{\vee} \to \Alb_S(U) \to \Alb_W(X) \to 0,
\end{equation}
where $\Lambda_U(X)^{\vee}$ is the Cartier dual of the constant group
scheme over $k$ associated to the lattice $\Lambda_U(X)$. In particular,
$\Lambda_U(X)^{\vee}$ is a split torus of rank equal to the rank of
the lattice $\Lambda_U(X)$.

\subsection{A Lefschetz theorem for $\Alb_S(U)$}\label{sec:Lef-alb}
We shall now prove a Lefschetz theorem for the generalized Albanese variety.
We let $X \subset \P^N_k$ be an integral normal projective scheme of dimension 
$d \ge 3$ over $k$ which is an $R_2$-scheme. Let $U \subset X^o$ be
a nonempty open subscheme. We let $Z = X \setminus U$ with the reduced
closed subscheme structure.
We let $H \subset \P^N_k$ be a 
hypersurface and $Y = X \cap H$ the scheme theoretic intersection.
We shall say that $Y$ is `$Z$-admissible' if the following hold.
\begin{enumerate}
\item
$Y$ is good (see \S~\ref{sec:LFG}).
\item
For every irreducible component $Z'$ of $Z$ of dimension $d-1$, 
the scheme theoretic intersection $Y \cap Z'$ is integral of dimension $d-2$.
\end{enumerate}

Let $\iota \colon Y \inj X$ be the inclusion of a $Z$-admissible
hypersurface section
of $X$. Then the construction of the pull-back map on algebraic cycles
in \cite[Chapter~2, \S~4]{Fulton} yields a homomorphism
$\iota^* \colon {\rm Div}(X) \to {\rm Div}(Y)$. Furthermore, it easily follows
from the proof of \cite[Corollary~2.4.1]{Fulton} that
it induces the pull-back maps
$\iota^* \colon {\rm Cl}(X) \to {\rm Cl}(Y)$ and 
$\iota^* \colon {\rm Cl}^0(X) \to {\rm Cl}^0(X)$. Taking the quotients, we get a
pull-back map $\iota^* \colon \ns(X) \to \ns(Y)$.
We shall follow the notations of \S~\ref{sec:LFG}.

\begin{lem}\label{lem:Lef-NS-tor}
Assume that $X$ is normal, $H$ is a 
hypersurface of degree $m \gg 0$ and $Y = X \cap H$ is good.
Then the map $\iota^* \colon \ns(X)_{\tor} \to \ns(Y)_{\tor}$ is injective.
\end{lem}
\begin{proof}
Since $X$ is normal and $Y$ is good, the latter
is also normal. It follows therefore from 
\cite[Example~10.3.4]{Fulton} that the pull-back maps
$j^* \colon \ns(X) \to \ns(X^o)$ and $\tilde{j}^* \colon \ns(Y) \to 
\ns(Y^o)$ are isomorphisms. Hence, the lemma is equivalent to the statement
that the map $(\iota^o)^* \colon \ns(X^o)_{\tor} \to \ns(Y^o)_{\tor}$ is 
injective.

Since $\Pic_W(X)(k)$ and $\Pic_W(Y)(k)$ are divisible, there is a 
commutative diagram of short exact sequences
\begin{equation}\label{eqn:Lef-NS-tor-0} 
\xymatrix@C.8pc{
0 \ar[r] & {\rm Cl}^0(X)_{\tor} \ar[r] \ar[d] & {\Pic}(X^o)_{\tor} \ar[r] \ar[d] &
\ns(X^o)_{\tor} \ar[r] \ar[d] & 0 \\
0 \ar[r] & {\rm Cl}^0(Y)_{\tor} \ar[r] & {\Pic}(Y^o)_{\tor} \ar[r] &
\ns(Y^o)_{\tor} \ar[r] & 0.}
\end{equation}

Since $Y \subset X$ is a general hypersurface section, it follows from the 
Lefschetz theorem for the Weil-Albanese
variety (see \cite[Chapter~VII, Theorem~5]{Lang}) that the
canonical map $\Alb_W(Y) \to \Alb_W(X)$ is an isogeny of abelian varieties
whose kernel is isomorphic to the finite infinitesimal 
group scheme $\alpha_{p^r}$ for
some $r \ge 0$. Considering the induced map between the dual abelian
varieties, we see that the pull-back morphism $\Pic_W(X) \to \Pic_W(Y)$
is an isogeny of abelian varieties (see \cite[Theorem~11.1]{Cor-Sil}). 
It is then an easy exercise that the map
${\rm Cl}^0(X)_{\tor} \to  {\rm Cl}^0(Y)_{\tor}$ is surjective.
Using a diagram chase in ~\eqref{eqn:Lef-NS-tor-0}, the lemma is now
reduced to showing that the
map ${\Pic}(X^o)_{\tor} \to {\Pic}(Y^o)_{\tor}$ is injective.

We first fix a prime-to-$p$ integer $n$. Since $k$ is algebraically closed,
we can identify $\mu_n$ with ${\Z}/n$.
Since $H^0_{\etl}(X^o, \sO^{\times}_{X^o}) \cong k^{\times}$ (because $X$ is
$R_1$) and the latter is a divisible group,
one observes using the Kummer sequence that $~_n{\Pic}(X^o) \cong 
H^1_{\etl}(X^o, {\Z}/n)$. By the same token, we have
$~_n{\Pic}(Y^o) \cong H^1_{\etl}(Y^o, {\Z}/n)$.
It follows therefore from \propref{prop:Lef-et} that the map
$~_n{\Pic}(X^o) \to ~_n{\Pic}(Y^o)$ is injective. We note here that
this part of \propref{prop:Lef-et} does not require $X$ to be $R_2$.

We now let $n = p^r$ for some $r \ge 1$. Using the short exact sequence
of {\'e}tale sheaves
\begin{equation}\label{eqn:Lef-NS-tor-1}
0 \to  \sO^{\times}_{X^o} \xrightarrow{n} \sO^{\times}_{X^o} \to
{\sO^{\times}_{X^o}}/{n} \to 0,
\end{equation}
we see that $~_n{\Pic}(X^o) \cong H^0_{\etl}(X^o, {\sO^{\times}_{X^o}}/{n})$.
Using the similar isomorphism for $Y^o$, we need to show that
the map $H^0_{\etl}(X^o, {\sO^{\times}_{X^o}}/{n}) \to
H^0_{\etl}(Y^o, {\sO^{\times}_{Y^o}}/{n})$ is injective.
Comparing the exact sequence ~\eqref{eqn:Lef-NS-tor-1} with the similar
sequence for $Y^o$, this injectivity is equivalent to showing that
$H^0_{\etl}(X^o, {\sK_{1, X^o|Y^o}}/n) = 0$, where we let
$\sK_{1, X^o|Y^o} = {\rm Ker}(\sO^{\times}_{X^o} \surj 
(\iota^o)^*(\sO^{\times}_{Y^o}))$. 
Note here that ${\sK_{1, X^o|Y^o}}/n = \Ker({\sK_{1, X^o}}/n \to 
(\iota^o)^*({\sO^{\times}_{Y^o}}/n))$ 
since $\sO^{\times}_{Y^o}$ is $p$-torsion free.

We let $W_r\Omega^\bullet_{X^o}$ be the $p$-typical de Rham-Witt complex of $X^o$
(e.g., see \cite{Illusie}) and let $W_r\Omega^i_{X^o, \log}$ be the image of the 
Bloch-Gabber-Kato homomorphism 
$\dlog \colon \sK^M_{i, X^o} \to W_r\Omega^i_{X^o}$. This map is
given by $\dlog(\{x_1, \ldots , x_i\}) = \dlog[x_1]_r \wedge \cdots \wedge
\dlog[x_i]_r$, where $[\cdot]_r$ denotes the Teichm{\"u}ller homomorphism
$[\cdot]_r \colon \sO^{\times}_{X^o} \to (W_r\sO_{X^o})^{\times}$.
The Bloch-Gabber-Kato homomorphism induces an isomorphism
$\dlog \colon {\sK^M_{i, X^o}}/{n} \xrightarrow{\cong} W_r\Omega^i_{X^o, \log}$.

We let $W_r\Omega^1_{{X^o}/{Y^o}, \log}$ denote the image of $\dlog \colon 
{\sK_{1, X^o|Y^o}}/n \inj W_r\Omega^1_{X^o}$.
It suffices to show that 
$H^0_{\etl}(X^o, W_r\Omega^1_{X^o|Y^o, \log}) = 0$.
Using the short exact sequence (see \cite[Theorem~1.1.6]{JSZ})
\begin{equation}\label{eqn:Lef-NS-tor-2}
0 \to W_{r-1}\Omega^1_{X^o|Y^o, \log} \xrightarrow{\un{p}}
W_r\Omega^1_{X^o|Y^o, \log} \to W_1\Omega^1_{X^o|Y^o, \log} 
\to 0
\end{equation}
and induction on $r$, it suffices to show that
$H^0_{\etl}(X^o, W_1\Omega^1_{X^o|Y^o, \log}) = 0$.

One easily checks that the image of the composite inclusion 
\[
{\sK_{1, X^o|Y^o}}/p \inj \Omega^1_{X^o} \inj \Omega^1_{X^o}(\log Y^o)\]
lies in the $\sO_{X^o}$-submodule 
$\Omega^1_{X^o|Y^o} := \Omega^1_{X^o}(\log Y^o)(-Y^o)$
(see \cite[Theorem~1.2.1]{JSZ}).
Hence, it suffices to show that $H^0_{\etl}(X^o, \Omega^1_{X^o|Y^o}) = 0$.

To show this, we use the exact sequence 
\begin{equation}\label{eqn:Lef-NS-tor-4}
0 \to \Omega^1_{X^o}(-Y^o) \to \Omega^1_{X^o|Y^o} 
\xrightarrow{{\rm Res}} \sO_{Y^o}(-Y^o) \to 0,
\end{equation}
where ${\rm Res}$ is the Poincar{\'e} residue map twisted by $\sO_{X^o}(-Y^o)$.
It suffices therefore to show that the left and the right terms of
the sequence ~\eqref{eqn:Lef-NS-tor-4} have no global sections.
Since ${\rm char}(k) = p > 0$, we are finally reduced to showing that
\begin{equation}\label{eqn:Lef-NS-tor-5}
H^0_{\zar}(Y^o, \sO_{Y^o}(-Y^o)) = H^0_{\zar}(X^o, \Omega^1_{{X^o}/k}(-Y^o)) = 0,
\end{equation}
where we note that these Zariski cohomologies coincide with the
corresponding {\'e}tale cohomologies.

Now, we first note that $\sO_{Y}(Y)$ is very ample on $Y$.
This already implies that $H^0_{\zar}(Y, \sO_{Y}(-Y)) = 0$ (see
\cite[Exercise~III.7.1]{Hartshorne}).
Since $Y$ is normal, we conclude from \corref{cor:Global-Hartogs-0} that
$H^0_{\zar}(Y^o, \sO_{Y^o}(-Y^o)) = 0$.
On the other hand, since $\Omega^1_{{X^o}/k}$ is locally free and
$m \gg 0$, it follows from
\lemref{lem:Vanishing-*} that $H^0_{\zar}(X^o, \Omega^1_{{X^o}/k}(-Y^o)) = 0$.
This concludes the proof of the lemma.
\end{proof}

\begin{prop}\label{prop:Lef-NS}
Assume that $X$ is an $(R_2 + S_2)$-scheme, $H$ is a 
hypersurface of degree $m \gg 0$ and $Y = X \cap H$ is good.
Then the map $\iota^* \colon \ns(X) \to \ns(Y)$ is injective.
\end{prop}
\begin{proof}
Since ${\rm Cl}^0(X)$ is divisible, the map
${\Pic(X^o)}/n \to {\ns(X)}/n$ is an isomorphism for every integer $n \neq 0$.
It follows from \cite[Th{\'e}or{\`e}me~2]{Neron} that $\ns(X)$ is a 
finitely generated abelian group. Hence, there exists a short exact sequence
\begin{equation}\label{eqn:Lef-NS-0}
0 \to \ns(X)_{\tor} \to \ns(X) \to \ns(X)_{\free} \to 0,
\end{equation}
where the first group is finite and the last group is free of finite rank
(called the Weil-Picard rank of $X$). 
We can therefore find a prime number
$\ell \neq p$ such that the map ${\ns(X)}/{\ell^r} \to {\ns(X)_\free}/{\ell^r}$
is an isomorphism for all $r \ge 1$. It follows now from the Kummer sequence
that there is a series of homomorphisms
\[
{\underset{r \ge 1}\varprojlim} \ {\Pic(X^o)}/{\ell^r} \xrightarrow{\cong} 
{\underset{r \ge 1}\varprojlim} \ {\ns(X)}/{\ell^r}
\xrightarrow{\cong} {\underset{r \ge 1}\varprojlim} \ {\ns(X)_\free}/{\ell^r} 
\inj {\underset{r \ge 1}\varprojlim} \ H^2_{\etl}(X^o, {\Z}/{\ell^r}) 
\cong H^2_{\etl}(X^o, \Z_\ell).
\]
Comparing with the similar maps for $Y$, we obtain
a commutative diagram
\begin{equation}\label{eqn:Lef-NS-1}
\xymatrix@C.8pc{
\ns(X)_\free \ar@{^{(}->}[r] \ar[d]_-{\iota^*} &
\wh{\ns(X)}_{\ell} \ar@{^{(}->}[r] \ar[d]^-{\iota^*} & H^2_{\etl}(X^o, \Z_\ell) 
\ar[d]^-{(\iota^o)^*} \\
\ns(Y)_\free \ar@{^{(}->}[r] &
\wh{\ns(Y)}_{\ell} \ar@{^{(}->}[r] & H^2_{\etl}(Y^o, \Z_\ell),}
\end{equation}
where $\wh{A}_{\ell}$ denote the $\ell$-adic completion of an abelian group
$A$.

Using \lemref{lem:Lef-NS-tor}, the exact sequence
~\eqref{eqn:Lef-NS-0} and the diagram ~\eqref{eqn:Lef-NS-1}, 
we reduce the proposition to showing that
the pull-back map
\begin{equation}\label{eqn:Lef-NS-2}
H^2_{\etl}(X^o, {\Z}/{\ell^r}) \to H^2_{\etl}(Y^o, {\Z}/{\ell^r})
\end{equation}
is injective for all $r \ge 1$. But this follows from 
\propref{prop:Lef-et}.
\end{proof}

We can now prove our Lefschetz theorem for the generalized
Albanese variety.

\begin{thm}\label{thm:Lef-Alb-main}
  Let $X \subset \P^N_k$ be an integral projective scheme of dimension $d \ge 3$
  over an algebraically closed field $k$ of characteristic $p > 0$. Let
  $U \subset X^o$ be a dense open subscheme and $Z = X \setminus U$.
  Assume that $X$ is an $(R_2 + S_2)$-scheme and $H \subset \P^N_k$ is a 
hypersurface of degree $m \gg 0$ such that $Y = X \cap H$ is $Z$-admissible.
Then the map $\Alb_S(U \cap Y)(k) \to \Alb_S(U)(k)$ is an isomorphism.   
\end{thm}
\begin{proof}
We let $V = U \cap Y$ and consider the commutative diagram of the short
exact sequences of abelian groups (see ~\eqref{eqn:Alb-0})
\begin{equation}\label{eqn:Lef-Alb-main-0}
\xymatrix@C.8pc{
0 \ar[r] & \Lambda_V(Y)^\vee(k) \ar[r] \ar[d]_-{\alpha} & 
\Alb_S(V)(k) \ar[r] \ar[d]^-{\beta} & \Alb_W(Y)(k) \ar[d]^-{\iota_*} \ar[r] & 
0 \\
0 \ar[r] & \Lambda_U(X)^\vee(k) \ar[r] & 
\Alb_S(U)(k) \ar[r]  & \Alb_W(X)(k) \ar[r] & 0,}
\end{equation}
where the vertical arrows are the canonical maps induced by the
inclusion $Y \inj X$. We have seen in the proof of \lemref{lem:Lef-NS-tor}
that the right vertical arrow $\iota_*$ in ~\eqref{eqn:Lef-Alb-main-0} is an
isomorphism. So we need to show that $\alpha$ is an isomorphism to prove the
theorem. 

Since $Y$ is $Z$-admissible, we see that the homomorphism
$\iota^* \colon \Lambda^1_U(X) \to \Lambda^1_V(Y)$ is bijective. But this 
implies by virtue of \propref{prop:Lef-NS} that the homomorphism
$\iota^* \colon \Lambda_U(X) \to \Lambda_V(Y)$ is also bijective.
Taking the Cartier duals of these groups, we conclude that $\alpha$ is an
isomorphism.
\end{proof}

\section{The Suslin homology}
\label{sec:LWS}
The goal of this section is to prove \thmref{thm:Main-4} which identifies
the Levine-Weibel Chow group of a projective $R_1$-scheme over an
algebraically closed field with the Suslin 
homology of its regular locus. We begin by recalling the definition of
Suslin homology of smooth schemes.

\subsection{Recollection  of Suslin homology}\label{sec:SH}
Let $k$ be any field. Let $\Delta_i$ denote the
algebraic $i$-simplex, i.e., the spectrum of the ring
${k[x_0, \ldots , x_i]}/{(x_0 + \cdots + x_i -1)}$.
Let $X \in \Sch_k$.
Recall from \cite{Suslin-Voev} that the Suslin homology $H^{S}_i(X, A)$
of $X$ with coefficients in an abelian group $A$
is defined to be the $i$-th homology 
of the complex $(C_*(X) \otimes_{\Z} A, \partial)$, where
$C_i(X)$ is the free abelian group on the set of
integral closed subschemes of $X \times \Delta_i$ which are
finite and surjective over $\Delta_i$. The boundary
map is given by the alternating sum:
\[
\partial = \sum_{j=0}^{i} (-1)^j \delta^*_j \colon \ C_i(X) \to C_{i-1}(X),
\]
where $\delta^*_j$ is the pull-back map between the cycle groups
induced by the inclusion $\delta_j \colon X \times \Delta_{i-1}
\inj X \times \Delta_{i}$, given by $x_j = 0$. Note that the finiteness
and surjectivity conditions on the cycles over $\Delta_i$ insure that this
pull-back is defined.

As explained in \cite{Suslin-Voev}, $H^S_*(X, A)$ is an algebraic
analogue of the singular homology of topological spaces.
We shall write $H^{S}_*(X, \Z)$ in short as
$H^{S}_*(X)$. One easily checks from the definition that 
$H^S_*(-, A)$ is a covariant functor on $\Sch_k$.
By \cite[Proposition~14.18]{MVW} and \cite[Chapter~4, \S~9]{FSV}, the Suslin
homology is also a part of the motivic homology and cohomology theories of 
algebraic varieties in the sense of $\A^1$-homotopy theory.

It is easy to see from the definition that the identity map
$C_0(X) \to \sZ_0(X)$ induces a surjective homomorphism
$H^{S}_0(X) \surj \CH^F_0(X)$. This is an isomorphism if $X$ is complete.
Otherwise, $H^{S}_0(X)$ carries more information about $X$ than its
Chow group.
We shall be interested in the group $H^{S}_0(X, A)$.
In this case, the universal coefficient theorem implies that
there is a functorial isomorphism ${H^{S}_0(X)}/{n} \xrightarrow{\cong}
H^{S}_0(X, {\Z}/n)$ for any integer $n \in \Z$.

In this paper, we shall use the following description of $H^S_0(X)$
due to Schmidt (\cite[Theorem~5.1]{Schmidt}).
Assume that $X$ is a reduced scheme which is dense open in a projective
scheme $\ov{X}$. Let $\nu \colon C \to \ov{X}$ be a finite morphism from
a regular projective integral curve whose image is not contained in
$\ov{X} \setminus X$. Let $f \in k(C)^{\times}$ be such that it is regular
in a neighborhood of $\nu^{-1}(\ov{X} \setminus X)$ and 
$f(x) = 1$ for every $x \in \nu^{-1}(\ov{X} \setminus X)$.
Then the identity map $C_0(X) \to \sZ_0(X)$ induces an isomorphism
between $H^{S}_0(X)$ and the quotient of $\sZ_0(X)$ by the subgroup generated by
$\nu_*(\divf(f))$, where the $(C,f)$ runs through the collection
of all curves $C$ and $f \in k(C)^{\times}$ as the above. We shall let
$\sR^{S}_0(X)$ denote this subgroup.

\subsection{Chow group with modulus and Suslin homology}
\label{sec:Modulus}
One of the key steps for proving \thmref{thm:Main-4} is to 
show that the Suslin homology coincides with the Chow group of 0-cycles
with modulus (see \S~\ref{sec:Res-surface} for the definition of the
latter)
in certain cases. We shall prove this result of
independent interest in this subsection.
We expect this to have many applications in the theory of 0-cycles
with modulus.

Let $X$ be a regular projective scheme over a field $k$ and 
$D \subset X$ an effective Cartier divisor.
It is then an easy exercise using Schmidt's description of Suslin homology that 
the identity map of $\sZ_0(X \setminus D)$
induces a surjection $\CH_0(X|D) \surj H^S_0(X \setminus D)$.
We let $\Lambda$ be ${\Z}[\tfrac{1}{p}]$ if ${\rm char}(k) = p > 0$
and ${\Z}/n$, where $n$ is any nonzero integer, if ${\rm char}(k) = 0$. 
The following result was obtained by the second author in a joint work with
F. Binda \cite{BK-snc}. Since the paper is not yet published, we present
a proof.

\begin{prop}\label{prop:Suslin-modulus}
Let $X$ be a regular projective scheme over a field $k$ and
$D \subset X$ a reduced effective Cartier divisor whose all
irreducible components are regular.
Then ${\CH_0(X|D)}_\Lambda \surj {H^S_0(X \setminus D)}_\Lambda$
is an isomorphism.
\end{prop}
\begin{proof}
We can assume that $X$ is connected.
We let $U = X \setminus D$. We need to show that $\sR^{S}_0(U)$
dies in $\CH_0(X|D)_\Lambda$. 
So we let $\nu \colon C \to X$ be a finite morphism from a regular
integral projective curve whose image is not contained in $D$.
We let $E = \nu^{-1}(D)$ and let $f \in \sO^{\times}_{C,E}$ be such that
either $E = \emptyset$ or $f(x) = 1$ for all $x \in E$.
Our assertion is immediate if $E = \emptyset$ and we therefore 
assume that this is not the case.

Since $\nu$ is a finite morphism of regular schemes, we can find a
factorization $C \xrightarrow{\nu'} \P^n_X \xrightarrow{\pi} X$ of $\nu$ 
such that
the first map is a closed immersion and the second map is the
canonical projection. Since $\pi$ is smooth, it follows that
$\pi^*(D)$ is reduced with regular irreducible components.
There is a push-forward map $\pi_* \colon \sZ_0(\P^n_U) \to
\sZ_0(U)$ such that $\pi_*(\sR_0(\P^n_X|\P^n_D)) \subset
\sR_0(X|D)$ (see \cite{Binda-Saito} or \cite[\S~2]{KP-MRL}). Since
$\nu_*(\divf(f)) = \pi_* \circ \nu'_*(\divf(f))$,
it suffices to show that $\nu'_*(\divf(f))$ dies in $\CH_0(\P^n_X|\P^n_D)_\Lambda$.
We can therefore assume that $\nu \colon C \inj X$ is a closed immersion.

Since $D$ is reduced with regular irreducible components, we can
apply \cite[Proposition~A.6]{Saito-Sato} to find a finite sequence of
blow-ups $\pi \colon X' \to X$ along the closed points lying over $D$
such that the scheme theoretic inverse image
$D' := X' \times_X D$ satisfies the following.

\begin{enumerate}
\item
The irreducible components of $D'_\red$ are regular.
\item 
The strict transform $C'$ of $C$ is regular.
\item
$C'$ intersects $D'_\red$ only in the regular locus of $D'_\red$ and transversely.
\end{enumerate}

Since $\pi$ is proper, we have a commutative diagram
\begin{equation}\label{eqn:Suslin-modulus-0}
\xymatrix@C.8pc{
\sZ_0(X'\setminus D') \ar[d]_-{\pi_*} \ar[r] & \CH_0(X'|D') \ar[d]^-{\pi_*} \\
\sZ_0(X\setminus D) \ar[r] & \CH_0(X|D),}
\end{equation}
where $\pi_*$ is the push-forward map between the 0-cycle groups.
Since $C$ is regular, the map $\pi \colon C' \to C$ is an isomorphism
and hence $f \in k(C')^{\times}$ such that
$\divf(f)_C = \pi_*(\divf(f)_{C'})$. Moreover, $f$ is a regular invertible
function in a neighborhood of $D' \cap C'$ with $f(x) = 1$ for every
$x \in D' \cap C'$.

It follows from ~\eqref{eqn:Suslin-modulus-0} that $\divf(f)_C$ will
die in ${\CH_0(X|D)}_\Lambda$ if we can show that $\divf(f)_{C'}$
  dies in ${\CH_0(X'|D')}_\Lambda$. Equivalently, $\divf(f)_{C'}$
  dies in ${\CH_0(X'|D'_\red)}_\Lambda$ by \cite[Theorem~1.3]{Miyazaki}.
We can therefore assume that our original curve $C \subset X$ has the
property that it is regular and it intersects $D$ transversely in
the regular locus of $D$. But in this case, it is easy to see that
$f \in \sO^{\times}_{C,E}$ and $f(x) = 1$ for all $x \in E$ if and only if
$f \in {\rm Ker}(\sO^{\times}_{C,E} \to \sO^{\times}_E)$.
This concludes the proof.
\end{proof}

\subsection{Relation with Levine-Weibel Chow group}
\label{sec:LWS*}
Let $k$ be any field and 
$X$ an integral projective $R_1$-scheme of dimension $d \ge 1$ over $k$.
We first define a canonical homomorphism from the Levine-Weibel
Chow group of $X$ to the Suslin homology of $X^o$.

\begin{lem}\label{lem:Chow-Suslin}
There is an inclusion of subgroups $\sR^{LW}_0(X) \subseteq
\sR^{S}_0(X^o)$ inside $\sZ_0(X^o)$. 
In other words, the identity map of $\sZ_0(X^o)$ defines a
canonical surjection 
\[
\theta_X \colon \CH^{LW}_0(X) \surj H^{S}_0(X^o).
\]
\end{lem}
\begin{proof}
By \lemref{lem:Moving-nor}, we can replace $\sR^{LW}_0(X)$ by 
$\sR^{LW}_0(X, X_\sing)$. We now let $C \subset X$ be an integral curve 
with $C \cap X_{\rm sing} = \emptyset$
and let $f \in k(C)^{\times}$. Since $C$ is closed in the projective 
scheme $X$ which does not meet $X_{\rm sing}$, it is clear that the pair 
$(C_n, f)$ defines a relation in $\sR^{S}_0(X^o)$ according to
Schmidt's description of $H^S_0(X^o)$.
\end{proof}

\enlargethispage{20pt}

\begin{lem}\label{lem:Surface}
Assume that ${\rm char}(k) = p > 0$ and $d =2$. 
Then the kernel of $\CH^{LW}_0(X) \surj H^{S}_0(X^o)$ is a $p$-primary
torsion group of bounded exponent. 
\end{lem}
\begin{proof}
Let $\pi \colon \wt{X} \to X$ be a resolution of singularities of 
$X$ such that the reduced exceptional divisor $E \subset \wt{X}$ has
strict normal crossings (see \cite{Lipman} for the existence of
$\wt{X}$). For an integer $m \ge 1$, let 
$mE \inj \wt{X}$ denote the infinitesimal thickening of $E$ in $\wt{X}$
of order $m$. 

It is clear from the definitions of $\CH^{LW}_0(X)$,
$\CH_0(\wt{X}|D)$ and $H^S_0(X^o)$ that the identity map of
$\sZ_0(X^o)$ defines, by the pull-back via $\pi^*$,
the canonical surjective maps
\begin{equation}\label{eqn:Surface-00}
\CH^{LW}_0(X) \stackrel{\pi^*}{\surj} \CH_0(\wt{X}|mE) \surj 
\CH_0(\wt{X}|E) \surj H^S_0(X^o)
\end{equation}
for every integer $m \ge 1$ such that the composite map
is $\theta_X$.
The first arrow from the left is an isomorphism for all $m \gg 1$
by \thmref{thm:Main-3} and the third arrow is an isomorphism
after inverting $p$ by \propref{prop:Suslin-modulus}.
We thus have to show that the kernel of the
surjection $\CH_0(\wt{X}|mE) \surj \CH_0(\wt{X}|E)$ is a $p$-group
of bounded exponent{\footnote{If we do not insist on bounded exponent, then
we can directly apply \cite[Theorem~1.3]{Miyazaki}.}} if $m \gg 1$.
But this follows by comparing the map (see ~\eqref{eqn:0-C-map-I-0})
$cyc_{\wt{X}|mE} \colon \CH_0(\wt{X}|mE) \to F^2K_0(\wt{X}, mE)$ with 
  $cyc_{\wt{X}|E} \colon \CH_0(\wt{X}|E) \to F^2K_0(\wt{X}, E)$ for $m \gg 1$, and
applying \propref{prop:RES-Norm} in combination with \cite[Lemma~3.4]{Krishna-Inv}.
\end{proof}

We shall now generalize \lemref{lem:Surface} to arbitrary dimension.

\begin{thm}\label{thm:p-torsion}
Let $k$ be a perfect field of characteristic $p > 0$
and $X$ an integral projective $R_1$-scheme of dimension $d \ge 2$ over $k$.
Then the kernel of the canonical surjection $\CH^{LW}_0(X) \surj 
H^S_0(X^o)$ is a $p$-primary torsion group. Equivalently, the map
\[
\theta_X \colon {\CH^{LW}_0(X)}[\tfrac{1}{p}] \to {H^S_0(X^o)}[\tfrac{1}{p}]
\]
is an isomorphism.
\end{thm}
\begin{proof}
We shall prove the theorem by induction on $d$. The case
$d = 2$ follows from \lemref{lem:Surface}. So we can assume $d \ge 3$.
Let $\nu \colon C \to X$ be a finite morphism from a regular integral
projective curve whose image is not contained in $X_{\rm sing}$
and let $f \in {\rm Ker}(\sO^{\times}_{C,E} \surj \sO^{\times}_{E})$,
where $E = \nu^{-1}(X_{\rm sing})$ with the reduced closed subscheme
structure. We need to show that $\nu_*(\divf(f)) \in
\CH^{LW}_0(X)$ is killed by a power of $p$.

We can get a factorization $C \stackrel{\nu'}{\inj} 
\P^n_X \xrightarrow{\pi} X$ of $\nu$, where $\nu'$ is a closed
immersion and $\pi$ is the canonical projection. Since
the singular locus of $\P^n_X$ coincides with $\P^n_{X_{\rm sing}}$, it is
clear that $\nu'_*(\divf(f)) \in \sR^{S}_0(\P^n_{X^o})$ and
$\nu_*(\divf(f)) = \pi_*(\nu'_*(\divf(f)))$.
Using \corref{cor:PF-LW}, it suffices therefore to show that 
$\nu'_*(\divf(f))$ is killed by some power of $p$ in $\CH^{LW}_0(\P^n_X)$.
We can thus assume that $\nu \colon C \to X$ is a closed immersion.

We now fix a closed embedding $X \inj \P^N_k$.
We let $C' = C \cap X^o$. Since $X^o$ and $C'$ are smooth (this uses
perfectness of $k$) and
$d \ge 3$, we can use \cite[Theorem~7]{KL} (for $k$ infinite)
and \cite[Theorem~3.1]{Wutz} (for $k$ finite) to find a hypersurface 
$H \subset \P^N_k$ containing $C$ and not containing $X$ 
such that the scheme theoretic intersection $X \cap H$
has the property that it contains
no irreducible component of $X_\sing$ and $H \cap X^o$ is smooth.

We let $W$ be the connected component of $H \cap X^o$ which contains $C'$
and $Y \subset X$ the closure of $W$ with the reduced closed
subscheme structure. Then $Y \subset X$ is an integral closed subscheme
of dimension $d-1$ containing $C$ which satisfies the following properties.

\begin{enumerate}
\item
$Y \cap X^o$ is smooth.
\item
$\dim(Y \cap X_\sing) \le \dim(X_\sing) - 1 \le d-3 = \dim(Y) - 2$. 
In particular, $Y$ is an $R_1$-scheme.
\end{enumerate}

We let $A = Y \cap X_\sing$ and $U = Y \cap X^o$ so that 
$Y_\sing \subset A$ and $U \subset Y^o$. 
Let $C \stackrel{\nu'}{\inj} Y \stackrel{\iota}{\inj} X$ 
be the factorization of $\nu$.
It follows from the choice of $Y$ that 
$\nu'_*(\divf(f)) \in \sR^{S}_0(U) \subset \sR^{S}_0(Y^o)$.
On the other hand, we have a commutative diagram
\begin{equation}\label{eqn:p-torsion-0}
\xymatrix@C.8pc{
\CH^{LW}_0(Y, A) \ar[r]^-{\cong} \ar@{->>}[d] & \CH^{LW}_0(Y) \ar@{->>}[d] \\
H^S_0(U) \ar[r] & H^S_0(Y^o),}
\end{equation}
where the existence of the left vertical arrow
follows directly from the proof of \lemref{lem:Chow-Suslin} 
and the top horizontal arrow is an isomorphism by \lemref{lem:Moving-nor}.

Since $\nu'_*(\divf(f)) \in \CH^{LW}_0(Y, A)$, it follows
from the above diagram and by induction that
$\nu'_*(\divf(f))$ is killed by a power of $p$ in $\CH^{LW}_0(Y)$.
Equivalently,
$\nu'_*(\divf(f))$ is killed by a power of $p$  in $\CH^{LW}_0(Y, A)$.
The push-forward map $\iota_* \colon \sZ_0(Y \setminus A) \to \sZ_0(X^o)$
and \corref{cor:PF-LW} together imply that
$\nu_*(\divf(f)) = \iota_*\circ \nu'_*(\divf(f))$ is killed by a power of $p$ 
in $\CH^{LW}_0(X)$. This concludes the proof.
\end{proof}

\subsection{The Albanese homomorphism}\label{sec:Alb-hom}
We now assume that $k$ is an algebraically closed field.
Let $X$ be a connected smooth quasi-projective scheme of dimension $d \ge 1$ 
over $k$. The covariance of Suslin homology
defines the push-forward map ${\deg}_X \colon H^S_0(X) \to H^S_0(k) \cong \Z$. 
This is also called the degree map since $\deg_X ([x]) = [k(x): k]$ for a 
closed point $x \in X$. We let $H^S_0(X)^0 = \Ker(\deg_X)$.
Let $\Alb_S(X)$ be the generalized
Albanese variety of $X$ (see \S~\ref{sec:Alb*}).

Let $\vartheta_X \colon \sZ_0(X)^0 \to \Alb_S(X)(k)$ be given by
$\vartheta_X(\sum_i n_i [x_i]) = \sum_i n_i (alb_X(x_i))$.
It was shown in \cite[Lemma~3.1]{SS} that this map factors through
the quotient by $\sR^S_0(X)$ to yield the Albanese homomorphism
\begin{equation}\label{eqn:alb-hom}
\vartheta_X \colon H^S_0(X)^0 \to \Alb_S(X)(k).
\end{equation}
Furthermore, $\vartheta_X$ defines a natural transformation of covariant
functors from $\Sm_k$ to abelian groups as $X$ varies.
The map $\vartheta_X$ was in fact discovered by Ramachandran 
\cite{Ramachandran} who showed more generally that there exists an Albanese 
group scheme ${\sA lb}_S(X)$ and an Albanese homomorphism 
$\vartheta_X \colon H^S_0(X) \to {\sA lb}_S(X)(k)$
such that $\Alb_S(X)$ is the identity component of ${\sA lb}_S(X)$
and ~\eqref{eqn:alb-hom} is the induced map on the degree zero part.
If $X$ is projective over $k$, then $\vartheta_X$ coincides with the 
classical Albanese homomorphism from the degree zero Chow group of 0-cycles
on $X$.

Suppose now that $X \in \Sch_k$ is an integral projective $R_1$-scheme of 
dimension $d \ge 1$.
Recall from \S~\ref{sec:Alb*} that the universal rational map
$alb^w_X \colon X_n \dashrightarrow \Alb_W(X)$ extends to a regular morphism
$alb^w_X \colon X^o \to \Alb_W(X)$. Moreover, the universal property of 
$\Alb_S(X^o)$ shows that this map is the composition
$X^o \xrightarrow{alb_{X^o}} \Alb_S(X^o) \to \Alb_W(X)$.

Recall from \cite[\S~7]{Krishna-Srinivas} that if $X$ is normal, then $alb^w_{X}
\colon X^o \to \Alb_W(X)$ gives rise to the Albanese homomorphism
$\alpha_{X} \colon \CH^{LW}_0(X)^0 \to \Alb_W(X)(k)$.
The construction of this homomorphism is identical to that of $\vartheta_{X^o}$.
If $X$ is smooth, this is the classical Albanese homomorphism for
$\CH^F_0(X)^0$. The main result of \cite[\S~7]{Krishna-Srinivas} is that
$\alpha_X$ is an isomorphism between the torsion subgroups,
extending the famous Bloch-Roitman-Milne torsion theorem for smooth 
projective schemes. Using \corref{cor:Pull-back-nor} and
birational invariance of $\Alb_W(X)$, this results immediately extends
to $R_1$-schemes, i.e.,

\begin{prop}\label{prop:Non-normal-tor}
Let $X \in \Sch_k$ be an integral projective $R_1$-scheme of dimension 
$d \ge 1$. Then the Albanese map
$alb^w_X \colon X^o \to \Alb_W(X)$ induces a homomorphism
\[
\alpha_X \colon \CH^{LW}_0(X)^0 \to \Alb_W(X)(k)
\]
which is an isomorphism on the torsion subgroups.
\end{prop}

\subsection{Proof of \thmref{thm:Main-4}}\label{sec:Main-4-pf}
We shall now prove \thmref{thm:Main-4}. 
We let $k$ be an algebraically closed field and $X \in \Sch_k$ an integral
projective $R_1$-scheme of dimension $d \ge 1$. 
We assume first that ${\rm char}(k) = p > 0$.
We need to show in this case that the map 
$\theta_X \colon \CH^{LW}_0(X) \to H^S_0(X^o)$ is an isomorphism.

By \thmref{thm:p-torsion}, we only have to show that $\Ker(\theta_X)\{p\} = 0$.
For this, we consider the diagram
\begin{equation}\label{eqn:Suslin-LW-iso-0}
\xymatrix@C.8pc{
\CH^{LW}_0(X)^0 \ar[r]^-{\theta_X} \ar[d]_-{\alpha_X} & H^S_0(X^o)^0 
\ar[d]^-{\vartheta_{X^o}} \\
\Alb_W(X)(k) & \Alb_S(X^o)(k) \ar[l]_-{\cong},}
\end{equation}
where the bottom horizontal arrow is an isomorphism by ~\eqref{eqn:Alb-0}.
It follows from the construction of various maps that this diagram is
commutative.  

  If $a \in \Ker(\theta_X)\{p\}$, then it must lie in $\CH^{LW}_0(X)^0$.
Moreover, its image under $\theta_X$ will die in $H^S_0(X^o)^0$.
This in turn implies by ~\eqref{eqn:Suslin-LW-iso-0} that
$\alpha_X(a) = 0$. \propref{prop:Non-normal-tor} implies that
$a = 0$.

We now assume that ${\rm char}(k) = 0$. In this case, we have to show that the 
map $\theta_X \colon {\CH^{LW}_0(X)}/n \to {H^S_0(X^o)}/n$ is an isomorphism
for all integers $n \neq 0$.

We let $\pi \colon X' \to X$ be the normalization map. We let
$A = \pi^{-1}(X_\sing)$ and $U = X' \setminus A$. We then have a commutative
diagram
\begin{equation}\label{eqn:Main-4-0}
\xymatrix@C.8pc{
{\CH^{LW}_0(X')}/n \ar@{->>}[d]_-{\theta_{X'}} & 
{\CH^{LW}_0(X', A)}/n \ar[r]^-{\pi_*}_-{\cong} 
\ar[l]^-{\cong} \ar@{->>}[d]^-{\theta_{X'}} & {\CH^{LW}_0(X)}/n 
\ar@{->>}[d]^-{\theta_{X}} \\
{H^S_0(X'^o)}/n & {H^S_0(U)}/n \ar[r]^-{\pi_*}_-{\cong}  \ar[l] &  
{H^S_0(X^o)}/n.}
\end{equation}

The top horizontal arrows are isomorphisms by \lemref{lem:Moving-nor}
and \corref{cor:Pull-back-nor}.
Suppose that the left vertical arrow in this diagram is an isomorphism.
Then the middle and the right vertical arrows also are isomorphisms.
We can therefore assume that $X$ is normal.

Let $\pi \colon \wt{X} \to X$ be a resolution of singularities of 
$X$ such that the reduced exceptional divisor $E \subset \wt{X}$ has
strict normal crossings. As in ~\eqref{eqn:Surface-0}, there are 
canonical surjections
\begin{equation}\label{eqn:Surface-0-higher}
{\CH^{LW}_0(X)}/n \stackrel{\pi^*}{\surj} {\CH_0(\wt{X}|mE)}/n \surj 
{\CH_0(\wt{X}|E)}/n \surj {H^S_0(X^o)}/n.
\end{equation}

The first arrow from the left is an isomorphism for all $m \gg 1$
by \cite[Theorem~1.8]{Gupta-Krishna} and the third arrow is an isomorphism
by \propref{prop:Suslin-modulus}.
We thus have to show that ${\CH_0(\wt{X}|mE)}/n \surj {\CH_0(\wt{X}|E)}/n$ is an
isomorphism for all $m \ge 1$. But this follows from 
\cite[Theorem~1.3(2)]{Miyazaki}. This concludes the proof of
\thmref{thm:Main-4}.
$\hfill \square$

\vskip .3cm

\subsection{Class field theory with finite coefficients}
\label{sec:CFT-fin}
We shall now prove \thmref{thm:Main-1-0} as an application of \thmref{thm:p-torsion}. We restate it for convenience.

\begin{thm}\label{thm:Rec-prime}
Let $X$ be an integral projective $R_1$-scheme of dimension $d \ge 1$ over a 
finite field. Let $n$ be any integer prime to ${\rm char}(k)$.
Then the reciprocity map 
\[
\phi_X: {\CH^{LW}_0(X)}/n \to {\pi^{\rm ab}_1(X^o)}/n
\]
is an isomorphism of finite abelian groups.
\end{thm}
\begin{proof}
By an argument identical to the one in ~\eqref{eqn:Main-4-0}, we can assume
that $X$ is normal. We shall first show by induction on $d$ that 
$\phi_X$ is surjective. This is clear for $d \le 2$ by \thmref{thm:Main-1}.
We assume therefore that $d \ge 3$. We fix an integer $n$ prime to
${\rm char}(k)$. We fix an embedding $X \subset \P^N_k$ and apply
\cite[Theorem~6.3]{Ghosh-Krishna}
to find a hypersurface $H \subset \P^N_k$ such that the scheme theoretic
intersection $Y = X \cap H$ is normal, smooth along $X^o$ and intersects $X_\sing$
properly. We had argued in the proof of part (2) of \thmref{thm:Main-1}
that $Y$ must be integral in this case.

By \corref{cor:PF-LW}, we get a commutative diagram
\begin{equation}\label{eqn:Rec-prime-0}
\xymatrix@C.8pc{
{\CH^{LW}_0(Y)}/n \ar[r]^-{\phi_Y} \ar[d]_-{\iota_*} &
{\pi^{\rm ab}_1(Y^o)}/n \ar[d]^-{\iota_*} \\
{\CH^{LW}_0(X)}/n \ar[r]^-{\phi_X} & {\pi^{\rm ab}_1(X^o)}/n,}
\end{equation}
where $\iota \colon Y \inj X$ is the inclusion.

It follows from \propref{prop:LEF*} and \cite[Expos{\'e}~X, Corollaire~2.6]{SGA-2} that
the right vertical arrow is surjective. The top horizontal arrow is
surjective by induction. We conclude that $\phi_X$ is surjective.
To finish the proof of the theorem, it suffices now to show that
${\CH^{LW}_0(X)}/n$ and ${\pi^{\rm ab}_1(X^o)}/n$ are both finite
abelian groups of the same cardinality.

By \thmref{thm:p-torsion}, we can replace ${\CH^{LW}_0(X)}/n$ by
${H^S_0(X^o)}/n \cong H^S_0(X^o, {\Z}/n)$. Similarly, we can replace 
${\pi^{\rm ab}_1(X^o)}/n$ by $H^1_{\et}(X^o, {\Z}/n)^* := 
\Hom_{{\Z}/n}(H^1_{\et}(X^o, {\Z}/n), {\Z}/n)$.
On the other hand, \cite[Corollary~7.1]{Kelly-Saito} implies that
$H^S_0(X^o, {\Z}/n) \cong H^1_{\et}(X^o, {\Z}/n)^*$ and
\cite[Theorem~4.1]{Geisser-Schmidt} says that
$H^1_{\etl}(X^o, {\Z}/n)^*$ is finite.
It follows that ${\CH^{LW}_0(X)}/n$ and ${\pi^{\rm ab}_1(X^o)}/n$ are finite
and have the same cardinality.
\end{proof}

\vskip .3cm

\subsection{Chow group vs. Suslin homology over finite 
fields}\label{sec:LWCH-fin}
We shall now show that the assumption that $k$ is algebraically closed
in \thmref{thm:Main-4} is essential. 
Assume that $k$ is a finite field and $X \in \Sch_k$ satisfies
one of the two conditions of \thmref{thm:Main-1}. Let $\pi^{t, \ab}_1(X^o)$
be the abelianized tame fundamental group of $X^o$ (see \cite{Schmidt})
which describes the finite {\'e}tale covers of $X^o$ which are tamely 
ramified along $X_\sing$.
We then have a commutative diagram
\begin{equation}\label{eqn:Tame}
\xymatrix@C.8pc{
0 \ar[r] & \CH^{LW}_0(X) \ar[r]^-{\phi_X} \ar@{->>}[d]_-{\theta_X} & 
\pi^{\ab}_1(X^o) \ar[r] \ar@{->>}[d] & {\wh{\Z}}/{\Z} \ar@{=}[d] \ar[r] & 0 \\
0 \ar[r] & H^S_0(X^o) \ar[r]^-{\phi_{X^o}} & \pi^{t, \ab}_1(X^o) \ar[r] &
{\wh{\Z}}/{\Z} \ar[r] & 0,}
\end{equation}
whose rows are exact.
The top row is given by \thmref{thm:Main-4} and the bottom row is given
by \cite[Theorem~8.7]{Schmidt}. 
It is clear that the middle vertical arrow may not be
injective in general. This implies that $\theta_X$ is not injective in general.

\section{The Roitman torsion theorem}\label{sec:RTT}
Let $k$ be an algebraically closed field of arbitrary characteristic.
Let $X$ be a smooth quasi-projective scheme over $k$ which admits an
open embedding $X \inj \ov{X}$, where $\ov{X}$ is smooth and projective
over $k$. Then Spie{\ss} and Szamuely \cite{SS} showed that
the Albanese homomorphism $\vartheta_X$ (see ~\eqref{eqn:alb-hom}) is an 
isomorphism on prime-to-$p$ torsion subgroups, where $p$ is
the exponential characteristic of $k$. 
This was an important development as it provided a crucial 
breakthrough in eliminating the projectivity hypothesis from the
famous Roitman torsion theorem \cite{Roitman}.
Geisser \cite{Geisser} subsequently showed that the prime-to-$p$ condition in
the torsion theorem of Spie{\ss} and Szamuely could be eliminated if one
assumed resolution of singularities. 

The goal of this section is prove \thmref{thm:Main-5}
which eliminates the prime-to-$p$ condition
from the torsion theorem of \cite{SS}  without assuming 
resolution of singularities.

\enlargethispage{30pt}

\subsection{Some preliminaries}\label{sec:PreRT}
We shall use the following results in our proof.

\begin{lem}\label{lem:Bertini-special}
Let $X \subset \P^N_k$ be an integral Cohen-Macaulay closed subscheme of 
dimension $d \ge 2$ and let $C \subset \P^N_k$ be a closed subscheme such that
the scheme theoretic intersection $C \cap X$ has codimension at least two in 
$X$.
Then for all $m \gg 0$, a  general hypersurface $H \subset \P^N_k$ of degree $m$
containing $C$ has the property that $X \cap H$ is an integral scheme 
of dimension $d-1$.
\end{lem}
\begin{proof}
Since $\codim (C\cap X, X)\ge 2$, by \cite[Theorem~1]{KL}, a general 
hypersurface $H \subset \P^N_k$ of any
degree $m \gg 0$ containing $C$ has the property that 
$X \cap H$ is irreducible of dimension $d-1$ and smooth along
$X_\reg \setminus C$. In particular, it is  generically smooth. That is,
it $X \cap H$ satisfies Serre's $R_0$ condition.

Since $X$ is Cohen-Macaulay, any
hypersurface $H \subset \P^N_k$ containing $C$ has the property 
that $X \cap H$ is Cohen-Macaulay. In particular, it satisfies 
Serre's $S_1$ condition. But it is classical that
a Noetherian scheme is reduced if and only if it satisfies 
$R_0$ and $S_1$ conditions. We therefore conclude that 
a general hypersurface $H \subset \P^N_k$ 
of any degree $m \gg 0$ containing $C$ has the property that 
$X \cap H$ is reduced and irreducible, hence integral of dimension $d-1$.
\end{proof}

\begin{lem}\label{lem:Bertini-general}
Let $X \subset \P^N_k$ be a smooth and connected 
projective scheme of dimension $d \ge 3$.
Let $Z \subset X$ be a nowhere dense reduced closed subscheme with
$(d-1)$-dimensional irreducible components $Z_1, \ldots , Z_r$.
Let $C \subset X$ be a reduced curve whose no component lies in 
$Z$. Assume that the embedding dimension of $C$ at
each of its closed points is at most two.
Then for all $m \gg 0$, a  general hypersurface $H \subset \P^N_k$ of degree $m$
containing $C$ has the property that $X \cap H$ is $Z$-admissible.
\end{lem}
\begin{proof}
Since each $Z_i$ is a Cartier divisor on a smooth scheme, 
it is Cohen-Macaulay of dimension $d-1 \ge 2$. Furthermore, 
our hypothesis implies that $C \cap Z_i$ has codimension at least two in $Z_i$ 
for each $i$. Since $\edim(C\cap X^o)<3$, it follows from 
\cite[Theorem~7]{KL} that a general hypersurface $H \subset \P^N_k$ 
of any degree $m \gg 0$ containing $C$ has the property that 
$X \cap H$ is smooth. We combine this with \lemref{lem:Bertini-special}
to conclude the proof.
\end{proof}

We shall also need the following result on the invariance of the
$p$-primary torsion subgroup of the generalized Albanese variety under
monoidal transformations.

\begin{lem}\label{lem:p-tor-inv}
Assume that ${\rm char}(k) = p > 0$ and let
$U$ be a smooth quasi-projective scheme of dimension $d \ge 1$ over $k$.
Suppose that there exists an open immersion $U \subset X$ such that
$X$ is a smooth projective scheme. Let $\pi \colon \wt{X} \to X$ be
the morphism obtained by a successive blow-ups along closed points.
Then the induced homomorphism $\pi_* \colon \Alb_S(\pi^{-1}(U))(k) \to 
\Alb_S(U)(k)$ is an isomorphism on the $p$-primary torsion subgroups.
\end{lem}
\begin{proof}
We let $\wt{U} = \pi^{-1}(U)$. It follows from ~\eqref{eqn:Alb-0} that 
there is a commutative diagram of exact sequences of abelian groups
\begin{equation}\label{eqn:p-tor-inv-0}
\xymatrix@C.8pc{
0 \ar[r] & \Lambda_{\wt{U}}(\wt{X})^\vee(k) \ar[r] \ar[d]_-{\pi_*} & 
\Alb_S(\wt{U})(k) \ar[r] \ar[d]^-{\pi_*}  & \Alb_W(\wt{X})(k) \ar[r] 
\ar[d]^-{\pi_*}  & 0 \\
0 \ar[r] & \Lambda_{U}({X})^\vee(k) \ar[r] & 
\Alb_S(U)(k) \ar[r] & \Alb_W(X)(k) \ar[r] & 0.}
\end{equation}

Since $T(k)$ is divisible and 
$T(k)\{p\} = 0$ for an algebraic torus $T$ over $k$, we see that  
the maps $\Alb_S(\wt{U})(k)\{p\} \to \Alb_W(\wt{X})(k)\{p\}$
and $\Alb_S(U)(k)\{p\} \to \Alb_W(X)(k)\{p\}$ are isomorphisms.
On the other hand, one knows that the Weil-Albanese variety of a 
smooth scheme is a birational invariant. This implies that the
right vertical arrow in ~\eqref{eqn:p-tor-inv-0} is an isomorphism.
We conclude that the middle vertical arrow is an isomorphism on the
$p$-primary torsion subgroups.
\end{proof}

\subsection{Proof of \thmref{thm:Main-5}}\label{sec:Main-5-pf}
We shall now prove \thmref{thm:Main-5}. We let $U$ be a smooth 
quasi-projective scheme of dimension $d \ge 1$ over $k$
with an open immersion $U \subset X$ such that
$X$ is smooth and projective over $k$. We have to show that 
the Albanese homomorphism 
$\vartheta_U \colon H^S_0(U)_{\tor} \to \Alb_S(U)(k)_\tor$ is an isomorphism.

We can assume $X$ to be integral. We can also assume that 
${\rm char}(k) = p >0$. 
We shall prove the theorem by induction on $d$. The case $d \le 2$
follows from \cite[Theorem~1.1]{Geisser}. 
We therefore assume $d \ge 3$.

We fix a closed embedding $X \inj \P^N_k$ and let $Z = X \setminus U$
with reduced structure. Let $H \subset \P^N_k$ be a
hypersurface such that the scheme theoretic intersection $Y = X \cap H$
satisfies the condition of \lemref{lem:Bertini-general}
Using  
the covariance of the Albanese homomorphism (see the beginning of 
\S~\ref{sec:Alb-hom}), we get a commutative diagram
\begin{equation}\label{eqn:Roitman-main-0}
\xymatrix@C2pc{
H^S_0(Y \cap U)^0 \ar[r]^-{\vartheta_{Y \cap U}} \ar[d]_-{\iota_*} &
\Alb_S(Y \cap U)(k) \ar[d]^-{\iota_*} \\
H^S_0(U)^0 \ar[r]^-{\vartheta_U} & \Alb_S(U)(k),}
\end{equation}
where $\iota \colon Y \inj X$ is the inclusion.
Using \thmref{thm:Lef-Alb-main}, the known case $d \le 2$ and an 
induction on $d$, we see that $\vartheta_U$ is surjective on the torsion 
subgroups. In the rest of the proof, we shall show that this map is injective 
too.

We shall prove the injectivity in several steps. 
We fix an element $\alpha \in \sZ_0(U)$ such that $\alpha \neq 0$ 
in $H^S_0(U)$ but $n \alpha = 0$ in $H^S_0(U)$ for some integer $n \ge 2$.
By the torsion theorem of Spie{\ss} and Szamuely \cite{SS},
we can assume $n = p^a$, where $a$ is a positive integer. 
We must then have $\alpha \in \sZ_0(U)^0$. We shall show that 
$\vartheta_U(\alpha) \neq 0$. This will finish the proof.

Since $n \alpha = 0$ in $H^S_0(U)$, we can find a finite collection of 
distinct integral normal
curves $\{C_1, \ldots , C_m\}$ with finite maps $\nu_i \colon C_i \to X$ 
none of whose images is contained in $Z$ and elements 
$f_i \in \sO^{\times}_{C_i, E_i}$
such that $f_i(x) = 1$ for every $x \in E_i$ and
$n \alpha = \sum_i (\nu_i)_*(\divf(f_i))$. Here, $E_i =
\nu^{-1}_i(Z)$. We let $C'_i = \nu_i(C_i)$ and $C' = \bigcup_i C'_i \subset X$.
Since we can not always find a hypersurface section of $X$ which
is smooth along $U$ and contains $C'$, we have to go through some
monoidal transformations of $X$.

\vskip .3cm

{\bf STEP~1}:
We can find a morphism $\pi \colon \wt{X} \to X$ which is a composition
of monoidal transformations whose centers are closed points such that the 
following hold.

\begin{enumerate}
\item
The strict transform $D'_i$ of each $C'_i$ is smooth so that
$D'_i \cong C_i$.
\item
$D'_i \cap D'_j = \emptyset$ for $i \neq j$.
\item
Each $D'_i \subset \wt{X}$ intersects the exceptional divisor $E$ (which is reduced)
transversely.
\end{enumerate}

It is clear that there exists a set of distinct blown-up closed points 
$T \subset X$ such that $\pi \colon \pi^{-1}(X \setminus T) \to 
X \setminus T$ is an isomorphism. 
Set $\wt{U} = \pi^{-1}(U)$ and $\wt{Z} = \pi^{-1}(Z)$ with reduced structure.
We shall identify $D'_i$ with $C_i$ and the composite map
$C_i \xrightarrow{\cong} D'_i \xrightarrow{\pi} C'_i$ with $\nu_i$.
Let $C$ denote the strict transform of $C'$ with irreducible components
$\{C_1, \ldots , C_m\}$. We then have
$E_i = \nu^{-1}_i(Z) = \wt{Z} \cap C_i$.
Since ${\rm Supp}(\alpha) \subset C' \cap U$, we can find
$\alpha' \in \sZ_0(\wt{U})$ supported on $C$ 
such that $\pi_*(\alpha') = \alpha$.
This implies that 
$\pi_*(n\alpha' - \sum_i \divf(f_i)) = 0$.
Setting $\beta = n\alpha' - \sum_i \divf(f_i)$, we get 
$\pi_*(\beta) = 0$ in the cycle group $\sZ_0(U)$.

\vskip .2cm

{\bf STEP~2}:
We let $T' = T \cap U = \{y_1, \ldots , y_s\}$.
We can then write $\beta = \stackrel{s}{\underset{i =0}\sum} \beta_i$,
where $\beta_i$ is a 0-cycle on $\wt{U}$ supported on $\pi^{-1}(y_i)$
for $1 \le i \le s$ and $\beta_0$ is supported on $\wt{U} \setminus E$. 
We then get 
$\stackrel{s}{\underset{i =0}\sum}  \pi_*(\beta_i) = 0$ in 
$\sZ_0(U) \subseteq \sZ_0(X)$.  
Since all closed points of $T$ are distinct and the support of $\pi_*(\beta_0)$
is disjoint from $T'$, and hence from $T$, one easily checks that we must have 
$\pi_*(\beta_i) = 0$ for all $0 \le i \le s$. Since $\pi$ is an isomorphism 
away from $T$, we must have $\beta_0 = 0$.
We can therefore assume that $\beta$ is a 0-cycle on $E \cap \wt{U}$.

We now note that each $\pi^{-1}(\{y_i\})$ is a $(d-1)$-dimensional projective 
scheme whose irreducible components are successive point blow-ups of 
$\P^{d-1}_k$. Moreover, we have 
$\pi_*(\beta_i) = 0$ under the push-forward map 
$\pi_*\colon \sZ_0(\pi^{-1}(\{y_i\})) \to \Z$, induced by the maps
$\pi\colon \pi^{-1}(\{y_i\})
\to \Spec(k(y_i)) \xrightarrow{\simeq} \Spec(k)$. 
But this means that ${\rm deg}(\beta_i) = 0$. Taking the sum, we get
${\rm deg}(\beta) = \stackrel{s}{\underset{i =1}\sum} {\rm deg}(\beta_i) = 0$.
We can therefore find finitely many smooth projective rational curves 
$L_1, \ldots , L_{m'}$ on $E \cap \wt{U}$ and rational functions 
$f'_j \in k(L_j)$ 
such that $\beta = \stackrel{m'}{\underset{j = 1}\sum}  \divf(f'_j)_{L_j}$
(see \cite[Lemma~6.3]{Krishna-Inv}).

\vskip .2cm

{\bf STEP~3}: Using an argument of Bloch (see \cite[Lemma~5.2]{Bloch-Duke}), 
after possibly further
blow-up of $\wt{X}$ along the closed points of $E \cap \wt{U}$, we can assume
that no more than two $L_j$'s meet at a point and they intersect
$C$ transversely (note that $C$ is smooth along $E$).
In particular, in combination with (1) - (3) above, this implies that
$D:= C \cup (\cup_j L_j)$ is a reduced curve with
following properties (see line 4 from the bottom of \cite[p.~5.2]{Bloch-Duke}).

\begin{listabc}
\item
Each component of $D$ is smooth (note that $D = C$ away
from $(\cup_j L_j)$).
\item
$D$ is smooth along $\wt{X} \setminus \wt{U}$. 
\item
$D \cap \wt{U}$ has only ordinary double point singularities, i.e., 
exactly two components of $D \cap \wt{U}$ meet at any of its 
singular points with distinct tangent directions. 
\end{listabc}

In particular, the embedding dimension of $D$ at each of its closed 
points is at most two. Furthermore, we have
\begin{equation}\label{eqn:Roitman-main-1}
n \alpha'= \stackrel{m}{\underset{i = 1}\sum}  \divf(f_i) + 
\beta = \stackrel{m}{\underset{i = 1}\sum} \divf(f_i) + 
\stackrel{m'}{\underset{j = 1}\sum}  \divf(f'_j).
\end{equation}
Since $L_j \cap \wt{Z} = \emptyset$ for each $j$, it follows that
$n \alpha' \in \sR^S_0(\wt{U})$. Note also that $\wt{X}$ is an integral
smooth projective scheme.

\vskip .2cm

\enlargethispage{20pt}

{\bf STEP~4}:  Let $\{\wt{Z}_1, \ldots , \wt{Z}_r\}$ be the set of irreducible 
components of $\wt{Z}$ of dimension $d-1$ 
with integral closed subscheme structure on each 
$\wt{Z_i}$. We fix a closed embedding $\wt{X} \inj \P^M_k$. 
It follows from \lemref{lem:Bertini-general} that all $q \gg 0$, a general 
hypersurface $H \subset \P^M_k$ of degree $q$
containing $D$ has the property that the scheme theoretic intersection
$Y = X \cap H$ is $\wt{Z}$-admissible. Since $q \gg 0$, we can also ensure
using the Enriques-Severi-Zariski vanishing theorem that
$H^0(\wt{X}, \Omega^1_{{\wt{X}}/k}(-Y)) = 0$. 
We choose such a hypersurface $H$ and let 
$\iota \colon Y \inj \wt{X}$ denote the inclusion.
We let $V = Y \cap \wt{U}$.

\vskip .2cm

{\bf STEP~5}: It follows from ~\eqref{eqn:Roitman-main-1} and STEP~4 that 
$\alpha' \in \sZ_0(V)$ and $n \alpha' \in \sR^S_0(V)$, i.e., 
$n \alpha' = 0$ in $H^S_0(V)$. Note that $\alpha' \neq 0$ in 
$H^S_0(V)^0$ since $\pi_*(\alpha') = \alpha$ is not zero in
$H^S_0({U})^0_\tor$. 
Since the Albanese homomorphism is a natural transformation between two
functors on $\Sm_k$ 
(see ~\eqref{eqn:alb-hom}), there is a commutative diagram
\begin{equation}\label{eqn:Roitman-main-2}
\xymatrix@C.8pc{
H^S_0(V)^0_\tor \ar[r]^-{\vartheta_{V}} \ar[d]_-{\iota_*} &
\Alb_S(V)(k)_\tor \ar[d]^-{\iota_*} \\
H^S_0(\wt{U})^0_\tor \ar[r]^-{\vartheta_{\wt{U}}} & \Alb_S(\wt{U})(k)_\tor.}
\end{equation}

By the choice of $H$ and \thmref{thm:Lef-Alb-main}, the right vertical
arrow is an isomorphism. Since $\alpha' \in H^S_0(V)^0_\tor$, it follows
by induction on $d$ that $\vartheta_{V}(\alpha') \neq 0$. Hence, we get
\begin{equation}\label{eqn:Roitman-main-3}
\vartheta_{\wt{U}}(\alpha') = \vartheta_{\wt{U}} \circ \iota_*(\alpha') 
= \iota_* \circ \vartheta_{V}(\alpha') \neq 0.
\end{equation}

We now consider another commutative diagram
\begin{equation}\label{eqn:Roitman-main-4}
\xymatrix@C.8pc{
H^S_0(\wt{U})^0\{p\} \ar[r]^-{\vartheta_{\wt{U}}} \ar[d]_-{\pi_*} &
\Alb_S(\wt{U})(k)\{p\} \ar[d]^-{\pi_*} \\
H^S_0(U)^0\{p\} \ar[r]^-{\vartheta_{U}} & \Alb_S({U})(k)\{p\}.}
\end{equation}
Using this diagram, we get
\[
\vartheta_{U}(\alpha) = \vartheta_{U} \circ \pi_*(\alpha') =
\pi_* \circ \vartheta_{\wt{U}}(\alpha').
\]
Since $\vartheta_{\wt{U}}(\alpha') \neq 0$ by ~\eqref{eqn:Roitman-main-3}, 
we conclude from \lemref{lem:p-tor-inv} that $\vartheta_{U}(\alpha) \neq 0$.
This concludes the proof of \thmref{thm:Main-5}.
$\hfill \square$

\vskip .4cm

\noindent\emph{Acknowledgement.}
The authors are indebted to the anonymous referee for reading the
manuscript very thoroughly and suggesting many improvements.

\end{document}